\documentclass[11pt, letterpaper]{amsart}
\usepackage[left=1in,right=1in,bottom=0.5in,top=0.8in]{geometry}
\usepackage{amsfonts}
\usepackage{amsmath, amssymb}
\usepackage{graphicx}
\usepackage[font=small,labelfont=bf]{caption}
\usepackage{epstopdf}
\usepackage[pdfpagelabels,hyperindex]{hyperref}
\usepackage{xcolor}
\usepackage{amsthm}
\usepackage{float}
\usepackage{pgfplots}
\usepackage{listings}
\usepackage{longtable}
\usepackage{mathrsfs}
\usepackage{bbold}
\usepackage{comment}
\usepackage{esint}

\usepackage{mathtools}
\usepackage{scalerel}[2014/03/10]
\usepackage[usestackEOL]{stackengine}
\DeclarePairedDelimiterX\set[1]\lbrace\rbrace{#1}

\hypersetup{
pdftitle={An abstract theory for parabolic equations and applications},
pdfsubject={Mathematics, PDE, Analysis},
pdfauthor={Pascal Auscher and Khalid Baadi},
pdfkeywords={Parabolic equations, Cauchy problems, variational methods, Fundamental solution, Energy equality.}
}

\newtheorem{thm}{Theorem}[section]
\newtheorem{cor}[thm]{Corollary}
\newtheorem{prop}[thm]{Proposition}
\newtheorem{lem}[thm]{Lemma}

\theoremstyle{definition}
\newtheorem{defn}[thm]{Definition}

\theoremstyle{remark}
\newtheorem{rem}[thm]{Remark}
\newtheorem{rems}[thm]{Remarks}

\definecolor{energy}{RGB}{114,0,172}
\definecolor{freq}{RGB}{45,177,93}
\definecolor{spin}{RGB}{251,0,29}
\definecolor{signal}{RGB}{203,23,206}
\definecolor{circle}{RGB}{217,86,16}
\definecolor{average}{RGB}{203,23,206}

\definecolor{kb}{rgb}   {.6, 0, 0} 
\definecolor{pa}{rgb}   {.4, .4, 0}

\newcommand{\llangle}{\langle\!\langle}
\newcommand{\rrangle}{\rangle\!\rangle}

\colorlet{shadecolor}{gray!20}
\pgfplotsset{compat=1.9}

\usepgflibrary{fpu}
\makeatletter
\let\c@equation\c@thm
\makeatother
\numberwithin{equation}{section}

\author{Pascal Auscher}
\address{Universit{\'e} Paris-Saclay, CNRS, Laboratoire de Math\'{e}matiques d'Orsay, 91405 Orsay, France}
\email{pascal.auscher@universite-paris-saclay.fr}
\author{Khalid Baadi}
\address{Universit{\'e} Paris-Saclay, CNRS, Laboratoire de Math\'{e}matiques d'Orsay, 91405 Orsay, France}
\email{khalid.baadi@universite-paris-saclay.fr}

\date{June 23, 2025}

\title[Fundamental solutions for parabolic systems]{Fundamental solutions for parabolic equations and systems:
universal existence, uniqueness, representation}


\begin{document}

\begin{abstract}
In this paper, we develop a universal, conceptually simple and systematic method to prove well-posedness to Cauchy problems for weak solutions of parabolic equations with non-smooth, time-dependent, elliptic part having a variational definition. Our classes of weak solutions are taken with minimal assumptions.  We prove the existence and uniqueness of a fundamental solution which seems new in this generality: it is shown to always coincide with the associated evolution family for the initial value problem with zero source and it yields representation of all weak solutions. Our strategy is a variational approach avoiding density arguments, a priori regularity of weak solutions or regularization by smooth operators. One of our main tools are  embedding results which yield time continuity of our weak solutions going beyond the celebrated Lions regularity theorem and that is addressing a variety of source terms.   We illustrate our results with  three concrete applications : second order uniformly elliptic part with Dirichlet boundary condition on domains, integro-differential elliptic part, and second order degenerate elliptic part. 
\end{abstract}
\maketitle

\tableofcontents

\section{Introduction}\label{Section 1}

Linear parabolic problems have a long history. The standard method usually begins by solving the Cauchy problem, with or without source terms, and representing the solutions through what is called fundamental solution. Abstractly, starting from an initial data $u_{0}$, the evolution takes the form $ \partial_t u + \mathcal{B}u =f$
where time describes an interval $(0,T)$, the sought solution $u$ and the source $f$ are valued in some different  vector spaces and $\mathcal{B}$ is an operator  that could also depend on time. Modeled after the theory of  ordinary scalar differential equations, it is assumed some kind of dissipativity for $-\mathcal{B}$. One issue toward extension to non-linear equations is to not use regularity, but measurability, of the coefficients. We stick here to the linear ones. 
The amount of literature is too vast to be mentioned and we shall isolate only a few representative works for the sake of motivation. 

In the abstract case, when $\mathcal{B}$ does not depend on time, that is the autonomous situation, the approach from semi-group of operators and interpolation has been fruitful, starting from the earlier works (see Kato's book \cite{kato2013perturbation}) up to the criterion for maximal regularity in UMD Banach spaces  (see L. Weis \cite{weis2001operator} or the review by Kunstmann-Weis \cite{kunstmann2004maximal}).  In the non-autonomous case, results can be obtained by perturbation of this theory, assuming some time-regularity; see, for example, Kato's paper \cite{kato1961abstract}.

A specific class of such problems is when $\mathcal{B}$  originates from a sesquilinear form with coercivity assumptions. Lions developed an  approach in \cite{lions1957problemes} which he  systematized in \cite{lions2013equations}. The nice thing about this approach is that there is no need for regularity with respect to time; its drawback is that it is restricted to  Hilbert initial spaces. 

In parallel, there has been a systematic study of    
concrete parabolic Cauchy problems of differential type starting  when the coefficients are regular. In this case, several methods exist for constructing the fundamental solution. The most effective technique involves a parametrix in combination with the freezing point method \cite{friedman2008partial}. This approach simplifies the problem to one where the coefficients become independent of space, leading to explicit solutions represented by kernels $ \Gamma(t, x, s, y) $ with Gaussian decay. When the coefficients are measurable (and possibly unbounded for the lower order terms), the theory of weak solutions, developed in the 1950s and 1960s, applies. This theory culminated in the book of Ladyzenskaja, Solonnikov and Ural'ceva \cite{ladyzhenskaia1968linear}. Although we shall not consider it here, the specific situation of second order parabolic equations with real, measurable, time-dependent coefficients was systematically treated by Aronson \cite{ aronson1967bounds, aronson1968non}: his construction of fundamental solutions and the proof of  lower and upper bounds relied on regularity properties of local weak solutions by Nash \cite{nash1958continuity} and its extensions, and by taking limits from operators with regular coefficients.  A recent result \cite{ataei2024fundamental} follows this approach for non autonomous degenerate parabolic problems in the sense of $A_{2}$-weights. In contrast, not using regularity theory, the article \cite{auscher2023universal} developed a framework for a Laplacian (and its integral powers), extending the Lions embedding theorem (see below) as a first step to obtain new results on fundamental solutions for equations with unbounded coefficients.

A natural question is whether one can develop a framework, going beyond the one of Lions, that provides us with  \begin{enumerate}
\item An  optimal embedding theorem with integral identities,
\item The largest classes of weak solutions for which one obtains existence and uniqueness,
\item Definition, existence and uniqueness of a fundamental solution.
\end{enumerate}
The answer is yes. The arguments can be developed in a very abstract manner, bearing on functional calculus of positive self-adjoint operators. We next present a summary of our results as a road map. Clearly, when it comes to concrete applications,  our results do not distinguish equations from systems, the nature of the elliptic part (local or non-local) and its order, boundary conditions, etc.  
We give three examples at the end as an illustration. 

The first example deals with second-order uniformly elliptic parts under Dirichlet boundary conditions on domains. This situation may not seem original but we still give the main consequences for readers to be able to compare with literature. Our second example is parabolic equations with integro-differential elliptic part. The typical example is the fractional Laplacian and such equations arise in many fields from PDE to probability \cite{lions1969quelques, caffarelli2007extension, bogdan2012estimates}. The usual theory for  the  fractional Laplacians  yields the fundamental solution  as a density of a probability measure. The fundamental solutions for general integro-differential  parabolic operators with kernels are considered  in the literature, assuming positivity condition and pointwise bounds. We refer to the introduction of \cite{kassmann2023upper} and the references there in. In that article, a proof of the poinswise upper bound is presented. Here, we do not assume any kind of positivity and we show the existence of a fundamental solution as an evolution family of operators. At our level of generality, these operators may not have kernels with pointwise bounds. Still, this family can be used to represent weak solutions without further assumptions. In any case,  it gives a universal existence result. For example, the fundamental solution  used in \cite{kassmann2023upper} must be the kernel of our fundamental solution operator. The third example is for degenerate operators as in \cite{ataei2024fundamental}, without assuming the coefficients to be real. It directly gives  existence of a fundamental solution not using local properties of solutions. In a forthcoming article, the second author will use this to give a new proof of the pointwise estimates.

\section{Summary of our results}\label{sec:summary}

\subsection{Embeddings and integral identities}
 
Lions' embedding theorem \cite{lions1957problemes} asserts that  if $V$ and $H$ are two Hilbert spaces such that $V$ is  densely embedded in $H$, itself densely embedded in  $V^\star$, the dual of $V$ in the inner product of  $H$, we have the continuous embedding
$$ L^2((0, \mathfrak{T}); V) \cap H^1((0, \mathfrak{T}); V^\star) \hookrightarrow C([0, \mathfrak{T}]; H),$$
and absolute continuity of the map $t\mapsto\|u(t)\|_{H}^2$, which yields integral identities. The triple $(V, H, V^\star)$ is said to be a Gelfand triple.
One of our main results is the following improvement.

\begin{thm}\label{thm:energyboundedinterval} Consider a positive, self-adjoint operator $S$ on a separable Hilbert space $H$.
    Let $I=(0,\mathfrak{T})$ be a  bounded, open interval of $\mathbb{R}$. Let $u \in L^1(I;H)$ such that $Su\in L^2(I;H)$. Assume that $\partial_t u =  Sf+S^\beta g$ with $f \in L^2(I;H)$ and 
     $g \in L^{\rho'}(I;H)$, where ${\beta}={2}/{\rho} \in [0,1)$ and $\rho'$ is the conjugate H\"older exponent to $\rho$. Then $ u \in C(\Bar{I},H)$
and $ t \mapsto \left \| u(t) \right \|^2_H$ is absolutely continuous on $\Bar{I}$ with, for all $ \sigma, \tau \in \Bar{I}$ such that $  \sigma < \tau$, the integral identity
    \begin{align}\label{eq:integralidentityintro}
        \left \| u(\tau) \right \|^2_H-\left \| u(\sigma) \right \|^2_H = 2\mathrm{Re}\int_{\sigma}^{\tau} \langle f(t),Su(t)\rangle_{H} + \langle g(t),S^\beta u(t)\rangle_{H} \  \mathrm d t.
    \end{align}
As a consequence,     $u\in L^r((0,\mathfrak{T});D(S^\alpha))$ for all $r\in (2,\infty]$ such that $\alpha=2/r$ with for a constant depending only on $\beta$, 
\begin{align}\label{eq:Lr}
 \| S^\alpha u \|_{L^r((0,\mathfrak{T});H)} \lesssim  \|Su \|_{L^2((0,\mathfrak{T});H)}+ \|f \|_{L^2((0,\mathfrak{T});H)}+\|g \|_{L^{\rho'}((0,\mathfrak{T});H)} + \inf_{\tau \in [0,\mathfrak{T}]}\|u(\mathfrak{\tau})\|_{H}.
\end{align}
      \end{thm}

Let us comment this result. Its core is the continuity and the integral identity proved in  Corollary \ref{corenergybounded} when $S$ is injective and  Proposition \ref{prop:energyboundedintervalinhomo} in the general case. If $u$ had belonged to $L^2(I;H)$ then $u\in L^2(I;V)$ with $V$ the domain of $S$. Here we only assume $u\in L^1(I;V)$, which is used qualitatively,  and $\partial_{t}u$ is  taken in the sense of distributions on $I$ valued in $V^*$.  
We allow an extra second term $S^\beta g$ in  the time derivative expression of $u$. In fact, $S^\beta g$ belongs to $SL^2(I;H)+ L^1(I;H)\subset L^1(I; V^*)$ (see Proposition \ref{prop:embed r>0} when $S$ is injective and the proof is the same otherwise).  Note that it could be a finite combination of such terms and the integral identity would  be modified accordingly. 
The $L^r$ estimate \eqref{eq:Lr} follows  \textit{a posteriori}: one first proves \eqref{eq:Lr} with  
$r=\infty$ using
   \eqref{eq:integralidentityintro},  together with  the  interpolation inequality 
$$
 \| S^\alpha u \|_{L^r((0,\mathfrak{T});H)} \le \|Su \|_{L^2((0,\mathfrak{T});H)}^\alpha \|u \|_{L^\infty((0,\mathfrak{T});H)}^{1-\alpha}  
 $$ 
 for $\alpha=\beta$ when $0<\beta<1$. Then we reuse this inequality for different $\alpha\in[0,1)$. See  Proposition \ref{prop:embed r>0} when $S$ is injective together with the remark that follows it and the proof applies \textit{verbatim} when $S$ is not injective.

Theorem \ref{thm:energyboundedinterval} admits versions on unbounded intervals.  On the half-line $(0, \infty)$ (resp.~on $\mathbb{R}$), it holds  assuming  that $u \in L^1((0, \mathfrak{T}); H)$ for all $\mathfrak{T} < \infty$ (resp.~$u \in L^1_{\mathrm{loc}}(\mathbb{R}; H)$) when $S$ is not injective. In this case, $u$ is bounded in $H$ on $(0, \infty)$ (resp. $\mathbb{R}$), {see Proposition \ref{prop:energyboundedintervalinhomo}}. 
When $S$ is injective, the local integrability condition of $u$ in $H$ can be dropped using completions of $V$ and $V^*$ for the homogeneous norms $\|Su\|_{H}$ and $\|S^{-1}u\|_{H}$.  In this case, not only $u$ is bounded, but it also tends to zero at infinity (resp.~at $\pm \infty$) in $H$, as shown in Proposition \ref{energy lemma} and Corollary \ref{corenergy}. As a consequence, we can eliminate the last term  in \eqref{eq:Lr}. 

Our strategy is to prove the embedding first  when $I = \mathbb{R}$, proceed by restriction to $(0,\infty)$.  
When it comes to bounded intervals, although the condition $u\in L^1(I;H)$ does not appear in \eqref{eq:Lr},  the condition  $Su\in L^2(I;H)$ alone does not suffice as an example shows. We added strong integrability  in $H$ but in fact, as the proof shows, it suffices that $u$ exists as a distribution on $I$ valued in $V^\star$ and there exists one $t$ for which $u(t)\in H$. But for applications to Cauchy problems, it is more natural to have the integrability condition and we stick to that. 

 Our proof of the embedding already has a  PDE flavor using 
$\partial_t u+S^2u=S^2u+Sf+S^\beta g$. This leads us to a thorough study of the abstract heat operator $\partial_t +S^2$ (hence the notation $\partial_tu$  rather than $u'$) which has some interest on its own right. This is done in Sections \ref{Section 2}, \ref{Section 3}, \ref{Section 4}.

\subsection{Weak solutions and Cauchy problems}
The embedding and its variants allow us to consider the largest possible class of weak solutions to abstract parabolic operators $\partial_{t}+\mathcal{B}$ with 
a time-dependent elliptic part $\mathcal{B}$ associated to a family of bounded and sesquilinear forms on the domain of  $S$. We do not assume any time-regularity on $\mathcal{B}$ apart its weak measurability. This can be done either with estimates being homogeneous  in $S$ if we decide to work on infinite intervals,   
(see Section \ref{sec:homogenoussetup}) or  inhomogeneous  (See Section \ref{sec:inhomogenousCP};  we called the elliptic operator $\Tilde{\mathcal{B}}$ there).

Let us state the final result in the latter case, that is with inhomogeneous $\Tilde{\mathcal{B}}$ as in Section \ref{sec:inhomogenousCP}, combining Theorems \ref{ThmCauchy inhomog} and \ref{thm: passage au concret} on a finite interval $(0,\mathfrak{T})$.   We fix $\rho \in (2,\infty)$ and set $\beta={2}/{\rho}$. Given an initial condition $a\in H$ and  source terms $f  \in L^2((0,\mathfrak{T});H)$ and $g \in L^{\rho'}((0,\mathfrak{T});H)$, $h\in L^{1}((0,\mathfrak{T});H)$ we wish to solve  the Cauchy problem 
\begin{align}\label{eq:Cauchy inhomogeneintro}
\left\{
    \begin{array}{ll}
        \partial_t u +\Tilde{\mathcal{B}}u =  S{f}+  S^\beta {g} + h  \ \ \mathrm{in} \ \mathcal{D}'((0,\mathfrak{T}); \mathrm{D}), \\
        u(0)=a \ \ \mathrm{weakly \ in} \ \mathrm{D},
    \end{array}
\right.
\end{align}
where $ \mathrm{D}$ is a Hausdorff topological dense subspace of the domain of $S$, equipped with the graph norm, that is, a core of $D(S)$. The first equation is thus understood in the weak sense against test functions in $\mathcal{D}((0,\mathfrak{T}); \mathrm{D})$. The meaning of the second equation is by taking the limit  $\langle u(t), \Tilde{a}\rangle_{H}\rightarrow \langle a, \Tilde{a}\rangle_{H}$ for all $\Tilde{a}\in \mathrm{D}$ along a sequence converging to 0.

\begin{thm}\label{Thm:Cauchyinhomogintro}  There exists a unique weak solution $u\in L^1((0,\mathfrak{T}); H)$ with $Su\in L^2((0,\mathfrak{T}); H)$ to the problem \eqref{eq:Cauchy inhomogeneintro}. Moreover,  
      $u \in C([0,\mathfrak{T}];H) \cap L^r((0,\mathfrak{T});D(S^\alpha))$ for all $r\in [2,\infty)$ with $\alpha=2/r$,  and we have the estimate
        \begin{align*}
            \sup_{t\in [0,\mathfrak{T}]} \| u(t) \|_{H}+\|  S^\alpha u  \|_{L^r((0,\mathfrak{T});H)}
            \leq C  ( \left \| f \right \|_{L^2((0,\mathfrak{T});H)} + \left \| g \right \|_{L^{\rho'}((0,\mathfrak{T});H)}+ \left \| h \right \|_{L^{1}((0,\mathfrak{T});H)}+ \left \| a \right \|_H  ),
        \end{align*} 
        where $C$ is a constant independent of $f,g,h$ and $a$. In addition, we can write the energy equality corresponding to the absolute continuity of $t\mapsto \|u(t)\|^2_{H}$.
\end{thm}

Theorem \ref{ThmCauchy inhomog} also contains a variant  on the interval $(0,\infty)$ (case (b) there) where we replace $S$ by the operator $\Tilde{S}=(S^2+1)^{1/2}$ and we also obtain decay of the solution at $\infty$ while the class of uniqueness is $u\in L^1((0,\mathfrak{T}); H)$ for all $\mathfrak{T}<\infty$ with $Su\in L^2((0,\infty); H)$.

We note that we consider classes of solutions in $L^1((0,\mathfrak{T}); H)$ rather than $L^\infty((0,\mathfrak{T}); H)$  as is customary. This condition suffices to obtain a priori continuity in time by Theorem \ref{thm:energyboundedinterval}. There is a similar theorem for the backward parabolic adjoint operator $ -\partial_s  +\Tilde{\mathcal{B}}^*$ with final condition $\Tilde{a}$ at $\mathfrak{T}$. 

The estimates and the energy equality are a consequence of Theorem \ref{thm:energyboundedinterval}. Uniqueness relies on the energy equality. Existence is obtained by restriction from constructions first on $\mathbb{R}$ and then on $(0,\infty)$.

The role of the core $\mathrm{D}$ of $D(S)$ is in fact irrelevant here and is only for the purpose of having a weak formulation with a small space of test functions. It can be equivalently replaced by $D(S)$ itself (see Theorem \ref{thm: passage au concret} and its proof). However, for concrete partial differential equations where $\mathrm{D}$ can be taken as  a space of smooth and compactly supported functions of the variable $x$,  we can work with smooth and compactly supported functions of the variables $(t,x)$.

Homogeneous variants on $(0,\infty)$ and $\mathbb{R}$ will be found in the text (see Section \ref{Section 5}),  where we had to develop an appropriate theoretical functional framework in Sections \ref{Section 2}, \ref{Section 3}, \ref{Section 4}. Actually, we start the proof by implementing a result of Kaplan (see Lemma \ref{lemme Hidden Coerc})  proving the invertibility of a parabolic operator on a sort of variational space involving the half-time derivative. This result has been central to other developments in the field  of parabolic problems recently (see \textit{e.g.}, \cite{nystrom2017l2, auscher20202, auscher2023universal}), while it was more like a consequence of the construction of weak solutions in earlier works of the literature, including Kaplan's work \cite{kaplan1966abstract}. As this result can only be formulated when time describes $\mathbb{R}$,  this explains why we proceed by restriction from this case.

\subsection{Fundamental solution} 

We come to the notion of fundamental solution and evolution family (or propagators, or Green operators, as suggested by Lions) to represent weak solutions. Although it seems well known that they are the same, we feel it is  essential to clarify the two different definitions. This distinction eventually leads to easy arguments, even in this very general context.  We assume that $\partial_t + \mathcal{B}$ is a parabolic operator as above for which one can prove existence and uniqueness of weak solutions on $I$ with test functions valued in the core $\mathrm{D}$
 of the Cauchy problems with the absolute continuity of $t\mapsto \|u(t)\|_{H}^2$ as in Theorem \ref{Thm:Cauchyinhomogintro}, and similarly for its backward adjoint.

\begin{defn}[Fundamental solution for $\partial_t + \mathcal{B}$ on $I$] \label{FSintro}
    A fundamental solution for $\partial_t + \mathcal{B}$ is a family $\Gamma=(\Gamma(t,s))_{t,s \in I}$ of bounded operators on $H$ such that : 
    \begin{enumerate}
        \item (Uniform boundedness on $H$) $\sup_{t,s \in I} \left\|\Gamma(t,s) \right\|_{\mathcal{L}(H)} < +\infty.$
        \item (Causality) $\Gamma(t,s)=0$ if $s>t$.
        \item (Measurability) For all $a,\Tilde{a}\in  \mathrm{D}$, the function $(t,s) \mapsto \langle \Gamma(t,s)a, \Tilde{a} \rangle_H$ is Borel measurable on $I^2$.
        \item (Representation) For all $\phi \in \mathcal{D}(I)$ and $a \in  \mathrm{D}$, the weak solution of the equation $\partial_t u + \mathcal{B}u = \phi \otimes a $ in $\mathcal{D}'(I; \mathrm{D})$ satisfies for  all $\Tilde{a} \in  \mathrm{D}$,  $\langle u(t), \Tilde{a} \rangle_H = \int_{-\infty}^{t} \phi(s) \langle \Gamma (t,s)a, \Tilde{a} \rangle_H \ \mathrm{d}s,$ for almost every $t\in I$.
    \end{enumerate}
    One defines a fundamental solution $\Tilde{\Gamma}=(\Tilde{\Gamma}(s,t))_{s,t \in I}$ to the backward operator $-\partial_s + \mathcal{B}^\star$ analogously and (2) is replaced by $\Tilde{\Gamma}(s,t)=0$ if $s>t$. \end{defn}

Such an object must be unique (see Lemma \ref{Unicité sol fonda}  in the case where $I=\mathbb{R}$, whose proof applies \textit{verbatim}).

\begin{defn}[Green operators]\label{def:Greenintro}
Let $t,s \in \overline{I}$ and $a,\Tilde{a} \in H$.
\begin{enumerate}
    \item For $t \ge s$, $G(t,s)a$ is defined as the value at time $t$ of the weak solution to the equation $\partial_t u + \mathcal{B}u = 0$ with initial data  $a$  at time $s$.
    \item For $s \le  t$, $\Tilde{G}(s,t)\Tilde{a}$ is defined as the value at time $s$ of the weak solution $\Tilde{u} $ of the equation $-\partial_s \Tilde{u} + \mathcal{B^\star}\Tilde{u} = 0$ with final data $ \Tilde{a}$ at time $t$. 
\end{enumerate}
We set $G(t,s)=0=\Tilde{G}(s,t)$ if $s>t$. 
The operators $G(t,s)$ and $\Tilde{G}(s,t)$ are called the Green operators for the parabolic operator $\partial_t +\mathcal{B}$ and the backward parabolic operator $ -\partial_s +\mathcal{B^\star}$, respectively.
\end{defn}

Uniqueness and the integral identities allow to obtain the identification of the two objects as follows, with proof being \textit{verbatim} the ones of Proposition \ref{Prop Green}  and Theorem \ref{thm:identification}.

\begin{thm}\label{thm:Green-FSintro} The following statements hold.
\begin{enumerate}
    \item (Adjoint relation) For all $s <t $, $G(t,s)$ and $\Tilde{G}(s,t)$ are  adjoint operators.
    \item  (Chapman-Kolmogorov identity) For any $s < r <t $, we have $G(t,s)=G(t,r)G(r,s)$.
 \item (Existence) The family of Green operators $(G(t,s))_{s,t\in \overline{I}}$ is a fundamental solution.
    \end{enumerate}
 \end{thm}

 With this in hand, we may first combine the estimates obtained for both families. See Corollary \ref{Cor Green} for the ones for the Green operators (again, transposed \textit{verbatim}). Next, we obtain full representation for the weak solutions in Theorem \ref{Thm:Cauchyinhomogintro} (again,  combining Theorems \ref{ThmCauchy inhomog} and \ref{thm: passage au concret}).

\begin{thm}\label{Thm:FSCauchyintro} Consider the fundamental solution $(\Gamma_{\Tilde{\mathcal{B}}}(t,s))_{0\le s\le t\le \mathfrak{T}}$ of $\partial_t  +\Tilde{\mathcal{B}}$ as in Theorem \ref{Thm:Cauchyinhomogintro}.
For all $t \in [0,\mathfrak{T}]$, we have the following representation of the weak solution $u$ of  \eqref{eq:Cauchy inhomogeneintro} : 
        \begin{align*}
            u(t)=\Gamma_{\Tilde{\mathcal{B}}}(t,0)a&+\int_{0}^{t}\Gamma_{\Tilde{\mathcal{B}}}(t,s)Sf(s)\mathrm ds 
+\int_{0}^{t}\Gamma_{\Tilde{\mathcal{B}}}(t,s)S^\beta g(s) \ \mathrm ds+\int_{0}^{t}\Gamma_{\Tilde{\mathcal{B}}}(t,s)h(s) \ \mathrm ds, 
        \end{align*}
         where the two integrals containing ${f}$ and ${g}$ are weakly defined in $H$, while the one involving $h$ converges strongly (i.e., in the Bochner sense). More precisely, for all $\Tilde{a} \in H$ and $t \in [0,\mathfrak{T}]$, we have the equality with absolutely converging integrals
        \begin{align*}
            \langle u(t) , \Tilde{a}\rangle_H &
            = \langle \Gamma_{\Tilde{\mathcal{B}}}(t,0)a , \Tilde{a}\rangle_H +  \int_{0}^{t} \langle {f}(s) ,  S \Tilde{\Gamma}_{\Tilde{\mathcal{B}}}(s,t)\Tilde{a}\rangle_H \ \mathrm ds  
       \\
       &
           \qquad +\int_{0}^{t} \langle {g}(s) ,  S^\beta \Tilde{\Gamma}_{\Tilde{\mathcal{B}}}(s,t)\Tilde{a}\rangle_H \ \mathrm ds+\int_{0}^{t} \langle \Gamma_{\Tilde{\mathcal{B}}}(t,s) h(s) , \Tilde{a}\rangle_H \  \mathrm ds.
        \end{align*}
\end{thm}

As before, this is obtained from the variants on $(0,\infty)$ and $\mathbb{R}$ which in fact come first.  We refer the reader to the text for details.  

\subsubsection*{\textbf{Notation}}

By convention, and throughout this paper, the notation $C(a,b,\ldots)$ denotes a constant depending only on the parameters $(a,b,\ldots)$. 

\subsubsection*{\textbf{Acknowledgements}}
The authors want to thank Moritz Egert for taking the time to look at  a first version of this article and making valuable suggestions.

\section{The abstract homogeneous framework}\label{Section 2}
Throughout this article, we are working in a separable complex Hilbert space $H$ whose norm is denoted by $\left \| \cdot \right \|_H$ and its inner product by $\langle \cdot, \cdot \rangle_H$, and $S$ is a positive and self-adjoint operator on  $H$. \textbf{From Sections \ref{Section 2} to \ref{Section 4}, we assume that $S$ is  injective and shall not repeat this in statements.} The general case when $S$ might not be injective will be considered in Section~\ref{sec:Inhomogenous}. We do not assume that $0 \in \rho(S)$,  that is,  $S$ is not necessarily invertible. The spectrum of $S$ is contained in $\mathbb{R}_+=[0,\infty)$. To make our approach accessible, it is useful to present facts from functional calculus and give the construction of spaces of test functions and distributions in an abstract context given that 0 might be in the spectrum of $S$.
\subsection{A review of the Borel functional calculus}\label{Section: calcul bor}
For general background on self-adjoint operators and the spectral theorem, we refer to \cite{reed1980methods} and \cite{davies1995spectral}.

By the spectral theorem for self-adjoint operators, there is a unique application $f\mapsto f(S)$ from the space of all locally bounded Borel functions on $(0,\infty)$ that we denote $\mathcal{L}^\infty_{\mathrm{loc}}((0,\infty))$ into the space of closed linear maps on $H$, which sends $1$ to the identity, $(t\mapsto(z-t)^{-1})$ to $(z-S)^{-1}$ for all $z \in \mathbb{C}\setminus \mathbb{R}_+$ and its restriction to the space of all bounded Borel functions on $(0,\infty)$ denoted $\mathcal{L}^\infty((0,\infty))$ is a $\star$-algebra homomorphism  into $\mathcal{L}(H)$, the space of bounded linear maps on $H$. More precisely, we have
$$\forall f \in \mathcal{L}^\infty((0,\infty)): \ \left\|f(S) \right\|\leq \left\|f \right\|_{\infty}.$$
Moreover, for all $f,g \in \mathcal{L}^\infty_{\mathrm{loc}}((0,\infty))$, we have $f(S)g(S)\subset (fg)(S)$ with equality if $g(S) \in \mathcal{L}(H)$. We also recall that $f(S)^*=f^\star (S)$ with $f^\star= \bar f$.

We shall use that if $\varphi :  (0,\infty) \rightarrow \mathbb{C}$ is a Borel function such that 
\begin{equation}\label{Psi+}
    \left| \varphi(t) \right| \leq C \min(\left|t \right|^s,\left|t \right|^{-s}),
\end{equation}
for some constants $C,s >0$ and for all $t>0$,
then the operators $\int_{\varepsilon}^{{1}/{\varepsilon}} \varphi(aS) \ \frac{\mathrm{d}a}{a}$ are uniformly bounded for $0<\varepsilon<1$ and converge strongly in $\mathcal{L}(H)$, namely for all $v \in H$,
\begin{equation}\label{eq:Calderon}
    \lim_{\varepsilon \to 0} \int_{\varepsilon}^{{1}/{\varepsilon}} \varphi(aS)v \ \frac{\mathrm{d}a}{a} = \left ( \int_{0}^{+\infty} \varphi(t) \frac{\mathrm d t}{t} \right ) v,
\end{equation}
where the limit is in $H$. This is the so-called Calder\'on reproducing formula.

In this entire section, we fix a function $\Phi \in \mathcal{D}((0,\infty))$ such that $\int_{0}^{+\infty} \Phi (t) \frac{\mathrm{d} t}{t}=1$. Remark that for all $\alpha \in \mathbb{R}$, $t\mapsto t^\alpha  \Phi(t) \in \mathcal{D} ((0,\infty) )$ and in particular verifies \eqref{Psi+} for some constants $\Tilde{C}, \Tilde{s}>0$ and for all $t>0$.

For $\alpha \in \mathbb{R}$, let $S^{\alpha}$ denote the closed operator $\mathbf{t^{\alpha}}(S)$, which is also injective, positive and self-adjoint. We recall that for all $\alpha, \beta \in \mathbb{R}$, we have 
$$S^{\alpha+\beta}=S^\alpha S^\beta.$$
Denote by $D(S^{\alpha})$ the domain of $S^{\alpha}$. For any element $u\in D(S^\alpha)$, we set
\begin{equation*}
    \left\| u \right\|_{S,\alpha}:=\left\|S^{\alpha}u \right\|_H.
\end{equation*}
We insist on the fact that $\left\| \cdot \right\|_{S,\alpha}$ denotes the homogeneous norm on the domain of $S^\alpha$ and the (Hilbertian) graph norm is 
$(\| \cdot \|_{S,\alpha}^2+ \|\cdot \|_H^2)^{1/2}$. The operator 
\begin{equation}\label{Salpha}
    S^{\alpha}:  (D(S^{\alpha}),  \left\| \cdot \right\|_{S,\alpha}  ) \rightarrow \left ( H, \left\| \cdot \right\|_H \right )
\end{equation}
is isometric with dense range. 
\subsection{An ambient space 
}
We construct an ambient space  in which we can perform all calculations.
Consider the vector space
\begin{align*}
    E_{-\infty}:= \bigcap_{\alpha \in \mathbb{R}} D(S^{\alpha }),
\end{align*}
 endowed with the topology defined using the norms family $(\left\| \cdot \right\|_{S,\alpha} )_{\alpha \in \mathbb{R}}$.
 We recall the following moments inequality 
 \begin{equation*}
    \left \| S^\gamma u \right \|_H\leq   \| S^\alpha u  \|_H^\theta  \| S^\beta u  \|_H^{1-\theta } \ (u \in E_{-\infty}), 
 \end{equation*}
 for all $\gamma=\theta \alpha +(1-\theta)\beta$ and $\theta \in [0,1]$ and $\alpha$ and $\beta$ with same sign \cite[Proposition 6.6.4]{haase2006functional}. Using the moment inequality with the closedness of the powers $S^\alpha$, one can see that $E_{-\infty}$ endowed with the countable norms family $(\left\| \cdot \right\|_{S,\alpha} )_{\alpha \in \mathbb{Z}}$ is in fact a Fr\'echet space. Notice that for all $\alpha \in \mathbb{R}$, $S^{\alpha} : E_{-\infty} \rightarrow E_{-\infty}$ is an isomorphism.
 
The space $E_{-\infty}$ is to be the test space as evidenced in the following lemma.
 \begin{lem}\label{density E_infty}
     $E_{-\infty}$ is dense in $ (D(S^{\alpha}),  \left\| \cdot \right\|_{S,\alpha} )$ for all $\alpha \in \mathbb{R}.$
\end{lem}
\begin{proof}
Let $v \in D(S^\alpha)$. We regularise $v$ by setting for all $\varepsilon \in (0,1)$,
\begin{align*}
    v_\varepsilon := \int_{\varepsilon }^{1/\varepsilon } \Phi (aS)v \ \frac{\mathrm{d} a}{a}.
\end{align*}
Indeed, we show that $v_\varepsilon \in E_{-\infty}$ and $S^\alpha v_\varepsilon \to S^\alpha v$ in $H$ as $\varepsilon\to 0$. First, for all $\beta \in \mathbb{R}$, since {$t \mapsto t^\beta \Phi(at) \in \mathcal{L}^\infty((0,\infty))$}, we have $v_\varepsilon \in D(S^\beta)$ with $S^\beta v_\varepsilon = \int_{\varepsilon }^{1/\varepsilon } S^\beta \Phi (aS)v \ \frac{\mathrm{d} a}{a} $. Hence, $v_\varepsilon \in E_{-\infty}$. Furthermore, as $v \in D(S^\alpha)$, we have  $S^\alpha \Phi (aS)v= \Phi (aS)S^\alpha v$, so that $S^\alpha v_\varepsilon$ converges to $S^\alpha v$ by  the Calder\'on reproducing formula.
\end{proof} 

\begin{rem} The approximation is universal, in the sense that if $v\in D(S^\alpha) \cap D(S^\beta) $, then the approximation occurs simultaneously in both semi-norms. {In particular, $E_{-\infty}$ is dense in $D(S^\alpha)$ for the graph norm. }
    
\end{rem}

Let $E_{\infty}$ denote the topological anti-dual space of $E_{-\infty}$. The reason for which we are interested in such a space is that it provides an ambient space containing a copy of a completion of all spaces $ ( D(S^\alpha), \left\| \cdot \right\|_{S,\alpha})  $. To clarify this claim, we define  \begin{align*}
    \left \| \varphi  \right \|_{E_\alpha} := \sup_{v \in E_{-\infty}\setminus \left \{ 0 \right \}}\frac {| \varphi (v)|}  {   \| v  \|_{S,-\alpha}}
\end{align*} and the vector space
\begin{align*}
    E_{\alpha}:= \{\varphi \in E_\infty :  \left \| \varphi  \right \|_{E_\alpha}<\infty\}. 
\end{align*}
The space $\left ( E_{\alpha},\left \| \cdot \right \|_{E_\alpha} \right )$ is a Banach space. We set
\begin{equation*}
j : H   \rightarrow E_\infty, \ \ v \mapsto j(v):=\langle v ,\cdot \rangle_{ H }.
\end{equation*}
The application $j$ is injective by the density of $E_{-\infty}$ in $H$ in Lemma \ref{density E_infty}. 
\begin{lem}\label{j identification}
  For all $\alpha \in \mathbb{R}$,  $ j_{\scriptscriptstyle{\vert D(S^\alpha)}}:  (D(S^{\alpha}),  \left\| \cdot \right\|_{S,\alpha} )   \rightarrow \left ( E_{\alpha},\left \| \cdot \right \|_{E_\alpha} \right )$ is isometric with dense range.
\end{lem}
\begin{proof}
     If $v \in D(S^\alpha)$ then we can write for all $\Tilde{v} \in E_{-\infty}$,  $j(v)(\Tilde{v})=\langle v, \Tilde{v} \rangle_H = \langle S^\alpha v, S^{-\alpha} \Tilde{v} \rangle_H$. This implies that $j(v) \in E_{\alpha}$ and $\left \| j(v) \right \|_{E_{\alpha }}=\left \| S^\alpha v \right \|_H=\left \| v \right \|_{S,\alpha}$. Now, if $\varphi \in E_\alpha$, then $\varphi \circ S^\alpha$ has a bounded extension on $H$. Using the Riesz representation theorem, there exists $v \in H$ such that $\varphi \circ S^\alpha = \langle v, \cdot \rangle_H$ or equivalently $\varphi = \langle v, S^{-\alpha}\cdot \rangle_H$. Moreover, we have $\left \| \varphi  \right \|_{E_\alpha}= \left \| v  \right \|_H$. Since the range of $S^\alpha$ is dense in $H$,  there exists a sequence $(v_n)_{n \in \mathbb{N}} \in D(S^\alpha)^{\mathbb{N}}$ such that $S^\alpha v_n \to v$ in $H$. Now, we have for all $\Tilde{v}\in E_{-\infty}$,
\begin{align*}
    j(v_n)(\Tilde{v})-\varphi (\Tilde{v})=    \langle v_n, \Tilde{v} \rangle_H -\langle v, S^{-\alpha}\Tilde{v}  \rangle_H  = \langle S^\alpha v_n-v, S^{-\alpha}\Tilde{v} \rangle_H.
\end{align*}
Therefore, $\left \| j(v_n)-\varphi  \right \|_{E_\alpha}=\left \| S^\alpha v_n -v  \right \|_H \to 0$.
\end{proof}

To make clear the identification we will adopt in the next paragraph, we temporarily define the operator $T$ on the Hilbert space $j(H)=H^\star$ by setting 
\begin{equation*}
    D(T):=j(D(S)) \ , \ \ T:=j \circ S \circ j^{-1}.
\end{equation*}
Since $j: H \rightarrow j(H)$ is a unitary operator by Lemma \ref{j identification} when $\alpha=0$, $T$ is unitarily equivalent to $S$. It follows that $T$ has the same properties as $S$. More precisely, $T$ is injective, positive and selfadjoint and we have for all $\alpha \in \mathbb{R}$
$$D(T^\alpha):= j(D(S^\alpha))\ , \ \ T^\alpha= j \circ S^\alpha \circ j^{-1}, \ \ \left \| T^\alpha (j(v))  \right \|_{j(H)}=\left \| S^\alpha v  \right \|_{H}.$$
Using Lemma \ref{j identification}, we have for all $\alpha \in \mathbb{R}$, $D(T^\alpha) \subset E_\alpha$ with dense and isometric inclusion for the homogeneous norm of $T^\alpha$ which means that $E_\alpha$ is a completion of $D(T^\alpha)$ for the homogeneous norm. Moreover, it follows from \eqref{Salpha} that $T^\alpha: D(T^\alpha) \rightarrow j(H)$ is isometric with dense range. Notice that by combining Lemma \ref{density E_infty} and Lemma \ref{j identification}, $j(E_{-\infty})=\bigcap_{\alpha \in \mathbb{R}} D(T^{\alpha })$ is dense in all the spaces $E_\alpha$. {Furthermore, we have for all $\alpha \in \mathbb{R}$, $E_\alpha \cap j(H)= D(T^\alpha)$}. The advantage is that all the spaces mentioned here are contained in $E_\infty$, an anti-dual space of a Fréchet space, which is in particular a Hausdorff topological vector space.

From now on, by making the identification of $H$ with $j(H)$ and $S$ with $j\circ S \circ j^{-1}$ as above, we assume that $H$ is contained in a Hausdorff topological vector space $E_\infty$ that contains a completion of all the domains of $S^\alpha$ for the homogeneous norms and we denote by $D_{S,\alpha}$ this completion of $D(S^\alpha)$ in $E_\infty$. Moreover, we  have that for all $\alpha \in \mathbb{R}$, $D_{S,\alpha}\cap H = D(S^\alpha)$.
By Lemma \ref{density E_infty}, the space $E_{-\infty} = \bigcap_{\alpha \in \mathbb{R}} D(S^\alpha )$ is dense in all these completions. Moreover, there is a sesquilinear continuous duality form $\langle w, v\rangle$ on $E_\infty\times E_{-\infty}$ which extends the inner product on $H$. The functional calculus of $S$ extends to $E_\infty$ by $\langle f(S)w, v\rangle= \langle w, f^*(S)v\rangle$ whenever {$f\in \mathcal{L}^\infty_{\mathrm{loc}}((0,\infty))$ and $f(t)=O(t^a)$ at  $0$ and $f(t)=O(t^b)$ at $\infty$} as $f^*(S):E_{-\infty}\to E_{-\infty}$ is bounded. In particular, $S^\alpha: E_\infty \to E_\infty$ is an automorphism and we have $D_{S,\alpha}=\left\{u \in E_\infty  :  S^\alpha u \in H \right\}$. The restriction of $S^\alpha$ to $D_{S,\alpha}$  agrees with  the unique extension of $S^{\alpha}:  (D(S^{\alpha}),  \left\| \cdot \right\|_{S,\alpha}  ) \rightarrow \left ( H, \left\| \cdot \right\|_H \right )$ (see \eqref{Salpha}). The norm on $D_{S,\alpha}$ is $\left \| S^\alpha \cdot \right \|_H$ and we keep denoting it by $\left \|  \cdot \right \|_{S,\alpha}$ (and it makes it a Hilbert space). We record the following lemma.
 
\begin{lem}\label{density}
    Let $\alpha \in \mathbb{R}$. Then, there are dense inclusions 
    \begin{equation*}
   E_{-\infty} \hookrightarrow    D(S^\alpha) \hookrightarrow   D_{S,\alpha} \hookrightarrow E_\infty.
    \end{equation*}
Moreover, the family $\left ( D_{S,\alpha} \right )_{\alpha \in \mathbb{R}}$ is a complex (and real) interpolation family.
\end{lem}
\begin{proof}
For the first statement, the first two inclusions are already known to be dense. We show it for the last one. Indeed, if $w\in E_\infty$, then let $\mathcal{F}_w$ be a finite subset of $\mathcal{F}$ such that we have $|\langle w, v\rangle| \le C_w \sup_{\gamma \in \mathcal{F}_w} \|S^\gamma v\|_H$ for all $v\in E_{-\infty}$. The approximation procedure using the duality form and  using that $\Phi^\star=\Phi$ shows that $w_\varepsilon$ converges to $w$ in $E_\infty$.  We claim that $ w_\varepsilon \in D_{S,\alpha}$ for all $\alpha$ (hence, $w_\varepsilon\in E_{-\infty}$). Indeed,  if we pick $v\in E_{-\infty}$, then $S^{\alpha} v_\varepsilon \in E_{-\infty}$ and  $S^{\alpha+\gamma} v_\varepsilon \in H$ with norm controlled by $C_{\varepsilon,\alpha+\gamma} \|v\|_H$. Thus, we have 
$$    |\langle S^\alpha w_\varepsilon, v\rangle| = |\langle w, S^{\alpha} v_\varepsilon \rangle| \le C_w \sup_{\gamma \in \mathcal{F}_w} \|S^{\alpha+\gamma} v_\varepsilon\|_H \le C_w \sup_{\gamma \in \mathcal{F}_w} C_{\varepsilon,\alpha+\gamma}\, \|v\|_H,$$
from which  the claim follows. Finally, the fact that family $\left ( D_{S,\alpha} \right )_{\alpha \in \mathbb{R}}$ is a complex (and real) interpolation family is proved in \cite{auscher1997holomorphic}.
\end{proof}
For any real $\alpha$, the sesquilinear form $(u,v) \mapsto \langle S^\alpha u, S^{-\alpha} v \rangle_H$ defines a canonical duality pairing between $D_{S,\alpha}$ and $D_{S,-\alpha}$ which is simply the inner product $\langle \cdot, \cdot \rangle_H$ extended from $E_{-\infty}\times E_{-\infty}$ to $D_{S,\alpha} \times D_{S,-\alpha}$. In fact, we have for all $(u,v) \in E_{-\infty} \times E_{-\infty}$,
\begin{equation*}
    \sup_{w \in E_{-\infty}\setminus\left\{0 \right\}} \frac{\left|\langle u, w \rangle_H \right|}{\ \ \ \ \left\|w \right\|_{S,-\alpha}}=\left\|u \right\|_{S,\alpha} \ \  , \ \ \ \sup_{w \in E_{-\infty}\setminus\left\{0 \right\}} \frac{\left|\langle w, v \rangle_H \right|}{\ \ \ \ \left\|w \right\|_{S,\alpha}}=\left\|u \right\|_{S,-\alpha}.
\end{equation*}
For $(u,v) \in D_{S,\alpha} \times D_{S,-\alpha}$, we denote $\langle u, v \rangle_{H,\alpha}:= \langle S^\alpha u, S^{-\alpha} v \rangle_H$. 
It also coincides with the sesquilinear duality  $\langle u, v \rangle$  on $E_\infty \times E_{-\infty}$ when $u\in D_{S,\alpha}$ and $ v\in E_{-\infty}$. We have the following lemma.

\begin{lem}\label{Interpolation Jalpha bien définie}
    Let $\alpha, \beta \in \mathbb{R}$. If $u \in D_{S,\alpha}\cap D_{S,\beta}$ and $v \in D_{S,-\alpha} \cap D_{S,-\beta}$, then
   \begin{equation*}
        \langle u, v \rangle_{H,\alpha}= \langle u, v \rangle_{H,\beta}.
   \end{equation*}
\end{lem}
\begin{proof}
    The approximations $u_\varepsilon$ and $v_\varepsilon$ belong to $E_{-\infty}$ so $\langle u_\varepsilon, v_\varepsilon \rangle_{H,\alpha}=\langle u_\varepsilon, v_\varepsilon \rangle_H= \langle u_\varepsilon, v_\varepsilon \rangle_{H,\beta}$ for all $\varepsilon>0$. The result follows when $\varepsilon$ tends to $0$ as $u_\varepsilon$ converges to $u$ in both spaces $D_{S,\alpha}$ and $D_{S,\beta}$ and $v_\varepsilon$ converges to $v$ in both spaces $D_{S,-\alpha}$ and $D_{S,-\beta}$.
\end{proof}
\subsection{Spaces of test functions and distributions}\label{S2S1}
For $I$ an open interval of $\mathbb{R}$, we denote by $\mathcal{D}(I;E_{-\infty})$ the space of $E_{-\infty}$-valued $C^\infty$ functions on $I$ with compact support. The space $\mathcal{D}(I;E_{-\infty})$ is endowed with the usual inductive limit topology and contains $\mathrm{span}(\mathcal{D}(I)\otimes E_{-\infty})$ as a dense subspace.

We refer to \cite{hytonen2016analysis} for Banach-valued $L^p(I;B)$ spaces. The density lemma below explains why it is relevant to take $\mathcal{D}(I;E_{-\infty})$ as the space of test functions.
\begin{lem}\label{Lemma density}
    $\mathcal{D}(I;E_{-\infty})$ is dense in $L^p(I; D_{S,\alpha})$ for all $\alpha \in \mathbb{R}$ and $p \in [1,\infty)$.
\end{lem}
\begin{proof}
     It is enough to consider the case $\alpha=0$, that is to prove the density of $\mathcal{D}(I;E_{-\infty})$ in $L^p(I, H)$ since $S^{-\alpha} : E_{-\infty} \rightarrow E_{-\infty}$ is an isomorphism and therefore $S^{-\alpha}: \mathcal{D}(I;E_{-\infty}) \rightarrow \mathcal{D}(I;E_{-\infty})$, where we set $(S^\alpha f)(t)=S^\alpha(f(t))$ for all $t\in I$ and $f\in \mathcal{D}(I;E_{-\infty})$. It is enough also to prove that $\mathcal{D}(I;E_{-\infty})$ is dense in $\mathcal{D}(I;H)$ for the $L^p$ norm since the latter is dense in $L^p(I; H)$. To do so, we fix $f \in \mathcal{D}(I;H)$ and regularize it by setting for all $\varepsilon>0$, and $t\in I$
\begin{align*}
    f_\varepsilon (t):= \int_{\varepsilon}^{1/\varepsilon} \Phi (aS)f(t)\ \frac{\mathrm d a}{a} =  \left ( \int_{\varepsilon}^{1/\varepsilon} \Phi (aS)\ \frac{\mathrm d a}{a} \right ) f(t).
\end{align*}
It is obvious that $f_\varepsilon \in \mathcal{D}( I ; E_{-\infty})$ and by the Calder\'on reproducing formula $f_\varepsilon(t) \underset{\varepsilon \to 0}{\rightarrow} f(t) $ pointwisely in $H$. The uniform boundedness in $H$ of the approximate Calder\'on operators yields the domination allowing to apply the dominated convergence theorem.
\end{proof}

 We can now define a space of distributions in which $S$ plays the role of a differential operator. Specifically, we denote by $\mathcal{D}'( I ; E_\infty)$ the space of all bounded anti-linear maps $u :  \mathcal{D}( I ; E_{-\infty})\to  \mathbb{C}$. Bounded means that for any compact set $K \subset  I $, there exist two constants $C_K >0$ and $m_K \in \mathbb{N}$, and a finite set $\mathcal{A}_K \subset \mathbb{Z}$ such that 
\begin{align*}
    \forall \varphi \in \mathcal{D}( I , E_{-\infty}), \ \mathrm{supp}(\varphi) \subset K \Rightarrow \left|\llangle u ,\varphi  \rrangle_{\mathcal{D}',\mathcal{D}} \right|\leq C_K \sup_{\left|j \right|\leq m_K,\,  \alpha \in \mathcal{A}_K} \sup_{t\in K} \|\partial_t^j S^\alpha \varphi(t) \|_H .
\end{align*}
For convenience, we use the notation with partial derivative in $t$ for the derivative in $t$ thinking of future applications to concrete PDE. We also use a double bracket notation for dualities between vector-valued functions and distributions.

 For $\alpha \in \mathbb{R}$, we can embed $L^1_{\mathrm{loc}}(I; D_{S,\alpha})$ in  $\mathcal{D}'\left ( I ; E_\infty \right )$ through the classical identification :
\begin{align*}
J_\alpha \colon L^1_{\mathrm{loc}}(I; D_{S,\alpha})   & \rightarrow  \mathcal{D}'\left ( I ; E_\infty \right )\\
f&\mapsto J_\alpha f :  \varphi \mapsto \int_{I} \langle f(t) ,\varphi(t)  \rangle_{H,\alpha} \ \mathrm{d} t = \int_{I } \langle S^\alpha f(t) ,S^{-\alpha }\varphi(t)  \rangle_H \ \mathrm{d} t.
\end{align*}
{This is well defined by the following lemma.
\begin{lem}\label{lem: Jalpha bien définie}
 For all $\alpha, \beta \in \mathbb{R}$, $J_\alpha = J_\beta $ on $L^1_{\mathrm{loc}}(I; D_{S,\alpha}) \cap L^1_{\mathrm{loc}}(I; D_{S,\beta})$. 
\end{lem}
\begin{proof}
    Straightforward corollary of Lemma \ref{Interpolation Jalpha bien définie}.
\end{proof}
Finally, the identification is achieved thanks to the following lemma.}
\begin{lem}
     $J_\alpha$ is injective for all $\alpha \in \mathbb{R}$.
\end{lem}
\begin{proof}
Testing with $\varphi = \theta \otimes a$, with $\theta\in \mathcal{D}(\mathbb{R)}$ real-valued and $a\in E_{-\infty}$, we easily conclude that $J_\alpha(f)=0$ implies $\langle S^\alpha f, S^{-\alpha} a\rangle_H=0$ a.e. on $I$. This implies that $S^\alpha f=0$ in $H$ a.e. on $I$, hence $f=0$ in $D_{S,\alpha}$  a.e. on $I$.
\end{proof}

Consequently, the space $L^p(I; D_{S,\alpha})$ can be identified to a sub-space of $\mathcal{D}'( I; E_\infty)$. We would like to apply powers of $S$ to a distribution like we apply them to functions valued in $E_{-\infty}.$ The definition below covers this.
\begin{defn}
    For $\alpha \in \mathbb{R}$ and $u \in \mathcal{D}'( I ; E_\infty)$, we define the distribution $ S^\alpha u $ by setting  
    \begin{align*}
        \llangle S^\alpha u , \varphi \rrangle_{\mathcal{D}',\mathcal{D}} := \llangle  u , S^\alpha\varphi \rrangle_{\mathcal{D}',\mathcal{D}}, \ \forall \varphi \in \mathcal{D}(I; E_{-\infty}).
    \end{align*}
\end{defn}
\begin{rem}
    For $u \in \mathcal{D}'( I ; E_\infty)$, $S^\alpha u \in L^p(I; H)$ is equivalent to $ u \in L^p(I; D_{S,\alpha})$. Furthermore, powers of $S$ commute with  derivatives in $t$ : for all $ k \in \mathbb{N}$ and $ \alpha \in \mathbb{R}$, $\partial_t^k S^\alpha  = S^\alpha \partial_t^k $. 
\end{rem}
When $I=\mathbb{R}$, we can use the space of tempered distributions adapted to $S$. Let us start by the Schwartz class $\mathcal{S}(\mathbb{R};E_{-\infty})$ defined by
\begin{align*}
    \mathcal{S}(\mathbb{R};E_{-\infty}):= \left \{ \varphi \in C^\infty(\mathbb{R};E_{-\infty}) \ : \ \forall k, \ell \in \mathbb{N}, \ t^k \partial_t^\ell \varphi(t) \underset{\left | t \right |\to \infty}{\rightarrow}0 \ \mathrm{in} \ E_{-\infty}  \right \},
\end{align*}
which is a Fréchet space for a suitable countable family of norms.  Moreover, $\mathcal{D}(\mathbb{R}; E_{-\infty})$ is dense in $\mathcal{S}(\mathbb{R};E_{-\infty})$ by the same argument as for the usual distributions. We denote by $\mathcal{S'}(\mathbb{R};E_\infty)$ the topological dual space of $\mathcal{S}(\mathbb{R};E_{-\infty})$. It is a subspace of $\mathcal{D}'(\mathbb{R}; E_{\infty})$ containing $L^p(\mathbb{R}; D_{S,\alpha})$ for all $\alpha \in \mathbb{R}$ and $p \in [1,\infty]$. For the proofs of this and the theorems below, we refer to \cite{zuily2002elements} for the classical distributions and the same proofs work here. It can be proven that $\mathcal{S}(\mathbb{R};E_{-\infty})$ is dense in $\mathcal{S'}(\mathbb{R};E_\infty)$, and more generally, that $\mathcal{D}(I;E_{-\infty})$ is dense in $\mathcal{D'}(I;E_\infty)$ for any open set $I \subset \mathbb{R}$. However, this is not important for the discussion that follows, so we leave the verification to the interested reader.

As in the classical case, we will first define $\mathcal{F}$ on $L^1(\mathbb{R};H)$, then on $\mathcal{S}(\mathbb{R};E_{-\infty})$ and finally on $\mathcal{S'}(\mathbb{R};E_\infty)$ by duality.
\begin{defn} The Fourier transform is defined on $L^1(\mathbb{R};H)$ by setting for all $f \in L^1(\mathbb{R};H)$ and $\tau \in \mathbb{R}$
\begin{align*}
    \mathcal{F}(f)(\tau):=\hat{f}(\tau):= \int_{\mathbb{R}} e^{-\textit{i}\tau t} f(t)  \ \mathrm{d}t.
\end{align*}
We define $\overline{\mathcal{F}}$ by changing $-i\tau t $ to $i\tau t$ in the integral.
\end{defn}
The Fourier transform on $\mathcal{S}(\mathbb{R};E_{-\infty})$ enjoys many properties as we recall below.

\begin{prop}
    \label{TFS} The Fourier transform $\mathcal{F}$ enjoys the following properties : 
\begin{enumerate}
    \item $\mathcal{F}: \mathcal{S}(\mathbb{R};E_{-\infty}) \rightarrow \mathcal{S}(\mathbb{R};E_{-\infty})$ is an automorphism verifying for all $\varphi \in \mathcal{S}(\mathbb{R};E_{-\infty})$, $k \in \mathbb{N}$ and $\alpha \in \mathbb{R}$ $$\mathcal{F}(S^\alpha \varphi )=S^\alpha \mathcal{F}(\varphi ), \ \partial_\tau^k \mathcal{F}(\varphi )=  \mathcal{F}((-\textit{i}t)^k\varphi), \ \mathcal{F}(\partial_t^k\varphi )= (\textit{i}\tau )^k  \mathcal{F}(\varphi ).$$
    \item For all $\alpha \in \mathbb{R}$, $\mathcal{F}$ extends to an isomorphism on $ L^2(\mathbb{R};D_{S,\alpha})$ which verifies a Plancherel equality.
\end{enumerate}
\end{prop}
We can now transport the Fourier transform $\mathcal{F}$ to $\mathcal{S'}(\mathbb{R};E_\infty)$ by sesquilinear duality.

\begin{defn}
We define the Fourier transform $\mathcal{F}$ on $\mathcal{S'}(\mathbb{R};E_\infty)$ by setting 
\begin{equation*}
    \llangle \mathcal{F}u,\varphi \rrangle_{\mathcal{S'},\mathcal{S}}:= \llangle \hat{u},\varphi \rrangle_{\mathcal{S'},\mathcal{S}} := \llangle u,\overline{\mathcal{F}}(\varphi) \rrangle_{\mathcal{S'},\mathcal{S}},\ u \in \mathcal{S'}(\mathbb{R};E_\infty), \ \varphi \in \mathcal{S}(\mathbb{R};E_{-\infty}).
\end{equation*}
\end{defn}
From Proposition \ref{TFS}, we deduce the following proposition regarding the Fourier transform on $\mathcal{S'}(\mathbb{R};E_\infty)$. 
\begin{prop}\label{TFS'}
     $\mathcal{F}: \mathcal{S'}(\mathbb{R};E_\infty) \rightarrow \mathcal{S'}(\mathbb{R};E_\infty)$ is an automorphism and satisfies the property (1) and its restriction to $ L^2(\mathbb{R};D_{S,\alpha})$ agrees with the operator in (2)  as in the statement above. 
\end{prop}

For $\alpha \in \mathbb{R}$, we denote by $D_t^\alpha$ the time-derivative of order $\alpha$. More precisely, if $u \in \mathcal{S'}(\mathbb{R};E_\infty)$ is such that $\left | \tau \right |^{\alpha} \mathcal{F}u  \in \mathcal{S'}(\mathbb{R};E_\infty)$, we set 
\begin{equation*}
    D_t^\alpha u := \mathcal{F}^{-1} \left ( \left | \tau \right |^{\alpha} \mathcal{F}u \right ) .
\end{equation*}

\section{Abstract heat equations }\label{Section 3}
 In this section, we study well-posedness of the abstract heat equation $\partial_t u + S^2u =f$ where the role of the Laplacian is played by the square of the self-adjoint operator $S$.  These well-posedness results will imply  embeddings and energy inequalities in the spirit of Lions that will be described in the next section. 
The abstract heat operator in $\mathbb{R}$ is $\partial_t +S^2$. The  
backward operator $-\partial_t +S^2$  corresponds to reversing time, and the results are exactly the same and are often proved and used simultaneously.

\subsection{Solving the abstract heat equation using the Fourier method}

Working on the real line makes the Fourier transform in time available and is a key tool to obtain homogeneous estimates in a simple way.

\subsubsection{Uniqueness in homogeneous energy space}
We begin with a uniqueness result which is key to our discussion. 
\begin{prop}[Uniqueness in homogeneous energy space]\label{Unicité}
    Let $u \in \mathcal{D}'(\mathbb{R}; E_\infty)$ be a solution of $\partial_t u + S^2 u =0$ in $\mathcal{D}'(\mathbb{R};E_\infty)$. If $u \in L^2(\mathbb{R};D_{S,\alpha})$ for some $\alpha \in \mathbb{R}$, then $u=0$.
\end{prop}
\begin{proof} 
As $S^\alpha$ is an isomorphism on $E_\infty$ which commutes with time derivatives, $v=S^\alpha u$ satisfies the same equation, hence we may assume $\alpha=0$ and $u\in L^2(\mathbb{R};H)$. As this is a subset of $ \mathcal{S}'(\mathbb{R};E_\infty)$, by applying the Fourier transform to this equation, we have for all $\varphi \in \mathcal{S}(\mathbb{R};E_{-\infty})$
\begin{equation}\label{v=0}
    \int_\mathbb{R} \langle \hat{u}(\tau ), (- i \tau+S^2 ) \varphi(\tau) \rangle_H \ \mathrm{d}\tau =0.
\end{equation}
Take a sequence $(\varphi_k)_{k \in \mathbb{N}} \in \mathcal{D}(\mathbb{R};E_{-\infty})^\mathbb{N}$ such that $\varphi_k \underset{}{\rightarrow}  \hat{u}$ in $L^2(\mathbb{R};H)$ and $0 \notin \mathrm{supp}(\varphi_k)$, for all $k \in \mathbb{N}$.
Taking $\tau \mapsto (-{i}\tau+S^2)^{-1}\varphi_k(\tau)$ as a test function in \eqref{v=0} and letting $k \to +\infty$, we have
\begin{equation*}
    \int_\mathbb{R}\left \| \hat{u}(\tau) \right \|_H^2 \ \mathrm{d}\tau =0 .
\end{equation*}
By Plancherel, we have then $u=0$.
\end{proof} 
\begin{cor}[Invertibility on abstract Schwartz functions and tempered distributions]\label{cor:iso} The operator $\partial_t  + S^2$ is an isomorphism on $\mathcal{S}(\mathbb{R};E_{-\infty})$  and on $\mathcal{S}'(\mathbb{R};E_{\infty})$.
\end{cor}
\begin{proof} We begin with the result on $\mathcal{S}(\mathbb{R};E_{-\infty})$.
        The boundedness is clear. The injectivity follows from the above proposition. The surjectivity is as follows. By Fourier transform, it suffices to show the surjectivity for $i\tau+S^2$. If $\hat f\in \mathcal{S}(\mathbb{R};E_{-\infty}) $, then $S^{-2}\hat f \in L^2(\mathbb{R};H)$ and by the uniform boundedness  of $(i\tau+S^2)^{-1} S^2$,   $g(\tau)=(i\tau+S^2)^{-1} S^2 (S^{-2}\hat f(\tau)) \in L^2(\mathbb{R};H)$ with $i\tau g(\tau)+ S^2 g (\tau)= \hat f(\tau)$. Shifting with powers of $S$, we have $g\in L^2(\mathbb{R}; E_{-\infty})$. Setting $\hat u=g$, we see that $\partial_t u=f-S^2u$, so $\partial_t u\in L^2(\mathbb{R}; E_{-\infty})$ and by iteration, we have $u\in C^\infty(\mathbb{R}; E_{-\infty})$. The decay is easily checked following the argument and using $\tau$-derivatives of the resolvent $(i\tau+S^2)^{-1}$. 

        This applies to the backward operator $-\partial_t+S^2$. Hence, by duality, we obtain the result on $\mathcal{S}'(\mathbb{R};E_{\infty})$.
\end{proof}

\begin{rem}
  In  the sequel, we shall focus on $\alpha=1$ in Proposition \ref{Unicité} to make $H$ the pivotal space, but clearly, one can shift to this case by applying powers of $S$. 
\end{rem}
\subsubsection{Solution and source spaces} We begin with recalling the following  result of Lions  for the sake of completeness, but we shall not use this result and prove a stronger one.
\begin{prop}[Solving the abstract heat equation à la Lions]\label{lem: Lions}
    If $f\in L^2(\mathbb{R}; D_{S,-1})$, then there exists $u\in  L^2(\mathbb{R}; D_{S,1})$ such that $\partial_t u + S^2 u =f$ in $\mathcal{D}'(\mathbb{R};E_\infty)$.
\end{prop}
\begin{proof}
    It is straightforward application of the Lions representation theorem \cite[Théorème 1.1]{lions2013equations} in the Hilbert space $L^2(\mathbb{R}; D_{S,1})$. 
\end{proof}

As we said, we now argue with in mind that $H$, or rather $L^2(\mathbb{R}; H)$,  is the pivotal Hilbert space. Define 
\begin{align*}
    V_1&:= L^2(\mathbb{R};D_{S,1}),\\
    V_{-1}&:=\left\{ u \in L^2(\mathbb{R};D_{S,1}) \ : \ \partial_t u \in L^2(\mathbb{R};D_{S,-1}) \right\}.
\end{align*}
$V_{1}$ is the uniqueness space and $V_{-1}$ is the space to which the solution belongs when the source is taken in $L^2(\mathbb{R}; D_{S,-1})$ according to Lions' result.
However, for the  heat equation, Fourier methods are particularly handy to prove this  and also allow more source spaces.  

We  introduce a hierarchy of intermediate solution and source spaces. For $ \alpha \in [-1,1]$, define the following respective solution and source spaces 
\begin{align*}
    V_\alpha &:=\left \{ u \in L^2(\mathbb{R};D_{S,1}) : D_t^{\frac{1-\alpha}{2}}u \in L^2(\mathbb{R};D_{S,\alpha})\right \},\\
    W_\alpha &:= \left \{ D_t^{\frac{1+\alpha}{2}}g \ : \ g \in L^2(\mathbb{R};D_{S,\alpha}) \right \}.
\end{align*}
with 
\begin{align*}
    \left \| u \right \|_{V_\alpha}&:=\left ( \left \| u \right \|_{L^2(\mathbb{R};D_{S,1})}^2+   \| D_t^{\frac{1-\alpha}{2}}u  \|_{L^2(\mathbb{R}; D_{S,\alpha})}^2\right )^{1/2},
    \\
    \| f\|_{W_\alpha}&:=\| D_t^{-\frac{1+\alpha}{2}}f\|_{L^2(\mathbb{R};D_{S,\alpha})}.
\end{align*}
  We can think of $V_\alpha =  L^2(\mathbb{R};D_{S,1}) \cap \dot{H}^{\frac{1-\alpha}{2}}(\mathbb{R};D_{S,\alpha})$ using homogeneous Sobolev spaces on the real line but this presentation avoids having to define these spaces. 
In the same manner, we think of $W_\alpha=\dot{H}^{-\frac{1+\alpha}{2}}(\mathbb{R};D_{S,\alpha})$. Remark that $W_{-1}=L^2(\mathbb{R}; D_{S,-1}).$

The following lemma summarizes some properties of the spaces $V_\alpha$ and $W_\alpha$ and their relation.
\begin{lem}[Properties of intermediate spaces]\label{Spaces V_alpha} Fix  $-1 \leq \alpha \leq \alpha' \leq 1$. We have the following assertions. 
\begin{enumerate}
        \item $V_\alpha $ is a well-defined  subspace of  $\mathcal{S'}(\mathbb{R};E_\infty)$, $\left ( V_\alpha, \left \| \cdot \right \|_{V_\alpha} \right )$ is a Hilbert space, and we have 
                 \begin{equation*}
                     V_\alpha = \left \{ u \in \mathcal{S'}(\mathbb{R};E_\infty) : S^\alpha (S+\left | \tau \right |^{1/2})^{1-\alpha} \hat{u} \in L^2(\mathbb{R};H)\right \}, \ \left \| u \right \|_{V_\alpha} \sim  \| S^\alpha (S+\left | \tau \right |^{1/2})^{1-\alpha}\hat{u}  \|_{L^2(\mathbb{R};H)}.
                 \end{equation*}
        \item We have the following chain of continuous and dense inclusions:
        \begin{align*}
            \mathcal{S}(\mathbb{R};E_{-\infty}) \hookrightarrow V_\alpha \hookrightarrow V_{\alpha'} \hookrightarrow \mathcal{S'}(\mathbb{R};E_{\infty}).
        \end{align*}
         \item { $W_\alpha$ is a subspace of $\mathcal{S'}(\mathbb{R}; E_\infty)$, and $\left( W_\alpha, \left\| \cdot \right\|_{W_\alpha} \right)$ is a Hilbert space. We have a dense inclusion $\mathcal{S}_0(\mathbb{R}; E_{-\infty}) \hookrightarrow W_\alpha$, where
$\mathcal{S}_0(\mathbb{R}; E_{-\infty}) := \{ f \in \mathcal{S}(\mathbb{R}; E_{-\infty}) \mid \hat{f}(0) = 0 \}. $}
        \item Let $V_\alpha^\star$ denote the anti-dual space of $V_\alpha$ with respect to $\langle \cdot , \cdot \rangle_{L^2(\mathbb{R};H)}$. It is a subspace of $\mathcal{S}'(\mathbb{R};E_\infty)$ and  $V_\alpha^\star = L^2(\mathbb{R};D_{S,-1}) + W_{-\alpha}$ with the following estimate
    \begin{align*}     
   \left \| \omega  \right \|_{V_\alpha^\star} \sim \inf  \left\{  \left \| f \right \|_{L^2(\mathbb{R};H)}+\left \| g \right \|_{L^2(\mathbb{R};D_{S,-\alpha})}  \ : \ {\omega=Sf+D_t^{\frac{1-\alpha}{2}}g}\right\}.
\end{align*}
         
    \end{enumerate}
\end{lem}
\begin{proof} Let us first prove (1) : $V_\alpha $ is a well-defined  subspace of  $\mathcal{S'}(\mathbb{R};E_\infty)$. In fact, if $u \in L^2(\mathbb{R};D_{S,1})$, then, by Proposition \ref{TFS'}, it follows that $ \left | \tau  \right |^{\frac{1-\alpha}{2}} S \hat{u} \in L^1_{\mathrm{loc}}(\mathbb{R};H)$. Furthermore, for $\varphi \in \mathcal{S}(\mathbb{R};E_{-\infty})$, we have using Cauchy-Schwarz inequality
\begin{align*}
     \int_{\mathbb{R}}  |\langle \left | \tau  \right |^{\frac{1-\alpha}{2}} S\hat{u}, S^{-1 }\varphi  \rangle_H |\, \mathrm d \tau  \leq \left \| S\hat{u} \right \|_{L^2(\mathbb{R};H)}   \||\tau|^{\frac{(1-\alpha)}{2}}  S^{-1 }\varphi   \|_{L^2(\mathbb{R};H)} , 
\end{align*}
and one can define $\left | \tau  \right |^{{\frac{1-\alpha}{2}}}\hat u \in \mathcal{S'}(\mathbb{R};E_\infty)$ by
\begin{align*}
 \llangle \left | \tau  \right |^{\frac{1-\alpha}{2}}\hat u, \varphi  \rrangle_{\mathcal{S'},\mathcal{S}} =  \int_{\mathbb{R}}  \langle \left | \tau  \right |^{\frac{1-\alpha}{2}} S\hat{u}, S^{-1 }\varphi  \rangle_H\,  \mathrm d \tau . 
\end{align*}
Finally, $ S^\alpha\left | \tau  \right |^{{\frac{1-\alpha}{2}}}\hat u$ exists in $\mathcal{S'}(\mathbb{R};E_\infty)$ and agrees with $ \left | \tau  \right |^{{\frac{1-\alpha}{2}}}S^\alpha\hat u$. The Hilbert space property (in particular, the completeness) is easy.
Next, the proof of the set equality and the norms equivalence in (1) is easy using the boundedness of the operators $S^{1-\alpha}(S+\left | \tau \right |^{1/2})^{-(1-\alpha)}$ and $|\tau|^{\frac{1-\alpha }2}(S+\left | \tau \right |^{1/2})^{-(1-\alpha)}$ on $L^2(\mathbb{R};H)$.

For the  point (2),  the inclusion of $\mathcal{S}(\mathbb{R};E_{-\infty})$ in $V_\alpha$ follows easily from (1). To check $V_\alpha \hookrightarrow V_{\alpha'}$ if $\alpha<\alpha'$, write
\begin{equation*}
    S^{\alpha'} (S+\left | \tau \right |^{1/2})^{1-\alpha'}=S^{\alpha'-\alpha}(S+\left | \tau \right|^{1/2})^{\alpha-\alpha'} S^\alpha (S+\left | \tau \right |^{1/2})^{1-\alpha},
\end{equation*}
and use (1) together with the boundedness of $S^{\alpha'-\alpha}(S+\left | \tau \right|^{1/2})^{\alpha-\alpha'}$ on $ L^2(\mathbb{R};H)$. The density of $\mathcal{S}(\mathbb{R};E_{-\infty})$ into $V_\alpha$ can be deduced using Lemma \ref{Lemma density} and that of $V_\alpha$ into $V_{\alpha'}$ follows. Finally, although we do not need this later on, the density of $V_\alpha$ in $\mathcal{S'}(\mathbb{R};E_{\infty})$ hold as $\mathcal{S}(\mathbb{R};E_{-\infty})$ is dense in $\mathcal{S'}(\mathbb{R};E_\infty)$. 

For point (3), since $-1\le \alpha\le 1$, the inclusions are clear together with the Hilbert space property. The density is as follows.  Let $f\in W_\alpha$.
By definition,  let $g \in L^2(\mathbb{R};D_{S,\alpha})$ such that $\hat{f}=\left| \tau \right|^{\frac{1+\alpha}{2}} \hat{g}$ in $\mathcal{S}'(\mathbb{R}; E_{\infty})$. Note that the right hand side also belongs to $L^1_{\mathrm{loc}}(\mathbb{R}; D_{S,\alpha})$. Take a sequence $(\varphi_k)_{k \in \mathbb{N}} \in \mathcal{S}(\mathbb{R}; E_{-\infty})^{\mathbb{N}}$ such that $\hat{\varphi_k } \to \hat{g}$ in $L^2(\mathbb{R};D_{S,\alpha})$ and $0 \notin \mathrm{supp}(\hat{\varphi_k}).$ Take $f_k:= \mathcal{F}^{-1}(\left| \tau \right|^{(1+\alpha) /2}\hat{\varphi_k})$, then $f_k \in \mathcal{S}_0(\mathbb{R}; E_{-\infty})$ and $f_k {\rightarrow} f$ in $W_\alpha$.

The proof of (4) is standard, using Fourier transform and that powers of $S$ commute with multiplication by powers of $|\tau|$, and the calculus occurs using the duality between $\mathcal{S'}(\mathbb{R};E_\infty)$ and $\mathcal{S}(\mathbb{R};E_{-\infty})$. \end{proof}

\subsubsection{The main Theorem}
Now, we come to the main result of this subsection.

\begin{thm}[Invertibility on intermediate spaces]\label{ISOMORPHISM}
    For all $\alpha\in [-1,1]$, the operator $ \partial_t + S^2 $, defined on $\mathcal{S}(\mathbb{R};E_{-\infty})$, extends to a bounded and invertible  operator $ A_\alpha :  V_\alpha \rightarrow V_{-\alpha}^\star$, which agrees with the restriction of $\partial_t + S^2$ acting on $\mathcal{S}'(\mathbb{R};E_\infty)$.
\end{thm}
\begin{proof}
From Fourier transform, $u\in V_\alpha$ if and only if $S\hat u \in L^2(\mathbb{R}; H)$ and $|\tau|^{\frac{1-\alpha}{2}} S^\alpha \hat u \in L^2(\mathbb{R}; H)$. 
Hence, $f=\partial_tu + S^2u$ is easily seen to belong to  $  L^2(\mathbb{R};D_{S,-1}) + W_{\alpha}=V_{-\alpha}^\star$. The density of $\mathcal{S}(\mathbb{R};E_{-\infty})$ in $V_\alpha$ yields the bounded extension operator $A_{\alpha}$ and notice that $A_{\alpha}=\partial_t + S^2$ on $V_\alpha$.

To show the invertibility, it suffices to prove that the restriction of the inverse of $\partial_t + S^2$ on $\mathcal{S}'(\mathbb{R};E_\infty)$ to  $V_{-\alpha}^\star$ is bounded into $V_\alpha$. Let $w\in V_{-\alpha}^\star$ and again, by Fourier transform,  write $\hat w= S \hat f+ |\tau|^{\frac{1+\alpha}{2}}S^{-\alpha} \hat g$, with $f,g\in L^2(\mathbb{R};H)$.  Define $u\in \mathcal{S}'(\mathbb{R};E_\infty) $ by $\hat u(\tau)= \left ( i\tau+S^2 \right )^{-1} \hat{g}(\tau)$. That $ S\hat u \in L^2(\mathbb{R};H) $ follows from the uniform boundedness (with respect to $\tau$) of $S^2(i\tau + S^2)^{-1}$ and $|\tau|^{\frac{1+\alpha}{2}}S^{1-\alpha}(i\tau + S^2)^{-1}$ and that $|\tau|^{\frac{1-\alpha}{2}}S^{\alpha} \hat u \in L^2(\mathbb{R};H)  $ from that of $|\tau|(i\tau + S^2)^{-1}$ and  $|\tau|^{\frac{1-\alpha}{2}}S^{1+\alpha}(i\tau + S^2)^{-1}$.  Hence $u\in V_\alpha$ and the estimate follows by taken infimum over all choices of $f$ and $g$. 
\end{proof}

\begin{rem}
 When $\alpha=-1$, we recover Proposition \ref{lem: Lions} (Lions'~result) : existence in $V_{-1}$, and  uniqueness follows from Proposition \ref{Unicité}. 
 Note however that uniqueness occurs 
 when $u\in V_1$ which is the largest possible space in that scale.
\end{rem}

\begin{rem}
\label{rem:counterexample}The Fourier method is rather elementary once the setup has been designed, but does not furnish time continuity: we mostly used that $\partial_t$ and $S$ commute. Something specific to time derivatives is the classical embedding theorem of Lions \cite{lions1957problemes} mentioned earlier. This embedding is not true any longer when $V=D(S)$ and $V^\star=D(S^{-1})$ are replaced with their completions $D_{S,1}$ and $D_{S,-1}$  if $I$ is bounded. Indeed, as the embedding $D_{S,1} \hookrightarrow H  $ fails,  pick $v \in D_{S,1}\setminus H$ and define the function $u(t)=v$, $0\leq t \leq 1$. We have $u \in L^2((0,1);D_{S,1})$ and $\partial_t u =0$   but $u \notin C([0,1];H)$. 

However, this counterexample is ruled out if $I=\mathbb{R}$ or $I$ unbounded and in fact, the continuity holds.  This can be obtained when $\alpha=-1$ by approximation from Lions' result but we present a different approach, which has the advantage of allowing $\alpha<0$ to conclude for regularity. 
Note however, that when $\alpha=0$, continuity cannot hold for all sources in $V_0^\star$ by the isomorphism property. We would have otherwise that any $u\in V_0$ is continuous, valued in $H$, but this is not the case. 
\end{rem}

\subsection{Solving the abstract heat equation using the Duhamel method}

Since $-S^2$ generates a $C_0$ contraction semigroup on $H$, the Duhamel formula
\begin{equation}\label{Duhamel}
    Tf(t):=\int_{-\infty}^{t} e^{-(t-s)S^2}f(s) \ \mathrm{d}s
\end{equation}
is a way of constructing  solutions to $\partial_t u +S^2 u =f$ in $\mathcal{D}'(\mathbb{R}; E_\infty)$. Remark that the adjoint Duhamel formula 
\begin{equation}\label{backwardDuhamel}
    \tilde T\tilde f(s):=\int_{s}^{\infty} e^{-(t-s)S^2}\tilde f(t) \ \mathrm{d}t
\end{equation}
is a way of constructing solutions to the backward equation
$-\partial_s \tilde u +S^2 \tilde u =\tilde f$ in $\mathcal{D}'(\mathbb{R}; E_\infty)$. All what we shall prove for the (forward) heat equation applies to the backward one. We leave to the reader the care of checking it.
For the moment, we assume $f$ to be a test function. 

\begin{lem}[A priori properties for the Duhamel solution] \label{lem:solution} If $f\in \mathcal{S}(\mathbb{R}; E_{-\infty}) $, then  $Tf $ defined by \eqref{Duhamel} belongs to $ \mathcal{S}'(\mathbb{R}; E_{\infty})$ and is a solution of $\partial_t u +S^2 u =f$ in $\mathcal{S}'(\mathbb{R}; E_{\infty})$.
    \end{lem}

    \begin{proof}  
    First using the regularity and contractivity of the semigroup,  
    \begin{equation*}
    \iint_{\mathbb{R}^2} 1_{t>s}|\langle e^{-(t-s)S^2}f(s), \varphi(t)\rangle| \ \ \mathrm{d} s \mathrm d t  \le \|f\|_{L^1(\mathbb{R};H)}  \|\varphi\|_{L^1(\mathbb{R};H)}
        \end{equation*}
        for any $f,\varphi \in \mathcal{S}(\mathbb{R};E_{-\infty})$. In particular $u$ is defined for all $t$ by a Bochner integral and belongs to $ L^\infty(\mathbb{R}; H)$.
        Hence we may apply Fubini's theorem freely, exchanging integrals and inner products in the calculation below:
      \begin{align*}
    -\int_{\mathbb{R}}  \langle u(t) , \partial_t\varphi(t)\rangle_H \ \mathrm{d} t 
    &
    = - \int_{\mathbb{R}} \int_{-\infty}^{t} \langle f(s) , e^{-(t-s)S^2}\partial_t\varphi(t)\rangle_H \ \mathrm{d} s \mathrm d t 
    \\ 
    & = - \int_{\mathbb{R}} \int_{s}^{+\infty} \langle f(s) , e^{-(t-s)S^2}\partial_t\varphi(t)\rangle_H \ \mathrm{d} t \mathrm d s  \\ & =- \int_{\mathbb{R}} \langle f(s) , \int_{s}^{+\infty}  e^{-(t-s)S^2}\partial_t\varphi(t)\rangle_H \ \mathrm{d} t \mathrm d s \\ & = - \int_{\mathbb{R}} \langle f(s) , -\varphi (s)+  \int_{s}^{+\infty} S^2e^{-(t-s)S^2} \varphi(t) \ \mathrm{d} t \rangle_H  \ \mathrm{d} s \\ & = \int_{\mathbb{R}} \langle f(s) , \varphi (s) \rangle_H  \ \mathrm{d} s -  \int_{\mathbb{R}} \int_{s}^{+\infty} \langle  e^{-(t-s)S^2}f(s), S^2 \varphi(t) \rangle_H \ \mathrm{d} t \mathrm d s. 
    \end{align*}
    Using Fubini once more, this shows that 
    $ -\llangle u, \partial_t\varphi \rrangle_{\mathcal{S}',\mathcal{S}}= \llangle f, \varphi \rrangle_{\mathcal{S}',\mathcal{S}} - \llangle u, S^2\varphi \rrangle_{\mathcal{S}',\mathcal{S}},$ 
which means $\partial_t u +S^2 u =f$ in $\mathcal{S}'(\mathbb{R};E_\infty)$. 
    \end{proof}

    We now gather a number of  a priori estimates  which are related to solving the heat equation  within $L^2(\mathbb{R};D_{S,1})$.

\begin{lem}[A priori estimates for the Duhamel operator] \label{lem:aprioriestimates} 
Let $ f \in \mathcal{S}(\mathbb{R}; E_{-\infty}) $, and define $ u = Tf $. For the inequalities involving $ \|f\|_{W_\alpha} $, we additionally assume that $ f \in \mathcal{S}_0(\mathbb{R}; E_{-\infty})$. 
    
    \begin{enumerate} \item $u\in C_0(\mathbb{R}; H)$ and one has the following uniform bounds
        \begin{align*}\sup_{t \in \mathbb{R}} \left\| u(t)\right\|_H &\le \|f\|_{L^1(\mathbb{R};H)}\\
            \sup_{t \in \mathbb{R}} \left\| u(t)\right\|_H  &\le \frac{1}{\sqrt{2}} \|f \|_{L^2(\mathbb{R}; D_{S,-1})}\\
            \sup_{t \in \mathbb{R}} \left\| u(t)\right\|_H &\le C(\alpha) \|f\|_{W_\alpha}, \ \alpha\in [-1,0).
        \end{align*} 
        \item  $u \in L^2(\mathbb{R};D_{S,1})$ and one has the following energy inequalities
        \begin{align*}\left\|u \right\|_{L^2(\mathbb{R}; D_{S,1})} &\le \frac{1}{\sqrt{2}} \|f\|_{L^1(\mathbb{R};H)}\\
            \left\|u \right\|_{L^2(\mathbb{R}; D_{S,1})}  &\le  \|f \|_{L^2(\mathbb{R}; D_{S,-1})}\\
            \left\|u \right\|_{L^2(\mathbb{R}; D_{S,1})} &\le C'(\alpha) \|f\|_{W_\alpha}, \ \alpha\in [-1,1].
        \end{align*} 
        \item  $u \in V_\alpha$ for all $\alpha\in [-1,1]$ and one has the following bound
        \begin{align*}\|D_t^{\frac{1-\alpha}{2}}S^\alpha u\|_{L^2(\mathbb{R};H)}  \le  \|f\|_{W_\alpha}.
        \end{align*} 
         
            \end{enumerate}
        
\end{lem}

\begin{proof} That $u$ belongs to $L^\infty(\mathbb{R};H)$ with $\| u\|_{L^\infty(\mathbb{R};H)} \le \|f\|_{L^1(\mathbb{R};H)}$ has been already observed above. Note that $\partial_t u= f-T(S^2f)\in  L^\infty(\mathbb{R};H)$. Thus $u$ is Lipschitz, hence continuous. The limit 0 at $-\infty$ is clear from the fact that $\|f(s)\|_H$ has rapid decay and the contraction property of the semigroup.  As for the limit at $+\infty$, we write for fixed and large $A$ and $t>A$, 
$$ u(t)= e^{-(t-A)S^2}u(A)+ \int_A^t e^{-(t-s)S^2}f(s) \ \mathrm{d}s. $$
The first term tends to 0 in $H$ by properties of the semigroup and, for the second term, one uses again the contraction property and  rapid decay of $\|f(s)\|_H$. 

We are left with proving the remaining estimates. 

\

\paragraph{\textit{Step 1:  $\left\| u(t)\right\|_H \leq \frac{1}{\sqrt{2}} \left\|f \right\|_{L^2(\mathbb{R}; D_{S,-1})}$ for all $t \in \mathbb{R}$}}
Using Cauchy-Schwarz inequality, we have for all $t \in \mathbb{R}$ and $a \in H$, 
\begin{equation*}
    \int_{-\infty}^{t} | \langle  S^{-1}f(s) ,Se^{-(t-s)S^2} a \rangle_H | \ \mathrm d s \leq \frac{1}{\sqrt{2}}\left ( \int_{-\infty}^{t} \|S^{-1}f(s) \|^2_H \ \mathrm{d}s \right )^{1/2} \left\| a\right\|_H,
\end{equation*}
where we have used the quadratic equality 
\begin{equation*}
    \int_{0}^{\infty} \|Se^{-sS^2}a \|^2_H \mathrm d s =\frac{1}{2}\left\| a\right\|_H^2.
\end{equation*}
As 
\begin{equation*}
    \langle  u(t) , a \rangle_H = \int_{-\infty}^{t} \langle  S^{-1}f(s) , Se^{-(t-s)S^2}a \rangle_H \ \mathrm d s, 
\end{equation*}
we obtain the desired bound for $\|u(t)\|_H$.

\

\paragraph{\textit{Step 2:  $\left\|u \right\|_{L^2(\mathbb{R}; D_{S,1})} \le \frac{1}{\sqrt{2}} \|f\|_{L^1(\mathbb{R};H)}$}}
We observe that by Fubini's theorem, we have $\langle Su, \tilde f\rangle=  \langle u, S\tilde f\rangle=\langle f, \tilde u\rangle$, where $\tilde u=\tilde T(S\tilde f)$. Thus
\begin{equation*}
    |\langle Su, \tilde f\rangle| 
    \le  
    \left \| f  \right \|_{L^1(\mathbb{R};H)}  \| \tilde T(S\tilde f)   \|_{L^\infty(\mathbb{R};H)}
    \le 
    \frac{1}{\sqrt{2}} \left \| f  \right \|_{L^1(\mathbb{R};H)}  \| \tilde f   \|_{L^2(\mathbb{R};H)}
\end{equation*}
using step 1 for $\tilde T$. 

\

\paragraph{\textit{Step 3:  $\left\|u \right\|_{L^2(\mathbb{R}; D_{S,1})} \le  \|f\|_{L^2(\mathbb{R};D_{S,-1})}$}} We already know from step 2 that $\left\|u \right\|_{L^2(\mathbb{R}; D_{S,1})}$ is finite. To obtain the desired bound, we use again Fubini's theorem several times and obtain 
\begin{align*}
    \left\|Su \right\|_{L^2(\mathbb{R};H)}^2&
    =\int_{\mathbb{R}} \langle  Su(t) , Su(t) \rangle_H \ \mathrm{d} t 
    \\
    &
    = \int_{\mathbb{R}} \int_{-\infty}^{t}\int_{-\infty}^{t} \langle  S^2 e^{-(t-s)S^2}S^{-1}f(s) , S^2 e^{-(t-s')S^2}S^{-1}f(s')\rangle_H \ \mathrm{d} s  \mathrm{d} s'  \mathrm{d} t  \\ 
    &
    = \int_{\mathbb{R}} \int_{\mathbb{R}}\int_{\max(s,s')}^{+\infty} \langle  S^4 e^{-(2t-(s+s'))S^2}S^{-1}f(s) ,  S^{-1}f(s')\rangle_H \ \mathrm{d} t  \mathrm{d} s  \mathrm{d} s'\\
    & = \frac{1}{2} \int_{\mathbb{R}} \int_{\mathbb{R}} \langle  S^2 e^{-(2\max(s,s')-(s+s'))S^2}S^{-1}f(s) , S^{-1}f(s')\rangle_H  \ \mathrm{d} s  \mathrm{d} s'\\
    &=  \int_{\mathbb{R}} \int_{s\leq s'} \langle  S^2 e^{-(s'-s)S^2}S^{-1}f(s) , S^{-1}f(s')\rangle_H  \ \mathrm{d} s  \mathrm{d} s'\\
    &=  \int_{\mathbb{R}} \langle  Su(s') , S^{-1}f(s')\rangle_H \ \mathrm{d} s'.
\end{align*}
Using Cauchy-Schwarz inequality, we deduce that
\begin{align*}
     \left\|Su \right\|_{L^2(\mathbb{R};H)}^2 \leq  \left\|Su \right\|_{L^2(\mathbb{R};H)}  \left\|f \right\|_{L^2(\mathbb{R};D_{S,-1})}.
\end{align*}
Therefore 
\begin{align*}
     \left\|u \right\|_{L^2(\mathbb{R};D_{S,1})} \leq   \left\|f \right\|_{L^2(\mathbb{R};D_{S,-1})}.
\end{align*}

\

\paragraph{\textit{Step 4:  $\left\| u(t)\right\|_H \le C(\alpha) \|f\|_{W_\alpha}$, $ \alpha\in [-1,0)$}}
For all $a\in E_{-\infty}$, we define 
\begin{align*}
  \forall t \in \mathbb{R};\ \   \varphi_a(t):= \mathbb{1}_{(-\infty,0]}(t) e^{tS^2}a.
\end{align*}
Remark that when $a\in H$, $\varphi_a\in L^2(\mathbb{R}, D_{S,1})$, hence $\varphi_a \in L^2(\mathbb{R}, E_{-\infty})$.
 For  $t \in \mathbb{R}$,
\begin{align*}
    \langle u(t), a \rangle_H 
    &=  \int_{-\infty}^{0} \langle e ^{sS^2}f(s-t) , a \rangle_H \  \mathrm d s
    \\ 
    &= \int_{\mathbb{R}} \langle f(s-t), \varphi_a (s) \rangle_H \ \mathrm d s 
    \\ 
    &
    = \langle \tau_t g, D_t^{\frac{1+\alpha}{2}}S^{-\alpha}\varphi_a \rangle_{L^2(\mathbb{R};H),L^2(\mathbb{R};H)}. 
\end{align*}
In the calculation, we used $\hat f(0)=0$ (see Lemma \ref{Spaces V_alpha}, point (3)) and wrote $f=D_t^{\frac{1+\alpha}{2}}S^{-\alpha}g$ with $g\in L^2(\mathbb{R}; H)$, defined $\tau_tg(s)=g(s-t)$ and  used that translations commute with $D_t$.  If we show  that 
\begin{align*}\label{Hinfini estimation}
     \| D_t^{\frac{1+\alpha}{2}} S^{-\alpha}\varphi_a  \|_{L^2(\mathbb{R}; H)} = C(\alpha) \left \| a \right \|_H,
\end{align*} then
\begin{align*}
    \left | \langle u(t), a \rangle_H \right |\leq C(\alpha) \left \| g \right \|_{L^2(\mathbb{R}; H)}\left \| a \right \|_H,
\end{align*}
and  we may conclude using the density of $E_{-\infty}$ in $H$. To see this, 
applying Fourier transform to $\varphi_a$, we get 
\begin{align*}
    \forall \tau \in \mathbb{R}; \ \ \hat{\varphi_a}(\tau)= (-i \tau+S^2)^{-1}a, 
\end{align*}
so that
\begin{equation*}
    \left | \tau  \right |^{1+\alpha} \left \| S^{-\alpha}(-{i} \tau+S^2 )^{-1}a \right \|^2_H= |\tau|^{-1} \langle \psi(|\tau|^{-1/2}S) a, a \rangle_H
\end{equation*}
with $\psi(t)= t^{-2\alpha}(1+t^4)^{-1}$.  Using simple computations and Calder\'on's identity $$\int_{-\infty}
^\infty \psi(|\tau|^{-1/2}S)a \ \frac{\mathrm{d}\tau}{|\tau|}=  \int_0^\infty \frac{ t ^{-2\alpha }}{1+t^4} \frac{\mathrm d t}{ t }\  a,$$ we obtain 
\begin{align*}
    \int_{\mathbb{R}} \left | \tau  \right |^{1+\alpha} \left \|S^{-\alpha} \hat{\varphi_a }(\tau) \right \|^2_{H} \ \mathrm{d} \tau =   \int_0^\infty \frac{ t ^{-2\alpha }}{1+t^4} \frac{\mathrm d t}{ t }\ \left \| a \right \|^2_H,
\end{align*}
and conclude using Plancherel identity that $C(\alpha)^2= \frac{1}{2\pi}\int_0^\infty \frac{ t ^{-2\alpha }}{1+t^4} \frac{\mathrm d t}{ t }$.
\

\paragraph{\textit{Step 5:  $\left\|u \right\|_{L^2(\mathbb{R}; D_{S,1})}\le C'(\alpha) \|f\|_{W_{\alpha}}$, $ \alpha\in [-1,1]$}}
Since $f\in \mathcal{S}(\mathbb{R}; E_{-\infty})$, we know a priori that $u\in L^2(\mathbb{R}; D_{S,1})$ from Step 2 and $u$ agrees with the solution given by Fourier transform of Theorem \ref{ISOMORPHISM}. Hence, we can use Fourier transform to compute. We have 
\begin{align*}
    \hat{u}(\tau) = ({i}\tau +S^2)^{-1} \left | \tau  \right |^{\frac{1+\alpha}{2} }S^{-\alpha}g(\tau).
\end{align*}
where $f=D_t^{\frac{1+\alpha}{2}}S^{-\alpha}g$ with $g\in L^2(\mathbb{R}; H)$.
Hence
\begin{align*}
    \|S\hat u(\tau)\|^2_H
&= \langle (\tau^2+S^4)^{-1} |\tau|^{1+\alpha}S^{2-2\alpha} \hat g(\tau), \hat g(\tau)
\rangle_H \\
&= \|(|\tau|^{-1/2}S)^{1-\alpha}(1+(|\tau|^{-1/2}S)^4)^{-1/2}\hat g(\tau)\|_H^2 
\\
&
\le C'(\alpha)^2 \|\hat g(\tau)\|_H
^2
\end{align*}
with $C'(\alpha)= \sup_{t>0} t^{1-\alpha}(1+t^4)^{-1/2}<\infty$ when $-1\le \alpha \le 1$.

\

\paragraph{\textit{Step 6:  $\|D_t^{\frac{1-\alpha}{2}}S^\alpha u\|_{L^2(\mathbb{R};H)}  \le  \|f\|_{W_{\alpha}}$, $ \alpha\in [-1,1]$}}

We proceed as in step 5, and compute
\begin{align*}
    \||\tau|^{\frac{1-\alpha}{2}}S^\alpha \hat u(\tau)\|^2_H
= \langle (\tau^2+S^4)^{-1} |\tau|^{2} \hat g(\tau), \hat g(\tau)
\rangle_H 
\le  \|\hat g(\tau)\|_H
^2.
\end{align*}
 The conclusion follows. 
 \end{proof}

\begin{rem}
     As noted in the proof, we can identify the Duhamel solution with the Fourier solution. So this gives an indirect proof that the Duhamel solution belongs to 
     $\mathcal{S}(\mathbb{R}; E_{-\infty})$ for $f\in \mathcal{S}(\mathbb{R}; E_{-\infty}).$
\end{rem}
\subsection{Regularity of solutions}

We can now deduce existence and uniqueness results together with regularity.
We begin with the simplest case.
\begin{thm}[Regularity for source in $L^1$]\label{Régula Cas L1}
    Let $f \in L^1(\mathbb{R};H).$ Then there exists a unique $u \in L^2(\mathbb{R};D_{S,1})$ solution of the equation $\partial_t u +S^2 u =f$ in $\mathcal{S'}(\mathbb{R}; E_\infty)$. Moreover $u \in C_0(\mathbb{R}; H)$ with 
    \begin{equation*}
        \sup_{t \in \mathbb{R}} \left\| u(t)\right\|_H  \leq  \left\|f \right\|_{L^1(\mathbb{R};H) } \qquad \mathrm{and} \qquad \left\|u \right\|_{L^2(\mathbb{R}; D_{S,1})} \leq \frac{1}{\sqrt{2}} \left\|f \right\|_{L^1(\mathbb{R};H) } .
    \end{equation*}
\end{thm}
\begin{proof}
Uniqueness in $L^2(\mathbb{R};D_{S,1})$ is provided by Proposition \ref{Unicité}.

The existence of such a regular solution with the estimates follow from Lemmas \ref{lem:solution} and \ref{lem:aprioriestimates}, when $f\in \mathcal{S}(\mathbb{R}; E_{-\infty}) $.

Density of $\mathcal{S}(\mathbb{R}; E_{-\infty})$ in $L^1(\mathbb{R};H)$ allows us to pass to the limit both in the weak formulation of the equation and in the estimates. That the limit stays in $C_0(\mathbb{R}; H)$  follows from the closedness of this space for the sup norm.
\end{proof}

We turn to the second result extending Proposition \ref{lem: Lions} (the case $\beta=-1$).

\begin{thm}[Regularity for source in $W_{-\beta}$]\label{cas Etheta}
    Let $\beta \in (0,1]$ and fix $f \in  W_{-\beta}$. Then there exists a unique $u \in L^2(\mathbb{R};D_{S,1})$ solution of $\partial_t u +S^2 u =f$ in $\mathcal{S'}(\mathbb{R}; E_\infty)$. Moreover $u \in V_{-\beta} \cap C_0(\mathbb{R}; H)$ and there exists a constant $C=C(\beta)>0$ independent of $f$ such that 
    \begin{equation*}
        \sup_{t \in \mathbb{R}} \left\| u(t)\right\|_H + \left\|u \right\|_{V_{-\beta}} \leq C \left\|f \right\|_{ W_{-\beta}}.
    \end{equation*}
\end{thm}
\begin{proof}
It is a repetition of that of Theorem \ref{Régula Cas L1}, gathering uniqueness of Proposition \ref{Unicité}, estimates of Lemma \ref{lem:aprioriestimates} with $\beta=-\alpha$, and density from Lemma \ref{Spaces V_alpha}.
\end{proof}

\begin{rem}
    For $\beta\le 0$, there is a solution in $V_{-\beta}$ by Theorem \ref{ISOMORPHISM}, but it does not belong to $C_0(\mathbb{R}; H)$.
\end{rem}

\section{Embeddings and integral identities}\label{Section 4}

The study of the abstract heat equation leads to embeddings for functions spaces in the spirit of Lions and then to integral identities expressing absolute continuity.

\subsection{Embeddings}

\begin{cor}[Extended Lions' embedding]
    For $\alpha \in [-1,0)$, we have $V_\alpha \hookrightarrow C_0(\mathbb{R};H)$.
\end{cor}
\begin{proof}
    Fix $\alpha \in [-1,0)$ and let $u \in V_\alpha$. We have $\partial_t u \in W_\alpha$ and $S^2u \in L^2(\mathbb{R};D_{S,-1})=W_{-1}$, hence $f=\partial_t u+ S^2u \in W_{-1}+W_{\alpha}= V_{-\alpha}^\star$. As $V_\alpha\subset L^2(\mathbb{R};D_{S,1})$, by Proposition \ref{Unicité},  $u$ is the unique solution in $L^2(\mathbb{R};D_{S,1})$ of the equation 
\begin{equation*}
    \partial_t \Tilde{u} +S^2 \Tilde{u} = f \ \ \mathrm{in} \ \mathcal{S}'(\mathbb{R};E_\infty).
\end{equation*} Using linearity and Theorem \ref{cas Etheta} for $\beta=-\alpha$ and $\beta=1$, we deduce that  $u \in C_0(\mathbb{R};H)$ and we have 
\begin{align*}
    \sup_{t \in \mathbb{R}} \left\| u(t)\right\|_H \leq C(\alpha) \left\|\partial_t u \right\|_{ W_\alpha}+\left ( 1+\frac{1}{\sqrt{2}} \right ) \left \| S^2u \right \|_{L^2(\mathbb{R};D_{S,-1})} \leq \Tilde{C}(\alpha) \left \| u \right \|_{V_\alpha}.
\end{align*}
\end{proof} 
 {
\begin{rem} The case $\alpha=-1$ is the homogeneous version of Lions result mentioned before.
For $\alpha \in [0,1]$, there is no chance to have an embedding $V_\alpha \hookrightarrow C_0(\mathbb{R};H)$. In fact, the embedding $\dot{H}^{1/2}(\mathbb{R};H) \hookrightarrow L^\infty(\mathbb{R};H)$ fails (case $\alpha=0$), as the scalar embedding ${H}^{1/2}(\mathbb{R}) \hookrightarrow L^\infty(\mathbb{R})$ for the classical inhomogeneous Sobolev space of order 1/2 already fails. 
\end{rem}
We complete the embeddings by exploring further the cases $\alpha \in [0,1]$ although this does not require the heat operator $\partial_t + S^2$.
\begin{lem}[Hardy-Littlewood-Sobolev embedding]\label{Hardy-Littlewood-Sobolev}
    Let $\alpha \in (0,1]$ and let $r={2}/{\alpha} \in [2,\infty)$. Then,  we have $V_{\alpha} \hookrightarrow L^{r}(\mathbb{R};D_{S,\alpha}) $ and there is a constant $C=C(r)>0$ such that for all $u \in V_{\alpha}$,  \begin{equation*} 
     \left\| u \right\|_{L^r(\mathbb{R};D_{S,\alpha})} \leq C(r) \|D_t^{\frac{1-\alpha}{2}}u \|_{L^2(\mathbb{R};D_{S,\alpha})}.\end{equation*}
    Consequently, we have $ L^{r'}(\mathbb{R};D_{S,-\alpha}) \hookrightarrow W_{-\alpha}$, where  $r'$ is the H\"older conjugate of $r$. 
\end{lem}
\begin{proof} The inequality holds for $u \in \mathcal{S}(\mathbb{R};E_{-\infty})$ using the  Sobolev embedding in $\mathbb{R}$ extended to $D_{S,\alpha}$-valued functions as the inverse of $D_t^{\frac{1-\alpha}{2}}$ is the Riesz potential with exponent $\frac{1-\alpha}{2}$.
    We conclude by density and a duality argument.
\end{proof}

The next result shows that $V_0$ and $V_1\cap L^{\infty}(\mathbb{R}; H)$ share similar embeddings.

\begin{prop}[Mixed norm embeddings]\label{prop:embed r>0} For $r\in (2,\infty)$ and $\alpha= 2/ r$, we have $V_1\cap L^{\infty}(\mathbb{R}; H)\hookrightarrow L^{r}(\mathbb{R};D_{S,\alpha})$ and $V_{0} \hookrightarrow L^{r}(\mathbb{R};D_{S,\alpha}) $,  
with 
$$   \left\| u \right\|_{L^r(\mathbb{R};D_{S,\alpha})} \leq \|u \|_{L^2(\mathbb{R};D_{S,1})}^\alpha \|u \|_{L^\infty(\mathbb{R};H)}^{1-\alpha}$$
and
 $$  \hspace{1cm} \left\| u \right\|_{L^r(\mathbb{R};D_{S,\alpha})} \leq  \|u \|_{L^2(\mathbb{R};D_{S,1})}^\alpha \|D_t^{{1}/{2}}u \|_{L^2(\mathbb{R};H)}^{1-\alpha}.$$
    Consequently, $L^{r'}(\mathbb{R};D_{S,-\alpha}) \hookrightarrow L^2(\mathbb{R}; D_{S,-1}) + L^1(\mathbb{R};H)$ and $L^{r'}(\mathbb{R};D_{S,-\alpha}) \hookrightarrow L^2(\mathbb{R}; D_{S,-1}) + W_0$.
\end{prop}

\begin{proof} 
    For the first inequality, use the moment inequality
    $$\|S^\alpha u(t)\|_H\le \|Su(t)\|_H^\alpha \|u(t)\|_H^{1-\alpha}\le \|Su(t)\|_H^\alpha \|u \|_{L^\infty(\mathbb{R};H)}^{1-\alpha} $$
    and integrate its $r$-power.
    
    For the second inequality, start with the moment inequality expressed in Fourier side when $u\in  \mathcal{S}(\mathbb{R};E_{-\infty}),$ for fixed $\tau$,
    $$\||\tau|^{\frac{1-\alpha}{2}}S^\alpha \hat u(\tau)\|_H\le \|S \hat u(\tau)\|_H^\alpha \||\tau|^{{1}/{2}}\hat u(\tau)\|_H^{1-\alpha}.$$
    Next, square the expression, integrate with respect to $\tau$, and apply H\"older inequality, Plancherel’s identity, density argument, and Lemma \ref{Hardy-Littlewood-Sobolev} to conclude.

    The consequences are standard by density and duality and we skip details.
\end{proof}

\begin{rem}
Note that the first inequality and its dual version in the statement hold whenever $\mathbb{R}$ is replaced by any interval. However, the second one and its dual version have a meaning only on  $\mathbb{R}$.
\end{rem}
\begin{rem}
Let $\alpha \in (0,1]$. Let $S^2=-\Delta_x$,  more precisely $S=(-\Delta_x)^{1/2}$, where $\Delta_x$ is the usual Laplace operator  defined as a self-adjoint operator on $L^2(\mathbb{R}^n)$.  
When $2\alpha < n$,  
Sobolev embedding in $\mathbb{R}^n$ gives us  
$$D(S^\alpha)\subset L^q(\mathbb{R}^n), \ \mathrm{with}\  \|v\|_{L^q} \le C \|S^\alpha v\|_{L^2}, \ q=\frac{2n}{n-2\alpha}.$$
 This is true for  $0\le \alpha\le 1$ if $n \geq 3$,  or $\alpha \in (0,\frac{1}{2})$ if $n=1$ or $\alpha<1$ if $n=2$.  Thus $D_{S,\alpha} \hookrightarrow L^q(\mathbb{R}^n)$. When $r=2/\alpha$, we have then $ V_\alpha \hookrightarrow L^r(\mathbb{R};L^q(\mathbb{R}^n))$. The constraints are equivalent to 
 $$ \frac{1}{r}+\frac{n}{2q}=\frac{n}{4} \ \ \mathrm{and} \ \ \ 2 \leq r,q < \infty.$$
Thus, we recover the mixed space $L^r_t L^q_x$ that appears in the classical theory \cite[chp.~3]{ladyzhenskaia1968linear} and deduce for them the classical embedding $\|u\|_{L^r_tL^q_x}\le C \|\nabla_x u\|_{L^2_tL^2_x}^{{2}/{r}} \|u\|_{L^\infty_tL^2_x}^{1-{2}/{r}} $ from the first inequality in  Proposition \ref{prop:embed r>0}. This argument  is inspired from the one in \cite{auscher2023universal}.
\end{rem}

\subsection{Integral identities}\label{S2S4} 
The Lions' embedding using domains of $S$ and $S^{-1}$ comes with integral identities. 
We now prove they hold using completions of the domains of $S$ and $S^{-1}$, and allowing more general right hand sides.   
\begin{prop}[Integral identities: the real line case]\label{energy lemma}
    Let $u \in L^2(\mathbb{R};D_{S,1})$ and let $\rho \in (2,\infty]$. Assume that $\partial_t u =  f+g$ with $f \in L^2(\mathbb{R};D_{S,-1})$ and 
     $g \in L^{\rho'}(\mathbb{R};D_{S,-{\beta}})$, where ${\beta}={2}/{\rho} \in [0,1)$ and $\rho'$ is the H\"older conjugate of $\rho$. Then $ u \in C_0(\mathbb{R};H)$, $ t \mapsto \left \| u(t) \right \|^2_H$ is absolutely continuous on $\mathbb{R}$ and for all $ \sigma < \tau $,
    \begin{align}\label{eq:integralidentity}
        \left \| u(\tau) \right \|^2_H-\left \| u(\sigma) \right \|^2_H = 2\mathrm{Re}\int_{\sigma}^{\tau} \langle f(t),u(t)\rangle_{H,-1}  + \langle g(t),u(t)\rangle_{H, -\beta}\  \mathrm d t.
    \end{align}
    {In particular, if $\rho=\infty$, then we infer that 
    \begin{align}\label{important existence cas L1}
        \sup_{t\in \mathbb{R}}\|u(t) \|_{H} \leq \sqrt{2 \left\| u \right\|_{L^2(\mathbb{R};D_{S,1})} \left\| f \right\|_{L^2(\mathbb{R};D_{S,-1})} }+(1+\sqrt{2})\left\| g \right\|_{L^1(\mathbb{R};H)}.
    \end{align}}
\end{prop}

Remark that with our notation, $\mathrm{Re} \langle f(t),u(t)\rangle_{H,-1}= \mathrm{Re}  \langle u(t),f(t)\rangle_{H,1}$.

\begin{proof}
The assumption $u \in L^2(\mathbb{R};D_{S,1})$ is equivalent to $S^2u \in L^2(\mathbb{R}; D_{S,-1})$, hence $u$ verifies the equation 
\begin{align*}
    \partial_t u+S^2u =S^2u + f+g=: h.
\end{align*}
Using Theorem \ref{cas Etheta} when $\rho<\infty$ and Theorem \ref{Régula Cas L1} when $\rho=\infty$, we know that $u\in L^2(\mathbb{R};D_{S,1})\cap C_0(\mathbb{R};H)$. It remains to prove the identity. 

Let $f_k,g_k \in \mathcal{S}(\mathbb{R};E_{-\infty})$ with $f_k \to S^2u +f$ in $L^2(\mathbb{R};D_{S,-1})$ and $g_k \to g$ in  $ L^{\rho'}(\mathbb{R};D_{S,-{\beta}})$ and set $h_k=f_k+g_k$.
Let $u_k \in L^2(\mathbb{R};D_{S,1})$ be the unique solution of the equation $\partial_t u_k + S^2 u_k = h_k$ given by Corollary \ref{cor:iso}.  We have $u_k \in \mathcal{S}(\mathbb{R};E_{-\infty})$.

The regularity of $u_k$ allows us to write for all $ \sigma < \tau $,
\begin{equation*}
    \left \| u_k (\tau) \right \|^2_H-\left \| u_k (\sigma) \right \|^2_H = 2 \mathrm{Re}\int_{\sigma}^{\tau} \langle \partial_t u_k (t),u_k (t)\rangle_H \ \mathrm d t.
\end{equation*}
Since $ \partial_t u_k=-S^2u_k+h_k$ by the equation, we have  for all $ \sigma < \tau $,
\begin{equation*}
    \left \| u_k (\tau) \right \|^2_H-\left \| u_k (\sigma) \right \|^2_H = 2 \mathrm{Re}\int_{\sigma}^{\tau} \langle f_k (t)-S^2u_k (t),u_k (t)\rangle_H + \langle g_k(t), u_k(t) \rangle_H \ \mathrm{d}t.
\end{equation*}
To pass to the limit when $k\to \infty$, we observe that  $u_k\to u$ in $L^2(\mathbb{R};D_{S,1})$ and in $C_0(\mathbb{R};H)$ in all cases, and also in $L^\rho(\mathbb{R};D_{S,\beta})$ when $\rho<\infty$. In particular $f_k-S^2u_k \to f$ in $L^2(\mathbb{R};D_{S,-1})$. We obtain \eqref{eq:integralidentity} at the limit.

In the case $\rho=\infty$, letting $\sigma\to -\infty$ and taking $\tau$ at which $\|u(\tau) \|_H=\sup\|u(t)\|_{H}=X$, we obtain 
\begin{align*}
        X^2= \|u(\tau) \|_H^2 \le 2\int_{-\infty}^{\infty}  \|S^{-1}f(t)\|_H\|Su(t)\|_{H} \  \mathrm d t+ 2X\int_{-\infty}^{\infty}  \| g(t)\|_H\  \mathrm d t .
    \end{align*}
Solving the inequality for $X$, we obtain the conclusion.
\end{proof}

We stress that the above result is false on bounded intervals as evidenced by the counter-example in Remark \ref{rem:counterexample}. But it remains valid on half-lines. On $(0,\infty)$ say, it can be shown either using the backward heat equation or an extension method. We describe the second method below.

\begin{cor}[Integral identities: the half-line case]\label{corenergy}
    Let $I$ be an open half-line of $\mathbb{R}$. Let $u \in L^2(I;D_{S,1})$ and let $\rho \in (2,\infty]$. Assume that $\partial_t u =  f+g$ with $f \in L^2(I;D_{S,-1})$ and 
     $g \in L^{\rho'}(I;D_{S,-{\beta}})$, where ${\beta}={2}/{\rho} \in [0,1)$. Then $ u \in C_0(\Bar{I},H)$, $ t \mapsto \left \| u(t) \right \|^2_H$ is absolutely continuous on $\Bar{I}$ and \eqref{eq:integralidentity} holds for all $ \sigma, \tau \in \Bar{I}$ such that $  \sigma < \tau$.
\end{cor}
\begin{proof}
We assume that $I=(0,\infty)$ because it is always possible to go back to this case. We will construct an even extension $u_e$ of $u$ and odd extensions $g_o,f_o$ of $g, f$  to  $\mathbb{R}$.   These extensions belong to the same spaces  as $u,f,g$ but in $\mathbb{R}$ and  $\partial_t u_e =  f_o+g_o$. Thus, Proposition \ref{energy lemma} applies to ${u_e}$. We obtain the conclusion by restricting to $\Bar I$.

We start by defining for all $a \in E_{-\infty}$ the distribution $\langle u, a \rangle  $ on $(0,\infty)$ by setting
\begin{align*}
    \forall \phi \in \mathcal{D}((0,\infty);\mathbb{C}), \langle \langle u, a \rangle, \phi \rangle_{\mathcal{D}',\mathcal{D}}:= \llangle u, \phi \otimes a \rrangle_{\mathcal{D}',\mathcal{D}}= \int_0^\infty \langle Su(t), S^{-1}a \rangle_H \Bar{\phi}(t) \ \mathrm{d}t. 
\end{align*}
Hence $\langle u, a \rangle$ is locally integrable and agrees with  $\langle u, a \rangle(t)=\langle Su(t), S^{-1}a \rangle_H$ almost everywhere. We have 
\begin{align*}
    \frac{\mathrm{d} }{\mathrm{d} t} \langle u,a \rangle = \langle g,a \rangle_{H,-\alpha}+ \langle f,a \rangle_{H,-1} \ \  in \ \mathcal{D}'((0,\infty);\mathbb{C}).
\end{align*}
The assumptions on $u, f,g$ imply that    $  \langle u,a \rangle \in W^{1,1}(0,T)$ for any $T>0$. It follows that $\langle u,a \rangle$ can be identified with a absolutely continuous function on $[0,\infty)$. We define ${u_e} \in \mathcal{D}'(\mathbb{R};E_\infty)$ by 
\begin{align*}
    \llangle {u_e},\phi \otimes a \rrangle_{\mathcal{D}',\mathcal{D}} := \int_{0}^{\infty} \langle u,a \rangle (t) (\Bar{\phi}(t)+\Bar{\phi}(-t)) \ \mathrm{d}t.
\end{align*}
using that distributions $\mathcal{D}'(\mathbb{R};E_{\infty})$ are uniquely determined on tensor products $\phi \otimes a$ with $\phi \in \mathcal{D}(\mathbb{R})$ and $a \in E_{-\infty}$. We have ${u_e}=u$ in $\mathcal{D}'((0,\infty);E_\infty)$ by taking $\phi$ supported in $(0,\infty).$ Next, integration by parts shows that 
\begin{align*}
    \llangle {u_e},\frac{\mathrm{d}}{\mathrm{d}t}(\phi \otimes a) \rrangle_{\mathcal{D}',\mathcal{D}}&= -\int_{0}^{\infty} (\langle S^{-\beta}g(t),S^{\beta}a \rangle_H + \langle S^{-1}f(t),Sa \rangle_H)(\Bar{\phi}(t)+\Bar{\phi}(-t)) \ \mathrm{d}t
    \\ & = - \int_{\mathbb{R}} (\langle g_o(t),a \rangle_{H,-\beta} + \langle f_o(t),a \rangle_{H,-1}) \Bar{\phi}(t) \ \mathrm{d}t,
\end{align*}
where $g_o$ and $f_o$ are the odd extensions of $g$ and $f$, respectively. Hence $\partial_t {u_e} = g_o+f_o$ in $\mathcal{D}'(\mathbb{R};E_\infty).$ Lastly,
\begin{align*}
   \llangle S{u_e},\phi \otimes a \rrangle_{\mathcal{D}',\mathcal{D}}= \llangle {u_e},S(\phi \otimes a) \rrangle_{\mathcal{D}',\mathcal{D}} &= \int_{0}^{\infty} \langle Su(t),a \rangle_H (\Bar{\phi}(t)+\Bar{\phi}(-t)) \ \mathrm{d}t 
    \\ & = \int_{\mathbb{R}} \langle (Su)_e(t),a \rangle_H \Bar{\phi}(t) \ \mathrm d t,
\end{align*}
where $(Su)_e$ is the even extension of $Su$, so that $S{u_e}=(Su)_e$ in $\mathcal{D}'(\mathbb{R};E_\infty)$. 
\end{proof}

The conclusion of Corollary \ref{corenergy} can be polarized, given two functions $u$, $\Tilde{u}$ that verify the assumptions of Corollary \ref{corenergy} with the same exponent $\rho \in (2,\infty]$ and $\beta = 2/\rho$. Thanks to the extendability seen in the previous proof, the same also works with open, half-infinite intervals and the conclusion is as follows.
\begin{cor}[Polarized integral identities]\label{EnergyPol}
    Assume that $u$, $\Tilde{u}$ satisfy the same assumptions as in Corollary \ref{corenergy} on two open infinite intervals $I$ and $J$ with non empty intersection. Then $t \mapsto \langle u(t), \Tilde{u}(t) \rangle_H$ is absolutely continuous on $\Bar{I}\cap \Bar{J}$ and  we have for all $ \sigma, \tau \in \Bar{I}\cap \Bar{J}$ such that $  \sigma < \tau$
    \begin{align*}
        \langle u(\tau), \Tilde{u}(\tau) \rangle_H - \langle u(\sigma), \Tilde{u}(\sigma) \rangle_H =\int_{\sigma}^{\tau} &\langle f(t),\Tilde{u}(t)\rangle_{H,-1} + \langle g(t),\Tilde{u}(t)\rangle_{H,-{\beta}}\ \mathrm d t \\&+ \int_{\sigma}^{\tau} \langle u(t),\Tilde{f}(t)\rangle_{H,1} + \langle u(t), \Tilde{g}(t)\rangle_{H,{\beta}}\ \mathrm d t.
    \end{align*}
\end{cor} 

\begin{rem}
We note that by linearity, the above identities hold with $g$ replaced by a sum of several terms in $L^{\rho'}(\mathbb{R};D_{S,-{\beta}})$ for different pairs $(\rho,\beta)$ and similarly for the polarized version. However, the inequality \eqref{important existence cas L1} should be modified accordingly.
\end{rem}

On a bounded interval there is a statement with  an extra  $L^1(I;H)$ hypothesis on $u$.

\begin{cor}[Integral identities: the bounded case]\label{corenergybounded}
    Let $I$ be a  bounded, open interval of $\mathbb{R}$. Let $u \in L^2(I;D_{S,1})\cap L^1(I;H)$ and let $\rho \in (2,\infty]$. Assume that $\partial_t u =  f+g$ with $f \in L^2(I;D_{S,-1})$ and 
     $g \in L^{\rho'}(I;D_{S,-{\beta}})$, where ${\beta}={2}/{\rho} \in [0,1)$. Then $ u \in C(\Bar{I},H)$, $ t \mapsto \left \| u(t) \right \|^2_H$ is absolutely continuous on $\Bar{I}$ and \eqref{eq:integralidentity} holds for all $ \sigma, \tau \in \Bar{I}$ such that $  \sigma < \tau$.
\end{cor}

\begin{proof} Assume that $I=(0,\mathfrak{T})$. 
    If we take $\chi$ a smooth real-valued function that is equal to 1 near 0 and 0 near $\mathfrak{T}$ and set $v=\chi u$ on $(0,\infty)$ then we can see  that  
     $v  \in L^2((0,\infty); D_{{S},1})$ and  
     $\partial_tv \in L^2((0,\infty); D_{{S},-1})+ L^{\rho'}((0,\infty);D_{{S},-\beta})+ L^1((0,\infty);H).$ We may apply  Corollary \ref{corenergy} with the above remark, and by restriction we have the conclusion on any subinterval $[0,\mathfrak{T}']$ with $\mathfrak{T}'<\mathfrak{T}$. If we now do  this with  a smooth real-valued function that is equal to 0 near 0 and 1 near $\mathfrak{T}$, and apply Corollary \ref{corenergy} {on $(-\infty,\mathfrak{T})$} we have by restriction the conclusion for $u$ on any subinterval $[\mathfrak{T}',\mathfrak{T}]$ with
     $0<\mathfrak{T}'$. We conclude on $[0,\mathfrak{T}]$ by gluing.
    \end{proof}

\section{Abstract parabolic equations}\label{Section 5}
In this section, we study parabolic equations of type $$ \partial_t u + \mathcal{B}u=f, $$
where $\partial_t  + \mathcal{B}$ is a parabolic operator with a time-dependent elliptic part $\mathcal{B}$ under ``divergence structure". Here, we do not assume any time-regularity on $\mathcal{B}$ apart its weak measurability. We provide a complete framework to prove well-posedness and to construct propagators and fundamental solution operators avoiding density arguments from parabolic operators with  time regular elliptic part. We also avoid time regularization like Steklov approximations. Uniqueness implies that our construction agrees with others under common hypotheses.
\subsection{Setup}\label{sec:homogenoussetup}
Throughout this section, we fix an operator $$T: D(T)\subset H \rightarrow K$$ which is injective, closed and densely defined from $ D(T)\subset H $ to another complex separable Hilbert space $K$. The operator $T^\star T$ is an injective, positive self-adjoint operator on $H$, so is $S:=(T^\star T)^{1/2}$. Moreover, by the Kato's second representation theorem \cite{kato2013perturbation}, we have 
    \begin{equation*}
      D(S)=D(T) \  \ \mathrm{and} \ \  \forall u,v \in D(T), \  \langle Su, Sv \rangle_H = \langle T^\star Tu, v\rangle_H= \langle Tu, Tv\rangle_K.
    \end{equation*}
As a result, $D_{S,1}$ is a completion of $D(T)$ for the norm $\left\|T\cdot \right\|_K$.

Next, $(B_t)_{t\in \mathbb{R}}$ is a fixed family of bounded and coercive sesquilinear forms on $D(T)\times D(T)$ {with respect to the homogeneous norm on $D(T)$} and with uniform bounds (independent of $t$). To be precise, $B_t : D(T) \times D(T) \rightarrow \mathbb{C}$ is a sesquilinear form verifying 
\begin{align}\label{EllipticityAbstract}
   \left | B_t(u,v) \right |\leq M \left \| u \right \|_{S,1}\left \| v \right \|_{S,1}\ ,\ \  \nu \left \| u \right \|_{S,1}^2 \leq \mathrm{Re} (B_t(u,u)),
\end{align}
for some $ M, \nu >0$ and for all $t \in \mathbb{R}$ and $u,v \in D(S)$. This is the equivalent to saying that for all $t \in \mathbb{R}$, there exists a bounded and strictly accretive linear map  $A(t)$ on $ \overline{\mathrm{ran}(T)}$ such that
\begin{align}\label{eq:TstarAT}
    \forall u,v \in D(T),\ B_t(u,v)=\langle A(t) Tu, Tv \rangle_K.
\end{align}
We assume in addition that the family $(B_t)_{t\in \mathbb{R}}$ is weakly measurable, \textit{i.e.}, $t \mapsto B_t(u,v)$ is a measurable function on $\mathbb{R}$, for all $u,v \in D(T)$. 

We keep denoting by $B_t$ the unique extension of $B_t$ to $D_{S,1}\times D_{S,1}$. Remark that the family $(B_t)_{t\in \mathbb{R}}$ is automatically weakly measurable for the reason that for all $u,v \in D_{S,1}$ the function $t \mapsto B_t(u,v)$ is a pointwise limit of a sequence of measurable functions. 

Note that the adjoint forms $B_t^\star$ defined by $B_t^\star(u,v)=\overline{B_t(v,u)}$ have the same properties and are associated to $A(t)^*$.

As
\begin{equation*}
       \int_\mathbb{R}\left |   B_t(u(t),v(t)) \right |\mathrm dt \leq M \left \| u \right \|_{L^2(\mathbb{R};D_{S,1})} \left \| v \right \|_{L^2(\mathbb{R};D_{S,1})},
\end{equation*}
the operator $\mathcal{B}$ defined by $$\llangle \mathcal{B}u, v \rrangle = \int_{\mathbb{R}} B_t(u(t),v(t))  \ \mathrm{d}t \ \ \mathrm{when} \  u,v \in L^2(\mathbb{R};D_{S,1})$$ is a bounded operator from $L^2(\mathbb{R};D_{S,1})$ to $L^2(\mathbb{R};D_{S,-1})$ with
\begin{equation*}
        \left \| \mathcal{B}u \right \|_{L^2(\mathbb{R};D_{S,-1})} \leq M \left \| u \right \|_{L^2(\mathbb{R};D_{S,1})}.
\end{equation*}
Next, the partial derivative is a well-defined tempered distribution given by 
   $$
   \llangle \partial_t u  , \varphi \rrangle_{\mathcal{S}',\mathcal{S}}= -\llangle  u, \partial_t\varphi \rrangle_{\mathcal{S}',\mathcal{S}}.
   $$
   When $u\in L^2(\mathbb{R};D_{S,1})$, one can compute the right hand side as 
   $$\llangle  u, \partial_t\varphi \rrangle_{\mathcal{S}',\mathcal{S}}= \llangle  Su, S^{-1}\partial_t\varphi \rrangle_{\mathcal{S}',\mathcal{S}}= \int_{\mathbb{R}}  \langle Su(t), S^{-1}\partial_t \varphi(t) \rangle_{H} \ \mathrm{d}t.
   $$

\begin{defn}[The forward parabolic operator associated to the family $(B_t)_{t \in \mathbb{R}}$]\label{Définition opérateur B}
The operator $$\partial_t + \mathcal{B}: L^2(\mathbb{R};D_{S,1}) \rightarrow \mathcal{S}'(\mathbb{R};E_{\infty})$$ defined using the weak formulation 
    \begin{equation*}
    \forall u \in L^2(\mathbb{R};D_{S,1}), \forall \varphi \in \mathcal{S}(\mathbb{R};E_{-\infty}), \  \llangle \partial_t u + \mathcal{B}u , \varphi \rrangle_{\mathcal{S}',\mathcal{S}}:=   -\llangle  u, \partial_t\varphi \rrangle_{\mathcal{S}',\mathcal{S}}+   \int_{\mathbb{R}}   B_t (u(t),\varphi(t)) \ \mathrm{d}t  
    \end{equation*}
     is called the parabolic operator associated to the family $(B_t)_{t \in \mathbb{R}}$. 
     The definition is the same as above when $\mathbb{R}$ is substituted by an open interval $I \subset \mathbb{R}$, replacing $\mathcal{S}'(\mathbb{R};E_{\infty})$ by $\mathcal{D}'(I;E_{\infty})$ and $\mathcal{S}(\mathbb{R};E_{-\infty})$ by $\mathcal{D}(I;E_{-\infty})$. In both cases, we  formally write $\partial_t  + \mathcal{B} = \partial_t  + T^\star A(t) T $.
     \end{defn}
Remark that this definition needs  no assumption $u\in L^1_{\mathrm{loc}}(\mathbb{R}; H)$ but if it is the case one can use also $\llangle  u, \partial_t\varphi \rrangle_{\mathcal{S}',\mathcal{S}}=\int_{\mathbb{R}}  \langle u(t), \partial_t \varphi(t) \rangle_{H} \ \mathrm{d}t$ (see Section \ref{subsection 4.9}). For $u,v \in \mathcal{S}(\mathbb{R};E_{-\infty})$,  an integration by parts then  yields 
$$ \llangle \partial_t u + \mathcal{B}u , v \rrangle_{\mathcal{S}',\mathcal{S}} = \overline{\llangle  -\partial_t v +\mathcal{B}^\star v , u \rrangle}_{\mathcal{S}',\mathcal{S}},  $$
where $-\partial_t + \mathcal{B}^\star$ is the backward parabolic operator associated to the adjoint family of forms $(B_t^\star)_{t \in \mathbb{R}}$ defined similarly.

We wish to find (weak) solutions $u\in L^2(\mathbb{R};D_{S,1})$ to  $\partial_tu+\mathcal{B}u= f $ for appropriate source terms. The challenge here is that we cannot use Fourier Transform anymore, nor a semi-group. We could start with Lions representation theorem but we choose a different route, introducing a variational 
parabolic operator.

We denote by $H_t$ the Hilbert transform with symbol $i \tau/ \left | \tau \right |$. More precisely, if $u \in \mathcal{S'}(\mathbb{R};E_\infty)$ is such that we have $i \tau/ \left | \tau \right | \mathcal{F}u \in \mathcal{S'}(\mathbb{R};E_\infty)$, then we set 
\begin{equation*}
    H_t u := \mathcal{F}^{-1}\left (  i \frac{\tau}{\left | \tau \right |} \mathcal{F}u \right ).
\end{equation*}

We define a bounded sesquilinear form $B_{V_0}: V_0 \times V_0 \rightarrow \mathbb{C}$ by  
\begin{equation*}\label{B_{V_0}}
    \forall u,v \in V_0, \ B_{V_0}(u,v):=\int_\mathbb{R}\langle H_t D_t^{1/2}u(t), D_t^{1/2}v(t) \rangle_H + B_t(u(t),v(t))\ \mathrm dt. 
\end{equation*}
By the Riesz representation theorem, there exists a unique $\mathcal{H} \in \mathcal{L}(V_0,V_0^\star)$ such that 
\begin{equation*}\label{eq:mathcal{H}}\llangle \mathcal{H}u, v \rrangle_{V_0^\star, V_0}:= B_{V_0}(u,v), \ u,v \in V_0.
\end{equation*}
We have 
\begin{equation*}\label{H et part + B}
   \left ( \partial_t+\mathcal{B} \right )_{\scriptscriptstyle{\vert V_0}} = \mathcal{H} \ , \ \ \left ( -\partial_t+\mathcal{B}^\star \right )_{\scriptscriptstyle{\vert V_0}} = \mathcal{H}^\star, 
\end{equation*}
where $\mathcal{H}^\star : V_0 \rightarrow V_0^\star$ is the adjoint of $\mathcal{H}$.
Indeed, we have the almost everywhere equality $$\langle H_t D_t^{1/2}u(t), D_t^{1/2}v(t) \rangle_H=-\langle u(t), \partial_t v(t) \rangle_{H}=-\langle Su(t), S^{-1}\partial_t v(t) \rangle_{H} $$
when $u,v \in \mathcal{S}(\mathbb{R};E_{-\infty})$ so that $\mathcal{H}$ and $\partial_t+\mathcal{B} $ agree on $\mathcal{S}(\mathbb{R};E_{-\infty})$ and we conclude by density in $V_0$. Thus, we may call $\mathcal{H}$ the variational parabolic operator associated to $\mathcal{B}$ as it comes from the sesquilinear form $B_{V_0}$  and $V_0$ plays the role of a variational space.

\subsection{Existence and uniqueness results}\label{S2S5} 
We now prove our main results.

\subsubsection{Source term in $V_0^\star$ : Kaplan's method}

The following lemma is essentially due to Kaplan \cite{kaplan1966abstract}. It expresses hidden coercivity of the variational parabolic  operator $\mathcal{H}$. We reproduce the argument for completeness.
\begin{lem}[Kaplan's lemma: invertibility on the pivotal variational space]\label{lemme Hidden Coerc} 
    For each $f \in V_0^\star$, there exists a unique $u \in V_0$ such that $\mathcal{H}u=f$. Moreover, 
    \begin{align*}\label{!}
         \left \| u \right \|_{V_0} \leq C(M,\nu)
         \left \| f \right \|_{V_0^\star}.
    \end{align*}
\end{lem}
\begin{proof}
By the Plancherel theorem and the fact that the Hilbert transform $H_t$ commutes with $D_t^{1/2}$ and $S$,  it is a bijective isometry on $V_0$. As  it is skew-adjoint, for all $\delta \in \mathbb{R}$,  $1+\delta H_t$ is an isomorphism on $V_0$ and $\|(1+\delta H_t)u\|_{V_0}^2= (1+\delta^2)\|u\|^2_{V_0}$. The same equality holds on $V_0^\star$.

Let $\delta>0$ to be chosen later. The modified sesquilinear form $B_{V_0}(\cdot,(1+\delta H_t)\cdot)$ is bounded on $V_0\times V_0$ and for all $u \in V_0$ 
\begin{align*}
    \mathrm{Re}\hspace{0.1cm} B_{V_0}(u, (1+\delta H_t)u) &=  \mathrm{Re}\int_\mathbb{R}\langle H_t D_t^{1/2}u, D_t^{1/2}(1+\delta H_t)u \rangle_H + B_t(u(t),(1+\delta H_t)u(t))\ \mathrm dt 
    \\
    & = \mathrm{Re}\int_\mathbb{R}  \delta \langle H_t D_t^{1/2}u(t), H_t D_t^{1/2}u(t) \rangle_H +  B_t(u(t),u(t))+\delta B_t(u(t),H_t u(t)) \ \mathrm dt. 
\end{align*}
where we have used that $H_t$ is skew-adjoint, hence 
\begin{align*}
    \mathrm{Re}\int_\mathbb{R}\langle H_t D_t^{1/2}u(t), D_t^{1/2}u(t) \rangle_H \ \mathrm dt =0.
\end{align*}
We obtain
\begin{align*}
    \mathrm{Re} (B_{V_0}(u, (1+\delta H_t)u)) \geq \delta  \| D_t^{1/2}u  \|^2_{L^2(\mathbb{R};H)}+(\nu -\delta M)\left \| u \right \|_{L^2(\mathbb{R};D_{S,1})}^2.
\end{align*}
Choosing $\delta = \frac{\nu}{1+M}$, it becomes 
\begin{align*}
    \mathrm{Re} (B_{V_0}(u, (1+\delta H_t)u)) \geq \frac{\nu}{1+M} \left \| u \right \|_{V_0}^2, \ \forall u \in V_0.
\end{align*}
Fix $f \in V_0^\star$. The Lax-Milgram lemma implies that  there exists a unique $u \in V_0$ such that $$B_{V_0}(u, (1+\delta H_t)\cdot)=  (1+\delta H_t)^\star \circ f.$$
Furthermore, we have the estimate 
\begin{align*}
    \left \| u \right \|_{V_0} \leq \frac{1+M}{\nu}\left \| (1+\delta H_t)^\star \circ f \right \|_{V_0^\star}.
\end{align*}
Using the fact that $(1+\delta H_t)^\star$ is an isomorphism on $V_0^\star$ with operator norm equal to $\sqrt{1+\delta^2}$, we have that for each $f \in V_0^\star$ there exists a unique $u \in V_0$ such that $B_{V_0}(u,\cdot)=f$ with 
\begin{align*}
    \left \| u \right \|_{V_0} \leq \frac{1+M}{\nu }\times \sqrt{ 1+\left ( \frac{\nu}{1+M} \right )^2 } \left \| f \right \|_{V_0^\star}. 
\end{align*}
\end{proof}

\noindent Now, we come to the uniqueness result below.
\begin{prop}[Uniqueness in energy space] \label{prop:uniqueness}
    Let $I$ be an interval which is a neighbourhood of $-\infty$. If $u \in L^2(I;D_{S,1})$ is a solution of $\partial_t u + \mathcal{B}u=0$ in $\mathcal{D}'(I;E_\infty)$, then $u=0$. 
\end{prop}
\begin{proof}
We have $u \in L^2(I;D_{S,1})$ and $\partial_t u = - \mathcal{B}u \in L^2(I;D_{S,-1})$. Using Corollary \ref{corenergy}, we have $u \in C_0(\overline{I};H)$ and verifies for $ \sigma, \tau \in \Bar{I}$ such that $  \sigma < \tau$, 
    \begin{align*}
        \left \| u(\tau) \right \|^2_H-\left \| u(\sigma) \right \|^2_H = - 2\mathrm{Re}\int_{\sigma}^{\tau} B_t (u(t), u(t)) \ \mathrm d t \leq 0 \ .
    \end{align*}
When $\sigma \to -\infty$, we deduce that $u(\tau)=0$, for all $\tau \in \overline{I}$.   
\end{proof}
\begin{rem}
When $I=\mathbb{R}$, one can directly prove Proposition \ref{prop:uniqueness} using uniqueness in Lemma \ref{lemme Hidden Coerc}. In fact, if $u \in L^2(\mathbb{R};D_{S,1})$ is a solution of $\partial_t u + \mathcal{B}u=0$ in $\mathcal{D}'(\mathbb{R};E_\infty)$ then $\partial_t u=-\mathcal{B}u \in L^2(\mathbb{R};D_{S,-1})$, so $u \in V_{-1} \subset V_0$ and $\mathcal{H}u=0$, therefore $u=0$ by Lemma \ref{lemme Hidden Coerc}.
\end{rem}
\subsubsection{Source term in $W_{-\beta}$, $\beta \in (0,1]$} Let us start with the following theorem.
\begin{prop}[Existence and uniqueness for $W_{-\beta}$ source] \label{prop:W-beta}
Let $\beta \in (0,1]$ and let $f \in W_{-\beta}$. Then, there exists a unique $u \in L^2(\mathbb{R};D_{S,1})$ solution to $\partial_t u +\mathcal{B} u =f$ in $\mathcal{D}'(\mathbb{R};E_\infty)$. Moreover, $u \in C_0(\mathbb{R}; H)\cap V_{-\beta}$ and there exists $C=C(M,\nu,\beta)>0$ such that 
\begin{equation*}\label{,}
       \sup_{t\in \mathbb{R}} \| u(t) \|_{H}+ \left \| u \right \|_{V_{-\beta}} \leq C\left \| f \right \|_{W_{-\beta}}.
\end{equation*}
    \end{prop}
\begin{proof}
 Since $W_{-\beta} \hookrightarrow V_{\beta}^\star \hookrightarrow V_0^\star$, Lemma \ref{lemme Hidden Coerc} provides us with a solution $u=\mathcal{H}^{-1}f \in V_0$ with the estimate  $\left \| u \right \|_{V_0} \leq C(M,\nu) \left \| f \right \|_{V_0^\star}$, so  in  particular,  $\left \| u \right \|_{L^2(\mathbb{R};D_{S,1})} \leq C(M,\nu, \beta) \left \| f \right \|_{W_{-\beta}}.$ Uniqueness in $L^2(\mathbb{R};D_{S,1})$  is provided by Proposition \ref{prop:uniqueness}.
Writing the equation as $$\partial_t u +S^2 u= S^2u -\mathcal{B}u+f \ \ \mathrm{in} \ \mathcal{D}'(\mathbb{R};E_\infty),$$ 
we may combine Theorems \ref{Régula Cas L1}, \ref{cas Etheta} together with $\mathcal{B}u \in L^2(\mathbb{R};D_{S,-1})$, to see that $u \in C_0(\mathbb{R};H)\cap V_{-\beta}$ with 
\begin{align*}
    \sup_{t\in \mathbb{R}} \| u(t) \|_{H}+\left \| u \right \|_{V_{-\beta}}&\leq C(M,\nu, \beta) \left ( \left \| u \right \|_{L^2(\mathbb{R};D_{S,1})} + \left \| f \right \|_{W_{-\beta}}  \right ).
\end{align*}
Therefore, 
$$\sup_{t\in \mathbb{R}} \| u(t) \|_{H}+ \left \| u \right \|_{V_{-\beta}} \leq C(M,\nu,\beta)\left \| f \right \|_{W_{-\beta}}.$$
\end{proof}

\begin{cor}[Boundedness properties of $\mathcal{H}^{-1}  $]\label{Cas L^2 pour L^r}
Fix $\rho \in [2,\infty)$ and set $\beta={2}/{\rho}\in (0,1]$. Then, 
$$\mathcal{H}^{-1} : L^{\rho'}(\mathbb{R};D_{S,-\beta}) \rightarrow V_{-\beta} \cap C_0(\mathbb{R};H) \ is \ bounded.$$  The same holds for $(\mathcal{H}^\star)^{-1}$.
\end{cor}

\begin{proof}
 Combine Proposition \ref{prop:W-beta}, Lemma \ref{Hardy-Littlewood-Sobolev} and Proposition \ref{prop:embed r>0}.
\end{proof}
\begin{rem}
   For fixed $f\in L^{\rho'}(\mathbb{R};D_{S,-\beta})$, we have $\mathcal{H}^{-1}f \in V_{-\beta} \subset V_0$. In particular, using Proposition \ref{prop:embed r>0}, we have $\mathcal{H}^{-1}f \in L^r(\mathbb{R};D_{S,\alpha})$ for any $ r\in (2,\infty)$ where $\alpha={2}/{r} $ and there exists a constant $C=C(M,\nu,\beta)>0$ such that 
$$ \left \| \mathcal{H}^{-1}f \right \|_{L^r(\mathbb{R};D_{S,\alpha})}\leq C\left \| f \right \|_{L^{\rho'}(\mathbb{R};D_{S,-\beta})}.$$ 
The same is true for $(\mathcal{H}^\star)^{-1}$.
\end{rem}

}

\subsubsection{Source term in $L^1(\mathbb{R};H)$}
The previous theorems rely on Lemma \ref{lemme Hidden Coerc} to prove the existence, so they do not apply anymore when $f \in L^1(\mathbb{R};H)$ since $L^1(\mathbb{R};H) \nsubseteq  V_0^\star$. Yet, we can solve with such source terms using a duality scheme.
 
\begin{prop}[Existence and uniqueness for source in $L^1$] \label{prop:L1}
    Let $f \in L^1(\mathbb{R};H)$. Then there exists a unique $u \in L^2(\mathbb{R},D_{S,1})$ solution to
   $ \partial_t u +\mathcal{B} u =f$ in $\mathcal{D'}(\mathbb{R}; E_\infty).$
    Moreover, $u \in C_0(\mathbb{R}; H)\cap L^2(\mathbb{R};D_{S,1})$ and there exists a constant $C=C(M,\nu)>0$ such that 
    \begin{equation}\label{asked estimate}
        \sup_{t\in \mathbb{R}} \| u(t) \|_{H}+\left \| u \right \|_{L^2(\mathbb{R};D_{S,1})}
        \leq C\left \| f \right \|_{L^1(\mathbb{R};H)}.
    \end{equation}
\end{prop}
\begin{proof}
Uniqueness is provided by Proposition \ref{prop:uniqueness}. To prove the existence, we remark that  
 Corollary \ref{Cas L^2 pour L^r} for the backward operator $\mathcal{H}^\star$ in the case  $\rho=2$ implies that $(\mathcal{H}^\star)^{-1}$ is bounded from $L^2(\mathbb{R};D_{S,-1})$ into $C_0(\mathbb{R}; H)$. We define 
\begin{align*}
    \mathcal{T}: L^1(\mathbb{R};H) \rightarrow \mathcal{D'}(\mathbb{R};E_\infty), \   \llangle \mathcal{T}f, \varphi  \rrangle_{\mathcal{D}',\mathcal{D}}:= \llangle f, (\mathcal{H}^\star)^{-1}\varphi  \rrangle_{L^1(\mathbb{R};H),L^\infty(\mathbb{R};H)},
\end{align*}
and we have
\begin{equation}\label{1}
   \left \| \mathcal{T}f \right \|_{L^2(\mathbb{R};D_{S,1})}
   \leq C(M,\nu) \left \| f \right \|_{L^1(\mathbb{R};H)}.
\end{equation}
 Next, let $f$ in $\mathcal{D}(\mathbb{R};E_{-\infty})$. We write for all $\varphi \in \mathcal{D}(\mathbb{R};E_{-\infty})$, observing that $\mathcal{D}(\mathbb{R};E_{-\infty})\subset V_0^\star$,
\begin{align*}
   \llangle \mathcal{T}f, \varphi  \rrangle_{\mathcal{D}',\mathcal{D}} = \llangle \mathcal{H}\mathcal{H}^{-1}f, (\mathcal{H}^\star)^{-1}\varphi  \rrangle_{V_0^\star, V_0}= \llangle \mathcal{H}^{-1}f, \varphi  \rrangle_{\mathcal{D}',\mathcal{D}}.
\end{align*}
Hence, $\mathcal{T}f=\mathcal{H}^{-1}f \in C_0(\mathbb{R}; H)$ and
\begin{equation}\label{Eq}
    \forall \varphi \in \mathcal{D}(\mathbb{R};E_{-\infty}), \ \ -\llangle \mathcal{T}f, \partial_t \varphi \rrangle_{\mathcal{D}',\mathcal{D}} + \int_{\mathbb{R}} B_t(\mathcal{T}f(t),\varphi(t))\ \mathrm d t = \int_{\mathbb{R}} \langle f(t), \varphi(t) \rangle_H \ \mathrm d t.
\end{equation} 
Using  \eqref{important existence cas L1}, we obtain
\begin{equation}\label{2}
    \sup_{t\in \mathbb{R}} \| \mathcal{T}f(t) \|_{H} \leq \sqrt{2M} \left\| \mathcal{T}f \right\|_{L^2(\mathbb{R};D_{S,1})}+(1+\sqrt{2})\left\| f \right\|_{L^1(\mathbb{R};H)} \leq C(M,\nu) \left\| f \right\|_{L^1(\mathbb{R};H)}.
\end{equation}

Now, let us pick $f \in L^1(\mathbb{R};H).$ Let $(f_k)_{k \in \mathbb{N}} \in \mathcal{D}(\mathbb{R};E_{-\infty})^{\mathbb{N}}$ such that $ f_k \to f$ in $L^1(\mathbb{R};H)$. By \eqref{1} and \eqref{2}, we have $\mathcal{T}f_k \to \mathcal{T}f$ in $L^2(\mathbb{R};D_{S,1}) \cap C_0(\mathbb{R}; H)$.
Using \eqref{Eq} with $\mathcal{T}f_k$ for a fixed $\varphi \in \mathcal{D}(\mathbb{R};E_{-\infty})$ and letting $k \to \infty$ imply that $\partial_t (\mathcal{T}f) + \mathcal{B}(\mathcal{T}f) = f$ in $\mathcal{D}'(\mathbb{R};E_{\infty}).$
\end{proof}

\subsubsection{Source term is a bounded measure on $H$}  
First, we define the space of bounded $H$-valued measures on $\mathbb{R}$, denoted $\mathcal{M}(\mathbb{R};H)$, as the topological anti-dual space of $C_0(\mathbb{R};H)$ with respect to the sup-norm. We denote by $\llangle \cdot, \cdot \rrangle_{\mathcal{M},C_0}$ the anti-duality bracket. We equip the space $\mathcal{M}(\mathbb{R};H)$ with the operator norm, that is 
\begin{equation*}
    \left \|\mu   \right \|_{\mathcal{M}}:= \sup_{\varphi \in C_0(\mathbb{R};H)\setminus\left \{ 0 \right \}} \frac{\left | \llangle \mu, \varphi \rrangle_{\mathcal{M},C_0} \right |}{\ \ \ \ \left \|\varphi  \right \|_{L^\infty(\mathbb{R};H)}}.
\end{equation*}
It is a Banach space containing a subspace isometric to $L^1(\mathbb{R};H)$.

For $\mu \in \mathcal{M}(\mathbb{R};H)$, $\mathrm{supp}(\mu)$ is the complement of the largest open set of $\mathbb{R}$ on which $\mu$ is equal to $0$. More precisely, we say that $\mu$ equals $0$ on an open set $\Omega \subset \mathbb{R}$ if for all $\phi \in C_0(\mathbb{R},H)$ with support contained in $\Omega$, $\llangle \mu, \phi \rrangle_{\mathcal{M},C_0}=0$. Let $\mathcal{N}$ denotes the set of all such open sets. We have 
\begin{align*}
    \mathrm{supp}(\mu ):= \mathbb{R}\setminus \bigcup_{\Omega \in \mathcal{N}} \Omega.
\end{align*}
An important example is the class of Dirac measures.  For any $s \in \mathbb{R}$ and $a \in H$, we denote by $\delta_s \otimes a$ the Dirac measure on $s$ carried by $a$ which is defined by 
\begin{equation*}
    \llangle \delta_s \otimes a,  \phi \rrangle_{\mathcal{M},C_0} =  \langle a,  \varphi(s) \rangle_H, \ \varphi \in C_0(\mathbb{R};H).
\end{equation*}

 We first state the classical lemma below for later use.
\begin{lem}\label{Radon}
    Let $\mu \in \mathcal{M}(\mathbb{R};H).$ Then there exists a sequence $(f_\varepsilon)_{\varepsilon>0}$ in $L^1(\mathbb{R};H)$ such that 
    $f_\varepsilon \rightharpoonup \mu$  (weak-$\star$  convergence)  and  $\sup_{\varepsilon >0} \left \| f_\varepsilon  \right \|_{L^1(\mathbb{R};H)} \leq \left \| \mu  \right \|_{\mathcal{M}}.$
   \end{lem}
\begin{proof}
We obtain the sequence $(f_\varepsilon)_{\varepsilon >0}$ by convoluting $\mu$ with a scalar mollifying sequence $(\varphi_\varepsilon)_{\varepsilon >0}$ and we easily check  that we have all the required properties.
\end{proof} 
\begin{prop}[Existence and uniqueness for bounded measure source] \label{CasRadon}
    Let $\mu \in \mathcal{M}(\mathbb{R};H)$. Then there exists a unique $u \in L^2(\mathbb{R};D_{S,1})$ solution to
       $ \partial_t u +\mathcal{B} u =\mu$ in $\mathcal{D'}(\mathbb{R}; E_\infty).$ Moreover, $u \in L^\infty(\mathbb{R};H)$ f and there is a constant $C=C(M,\nu)>0$ such that 
    \begin{equation}\label{ESTT}
     \left \| u \right \|_{L^2(\mathbb{R};D_{S,1})}+\left \| u \right \|_{L^\infty(\mathbb{R};H)} \leq C(M,\nu) \left \| \mu \right \|_{\mathcal{M}}.
    \end{equation}
    If $I \subset \mathbb{R}\setminus \mathrm{supp}(\mu)$ an unbounded open interval, then $u \in C_0(\overline{I},H)$ and $t \mapsto \left \| u(t) \right \|_H^2$ is absolutely continuous on $\overline{I}$. Moreover, if $I$ is a neighbourhood of $-\infty$, then $u=0$ on $\overline{I}$.
\end{prop}
\begin{proof}
    Uniqueness is provided by Proposition \ref{prop:uniqueness}. To prove the existence, we use lemma \ref{Radon} to pick  $\left ( f_n \right )_{n \in \mathbb{N}}\in L^1(\mathbb{R};H)^\mathbb{N}$ such that 
    $f_n  \rightharpoonup \mu$   and   $\sup_{n\in \mathbb{N}} \left \| f_n \right \|_{L^1(\mathbb{R};H)} \leq \left \| \mu \right \|_{\mathcal{M}}.$
By Proposition \ref{prop:L1}, for all $n \in \mathbb{N}$, there is a unique $u_n \in L^2(\mathbb{R};D_{S,1})$ solution of the equation $\partial_t u_n+\mathcal{B}u_n= f_n$ in $\mathcal{D}'(\mathbb{R};E_\infty)$ and we have \eqref{asked estimate}, implying
\begin{equation}\label{EST}
    \left \| u_n \right \|_{L^2(\mathbb{R};D_{S,1})}+\left \| u_n \right \|_{L^\infty(\mathbb{R};H)} \leq C(M,\nu) \left \| \mu \right \|_{\mathcal{M}}.
\end{equation}
Using the Banach-Alaoglu theorem, there exists $u \in L^2(\mathbb{R};D_{S,1}) \cap L^\infty(\mathbb{R};H)$ such that, up to extracting a sub-sequence, $u_n \rightharpoonup u$ weakly in $L^2(\mathbb{R};D_{S,1})$ and weakly-$\star$ in $L^\infty(\mathbb{R};H)$ for the duality pairing $L^\infty$-$L^1$. We have \eqref{ESTT} and  easily  pass to the limit in the equation to obtain the desired solution.

Now, if $I \subset \mathbb{R}$ is an unbounded open interval such that $ I \cap \mathrm{supp}(\mu)= \varnothing$, then by Corollary \ref{corenergy}, $u \in C_0(\overline{I},H)$ and $t \mapsto \left \| u(t) \right \|_H^2$ is absolutely continuous on $\overline{I}$. If $I$ is a neighbourhood of $-\infty$, then $u=0$ by Proposition \ref{prop:uniqueness} on $I$. 
 \end{proof}

The next corollary is crucial to construct fundamental solution and Green operators.
\begin{cor}[Existence and uniqueness for Dirac measure source]\label{CorRadon}
Let $s \in \mathbb{R}$ and $a \in H$. Then there exists a unique $u \in L^2(\mathbb{R};D_{S,1})$ solution to
   $ \partial_t u + \mathcal{B}u = \delta_s \otimes a \ \ in \ \mathcal{D}'(\mathbb{R};E_{\infty}).$
Moreover, $u\in C_0(\mathbb{R}\setminus  \{ s  \}; H)$, equals  $0$ on $(-\infty,s)$ and $\lim_{t \to s^+} u(t)=a$ in $H$,  and there is a constant $C=C(M,\nu)>0$ such that 
    \begin{equation*}
        \sup_{t\in \mathbb{R}\setminus\{s\}}\left \|u(t)  \right \|_H + \left \| u \right \|_{L^2(\mathbb{R};D_{S,1})}
        \leq C \left \| a \right \|_H.
    \end{equation*}
Furthermore, {$u_{\scriptscriptstyle{\vert (s,\infty)}}$ is the restriction of an element in $V_{-1}$}.
\end{cor}
\begin{proof} Applying the previous result Proposition \ref{CasRadon}, we have existence and uniqueness of $u$ with the estimates and  $u \in C_0(\mathbb{R}\setminus  \{ s  \};H)$, equals  $0$ on $(-\infty,s)$ and has a limit $u(s^{+})$ when  $t \to s^+$. 

 That $u_{\scriptscriptstyle{\vert (s,\infty)}}$ is a restriction to $(s,\infty)$ of an element in $V_{-1}$ follows from the fact that $\partial_t u=f$ {in $\mathcal{D}'((s,\infty);E_\infty)$} with $f= -\mathcal{B}u$ on $(s,\infty)$ and the method of Corollary \ref{corenergy}, by taking the even extension of $u$ and the odd extension of $f$ with respect to $s$. 
 
It remains to show $u(s^{+})=a$. Let $\Tilde{a}\in E_{-\infty}$ and $\theta\in \mathcal{D}(\mathbb{R})$ with $\theta(s)=1$ and set $\varphi=\theta\otimes \Tilde{a}$. Using the absolute continuity of $t\mapsto \langle u(t), \varphi(t)\rangle_H$ on both $[s,\infty)$ and $(-\infty,s]$ using Corollary \ref{EnergyPol}, we have
\begin{align*}
    - \langle u(s^+), \Tilde{a}\rangle_H &= \int_s^\infty - B_t(u(t),\varphi(t)) + \langle u(t), \partial_t\varphi(t)\rangle_H \  \mathrm d t
    \\
    + \langle u(s^-), \Tilde{a}\rangle_H &= \int_{-\infty}^s - B_t(u(t),\varphi(t)) + \langle u(t), \partial_t\varphi(t)\rangle_H \  \mathrm d t
    \end{align*}
    By the equation for $u$ on $\mathbb{R}$, and since $u\in L^1_{\mathrm{loc}}(\mathbb{R}; H)$,   we obtain
    \begin{equation*}
    \int_{-\infty}^\infty  B_t(u(t),\varphi(t)) - \langle u(t), \partial_t\varphi(t)\rangle_H \  \mathrm d t = \llangle \delta_s \otimes a,  \varphi \rrangle_{\mathcal{M},C_0}= 
    \langle a, \Tilde{a}\rangle_H.
\end{equation*}
Summing up,  $- \langle u(s^+), \Tilde{a}\rangle_H+ \langle u(s^-), \Tilde{a}\rangle_H= - \langle a, \Tilde{a}\rangle_H$. As $u(s^-)=0$, this yields 
$\langle u(s^+), \Tilde{a}\rangle_H= \langle a, \Tilde{a}\rangle_H$ and we conclude by density of $E_{-\infty}$ in $H$.
\end{proof}
\subsection{Green operators}\label{S2S6}
The notion below of Green operators was first introduced by J.-L. Lions \cite{lions2013equations}.
\begin{defn}[Green operators]
Let $t,s \in \mathbb{R}$ and $a,\Tilde{a} \in H$.
\begin{enumerate}
    \item For $t \neq s$, $G(t,s)a$ is defined as the value at time $t$ of the solution $u \in L^2(\mathbb{R};D_{S,1})$ of the equation $\partial_t u + \mathcal{B}u = \delta_s \otimes a$ in Corollary \ref{CorRadon}. 
    \item For $s \neq t$, $\Tilde{G}(s,t)\Tilde{a}$ is defined as the value at time $s$ of the solution $u \in L^2(\mathbb{R};D_{S,1})$ of the equation $-\partial_s u + \mathcal{B^\star}u = \delta_t \otimes \Tilde{a}$ in Corollary \ref{CorRadon}. 
\end{enumerate}
The operators $G(t,s)$ and $\Tilde{G}(s,t)$ are called the Green operators for the parabolic operator $\partial_t +\mathcal{B}$ and the backward parabolic operator $ -\partial_t +\mathcal{B^\star}$, respectively.
\end{defn}
 The properties discussed in the last section can be summarized in the following corollary.
\begin{cor}[Estimates for Green operators]\label{Cor Green}  There is a constant $C=C(M,\nu)>0$ such that one has the following statements.
\begin{enumerate}
    \item For all $t<s \in \mathbb{R}$, $G(t,s)=0$ and for all $a \in H$, $t \mapsto G(t,s)a \in C_0([s,\infty);H)$ with $ G(s,s)a =a $ {and it is a restriction to $(s,\infty)$ of an element in $V_{-1}$}, and for any $r \in [2,\infty)$, $a \in H$ and $s\in \mathbb{R}$, we have $G(\cdot,s)a \in L^r((s,\infty);D_{S,\alpha})$ where $\alpha={2}/{r}$ with
    \begin{equation*}
        \sup_{t \geq s} \left \| G(t,s)  \right \|_{\mathcal{L}(H)} \leq C \ \ \mathrm{and} \ \  \int_{s}^{+\infty} \left \| G(t,s)a \right \|^r_{S,\alpha}\mathrm d t \leq C^r \left \| a \right \|_H^r.
    \end{equation*}
    \item For all $s>t \in \mathbb{R}$, $\Tilde{G}(s,t)=0$ and for all $\Tilde{a} \in H$, $s \mapsto \Tilde{G}(s,t)\Tilde{a} \in C_0((-\infty,t];H)$ with $\Tilde{G}(t,t)\Tilde{a} =\Tilde{a}$ {and it is a restriction to $(-\infty,t)$ of an element in $V_{-1}$}, and for any $r \in [2,\infty)$,  $t\in \mathbb{R}$, we have $\Tilde{G}(s,\cdot)a \in L^r((-\infty,t);D_{S,\alpha})$ where $\alpha={2}/{r}$  with 
    \begin{equation*}
       \sup_{t \geq s}  \| \Tilde{G}(s,t)   \|_{\mathcal{L}(H)} \leq C \ \ \mathrm{and} \ \ \int_{-\infty}^{t}  \| \Tilde{G}(s,t)\Tilde{a}  \|^r_{S,\alpha}\mathrm d s \leq C^r \| \Tilde{a}  \|_H^r.
    \end{equation*}
\end{enumerate}
\end{cor}

\begin{proof}
The properties follows from construction in Corollary \ref{CorRadon} and the interpolation inequalities in Proposition \ref{prop:embed r>0} for $2<r<\infty$.
\end{proof}

 Moreover, expected adjointness and Chapman-Kolmogorov relations  hold.
\begin{prop}[Adjointess and Chapman-Kolmogorov identities]\label{Prop Green} The following statements hold.
\begin{enumerate}
    \item For all $s <t $, $G(t,s)$ and $\Tilde{G}(s,t)$ are  adjoint operators.
    \item  For any $s < r <t $, we have $G(t,s)=G(t,r)G(r,s)$.
\end{enumerate}
 
\end{prop}
\begin{proof}
We first prove point (1). We fix $t,s \in \mathbb{R}$ such that $s<t$. For $a,\Tilde{a} \in H$, we can apply the integral identity of Corollary \ref{EnergyPol} to $u:= G(\cdot,s)a$ and $v = \Tilde{G}(\cdot,t)\Tilde{a}$ between $s$ and $t$. Note that by duality the integrand vanishes almost everywhere, hence
\begin{equation*}
 \langle G(t,s)a, \Tilde{a}\rangle_H=    \langle G(t,s)a, \Tilde{G}(t,t)\Tilde{a}\rangle_H =   \langle G(s,s)a, \Tilde{G}(s,t)\Tilde{a}\rangle_H= \langle a, \Tilde{G}(s,t)\Tilde{a}\rangle_H.
    \end{equation*}
    The adjunction property follows. For point (2), we apply the same equality between $r$ and $t$ and use that the adjoint of $G(t,r)$ is $\Tilde{G}(r,t)$ from point (1), to obtain
    \begin{equation*}
 \langle G(t,s)a, \Tilde{a}\rangle_H=    \langle G(t,s)a, \Tilde{G}(t,t)\Tilde{a}\rangle_H =  \langle G(r,s)a, \Tilde{G}(r,t)\Tilde{a}\rangle_H= \langle G(t,r)G(r,s) a, \Tilde{a}\rangle_H.
    \end{equation*}
    
\end{proof}

\subsection{Fundamental solution}\label{sec:FS} We define the fundamental solution as representing
the inverse of $\partial_t + \mathcal{B}$.

\begin{defn}[Fundamental solution for $\partial_t + \mathcal{B}$] \label{FS}
    A fundamental solution for $\partial_t + \mathcal{B}$ is a family $\Gamma=(\Gamma(t,s))_{t,s \in \mathbb{R}}$ such that : 
    \begin{enumerate}
        \item $\sup_{t,s \in \mathbb{R}} \left\|\Gamma(t,s) \right\|_{\mathcal{L}(H)} < +\infty.$
        \item $\Gamma(t,s)=0$ if $s>t$.
        \item For all $a,\Tilde{a}\in E_{-\infty}$, the function $(t,s) \mapsto \langle \Gamma(t,s)a, \Tilde{a} \rangle_H$ is Borel measurable on $\mathbb{R}^2$.
        \item For all $\phi \in \mathcal{D}(\mathbb{R})$ and $a \in E_{-\infty}$, the solution $u \in L^2(\mathbb{R};D_{S,1})$ of the equation $\partial_t u + \mathcal{B}u = \phi \otimes a $ in $\mathcal{D}'(\mathbb{R};E_\infty)$ satisfies for  all $\Tilde{a} \in E_{-\infty}$,  $\langle u(t), \Tilde{a} \rangle_H = \int_{-\infty}^{t} \phi(s) \langle \Gamma (t,s)a, \Tilde{a} \rangle_H \ \mathrm{d}s,$ for almost every $t\in \mathbb{R}$.
    \end{enumerate}
    One defines a fundamental solution $(\Tilde{\Gamma}(s,t))_{s,t \in \mathbb{R}}$ to the backward operator $-\partial_s + \mathcal{B}^\star$ analogously and (ii) is replaced by $\Tilde{\Gamma}(s,t)=0$ if $s>t$.
\end{defn}
\begin{lem}[Uniqueness of fundamental solutions]\label{Unicité sol fonda} There is at most one fundamental solution to $\partial_t + \mathcal{B}$ in the sense of Definition \ref{FS}.
\end{lem}
\begin{proof}
     Assume $\Gamma_1$, $\Gamma_2$ are two fundamental solutions to $\partial_t + \mathcal{B}$. Fix $a$ and $\Tilde{a}$ in $E_{-\infty}$. The function $(t,s) \mapsto \langle \Gamma_k(t,s)a, \Tilde{a} \rangle_H$ is bounded and measurable for $k \in \left\{1,2 \right\}$ by (1) and (3), hence Fubini's Theorem with (2) and (4) yield  for all $\phi, \Tilde{\phi} \in \mathcal{D}(\mathbb{R})$,
\begin{equation*}
    \iint_{\mathbb{R}^2} \Tilde{\phi}(s) {\phi}(t) \langle \Gamma_1 (t,s)a, \Tilde{a} \rangle_H \ \mathrm{d}s\mathrm{d}t = \iint_{\mathbb{R}^2} \Tilde{\phi}(s) \phi(t) \langle \Gamma_2 (t,s)a, \Tilde{a} \rangle_H \ \mathrm{d}s\mathrm{d}t.
\end{equation*}
 Therefore, we obtain $\langle \Gamma_1 (t,s)a, \Tilde{a} \rangle_H = \langle \Gamma_2 (t,s)a, \Tilde{a} \rangle_H$ for almost every $(t,s) \in \mathbb{R}^2$. At this stage, the almost everywhere equality can depend on $a$ and $\Tilde{a}$. Applying this for test elements $a,\Tilde{a} \in E_{-\infty}$ describing a countable set in $E_{-\infty}$ that is dense in $H$ and using that the  operators $\Gamma_1(t,s)$, $\Gamma_2(t,s)$ are bounded on $H$ by (1), one deduces that $\Gamma_1$ and $\Gamma_2$ agree almost everywhere.
\end{proof}

\subsection{The Green operators are {the} fundamental solution operators} The two notions are well defined and we show that they lead to the same families. We borrow partially ideas from \cite{auscher2023universal}.
\begin{thm}[Green operators and fundamental solution operators agree]\label{thm:identification}
    The family of Green operators is the fundamental solution (up to almost everywhere equality) and (4) holds for all $t\in \mathbb{R}$.
\end{thm}
\begin{proof}
As there is at most one fundamental solution, it suffices to show that the family of Green operators satisfies the requirements (1)-(4) in Definition \ref{FS}.

The Green operators verify (1) and (2) of the Definition \ref{FS} by Corollary \ref{Cor Green}. For the measurability issue (3), remark that for all $a,\Tilde{a} \in E_{-\infty}$, we have $(t,s)\mapsto \langle G(t,s)a, \Tilde{a} \rangle_H$ is separately continuous on $\mathbb{R}^2\setminus \left \{ (t,t), t \in \mathbb{R} \right \}$, so Borel measurable on $\mathbb{R}^2$. We only have to prove (4), namely that for any $\phi \in \mathcal{D}(\mathbb{R})$ and $a,\Tilde{a}\in E_{-\infty},$  if $u$ is the weak solution for the source term $\phi\otimes a$, we have  for  all (not just almost all) $t \in \mathbb{R}$,
\begin{equation}\label{GreenLemma}
        \left\langle u(t), \Tilde{a} \right \rangle_H = \int_{-\infty}^{t}\phi(s)  \left\langle G(t,s)a, \Tilde{a} \right \rangle_H \ \mathrm{d}s.
\end{equation}
Fix $t\in \mathbb{R}$. Introduce $\Tilde{u}=\Tilde{G}(\cdot, t)\Tilde{a}$. Using the absolute continuity of $s\mapsto \langle u(s),\Tilde{u}(s)\rangle_H$ on $(-\infty, t]$ with its zero limit at $-\infty$, and seing $\phi\otimes a \in L^1(\mathbb{R};H)$, we have by Corollary \ref{EnergyPol}  
\begin{equation*}
    \langle u(t),\Tilde{a}\rangle_H= \langle u(t),\Tilde{u}(t)\rangle_H= \int_{-\infty}^t \langle \phi(s) a, \Tilde{u}(s)\rangle_H \ \mathrm ds.
\end{equation*}
For $s\le t$,  using Proposition \ref{Prop Green} in the last equality below
\begin{equation*}
      \langle \phi(s) a, \Tilde{u}(s)\rangle_H =  \phi(s)\langle  a, \Tilde{G}(s,t)\Tilde{a}\rangle_H=  \phi(s)\langle  G(t,s)a, \Tilde{a}\rangle_H
\end{equation*}
and we are done.
\end{proof}
\subsection{Representation with the fundamental solution operators} Having identified Green operators to fundamental solution operators, the latter inherits the properties of the former. From now on, we use the more traditional notation $\Gamma(t,s)$. We can now state a complete representation theorem
for all the distributional solutions seen in the last subsections, with specified convergence issues.

\begin{thm}\label{ThmGreenReprese}
Let $s\in \mathbb{R}$, $a \in H$, $g \in L^{\rho'}(\mathbb{R};H)$, where $\rho \in [2,\infty]$ and $\beta={2}/{\rho}$. Then the unique solution $u \in L^2(\mathbb{R};D_{S,1})$ of the equation 
    \begin{equation*}
        \partial_t u +\mathcal{B}u= \delta_s \otimes a + S^\beta g  \ \ \mathrm{in} \ \mathcal{D}'(\mathbb{R};E_\infty)
    \end{equation*}
obtained by  combining Propositions  \ref{prop:W-beta}, \ref{prop:L1} and Corollary \ref{CorRadon} can be represented pointwisely by the equation 
    \begin{equation*}
        u(t) = \Gamma(t,s)a 
        + \int_{-\infty}^{t} \Gamma(t,\tau) S^\beta g(\tau)\ \mathrm d \tau,
    \end{equation*}
    where the integral is weakly defined in $H$ when $\rho<\infty$ and strongly defined when $\rho=\infty$ (i.e.,  in the Bochner sense). More precisely, for all $\Tilde{a}\in H$, we have the equality with absolutely converging integral
    \begin{equation}\label{WeaksensGreen}
        \langle u(t), \Tilde{a} \rangle_H = \langle \Gamma(t,s)a, \Tilde{a} \rangle_H 
        + \int_{-\infty}^{t} \langle  g(\tau),S^\beta\Tilde{\Gamma}(\tau,t) \Tilde{a} \rangle_{H} \ \mathrm d \tau.
    \end{equation}
    
\end{thm}

\begin{rem}
    Remark that by Proposition \ref{prop:embed r>0}, one could even reduce to proving the result when $\rho=2$ and $\rho=\infty$.
\end{rem}
\begin{proof}
It is enough to prove \eqref{WeaksensGreen}. By uniqueness and linearity, we start by writing $u=u_1+u_2$ where $u_k$ is the solution of the equation considering only the $k^{th}$ term in the right-hand side of the equation. Recall that we have identification of $G$ and $\Gamma$. Fix $t \in \mathbb{R}.$

The first term involving $a$ is  $u_1(t)=\Gamma(t,s)a$ by construction and identification.

The argument for $u_2$ is as follows.  
According to the proof of Theorem \ref{Unicité sol fonda} the weak solution $v$ obtained from source $\phi\otimes c$, where $\phi\in \mathcal{D}(\mathbb{R})$ and $c \in E_{-\infty}$, satisfies for any 
$\Tilde{a} \in E_{-\infty}$ and $t\in \mathbb{R}$, 
\begin{align*}
    \langle v(t), \Tilde{a} \rangle_H &= \int_{-\infty}^{t}\phi(\tau)  \langle \Gamma(t,\tau)c, \Tilde{a} \rangle_H \ \mathrm{d}\tau
    \\&= \int_{-\infty}^{t} \langle (\phi\otimes c)(\tau), \Tilde{\Gamma}(\tau,t)\Tilde{a}  \rangle_H \ \mathrm{d}\tau.
    \end{align*}
For any $\rho\in [2,\infty]$, we have  $$ \int_{-\infty}^{t} |\langle h(\tau),\Tilde{\Gamma}(\tau,t) \Tilde{a}  \rangle_{H,-\beta}| \ \mathrm{d}\tau 
\le C \|h\|_{L^{\rho'}(\mathbb{R};D_{S,-\beta})} \|\Tilde{a}\|_H$$
 by using Cauchy-Schwarz inequality in $H$ and  H\"older inequality invoking   estimates for  $\Tilde{G}$ in Proposition \ref{Prop Green}.
 Writing $\langle (\phi\otimes c)(\tau), \Tilde{\Gamma}(\tau,t)\Tilde{a}  \rangle_H$ as $\langle (\phi\otimes c)(\tau), \Tilde{\Gamma}(\tau,t)\Tilde{a}  \rangle_{H,-\beta}$, we can conclude for $u_2$ 
 by density  of the span of tensor products $\phi\otimes c$  in $L^{\rho'}(\mathbb{R};D_{S,-\beta})$, 
 and density of $E_{-\infty}$ in $H$.   
For $\rho=\infty$, we may also verify the strong convergence.
\end{proof}
We record the following operator-valued Schwartz kernel result.
\begin{prop}
    Let $\phi,\Tilde{\phi}\in \mathcal{D}(\mathbb{R})$ and $a,\Tilde{a}\in H$. Then,
    \begin{equation*}
        \llangle \mathcal{H}^{-1}(\phi  \otimes a), \Tilde{\phi}\otimes \Tilde{a} \rrangle_{\mathcal{D}', \mathcal{D}} = \iint \phi(s) \langle \Gamma(t,s)a,  \Tilde{a} \rangle_H \overline{\Tilde{\phi}}(t) \ \mathrm d s\mathrm d t.
    \end{equation*}
    In other words, one can see $\langle \Gamma(t,s)a,  \Tilde{a} \rangle_H$ as the Schwartz kernel of the sesquilinear map $(\phi,\Tilde{\phi})\mapsto \langle \mathcal{H}^{-1}(\phi  \otimes a), \Tilde{\phi}\otimes \Tilde{a} \rangle_{\mathcal{D}', \mathcal{D}}$ on $\mathcal{D}(\mathbb{R})\times \mathcal{D}(\mathbb{R}).$
\end{prop}
\begin{proof}
By density of $E_{-\infty}$ in $H$ and boundedness of the Green operators, we may use \eqref{GreenLemma} for $a,\Tilde{a}\in H$ and we obtain,
\begin{align*}
     \llangle \mathcal{H}^{-1}(\phi  \otimes a), \Tilde{\phi}\otimes \Tilde{a} \rrangle_{\mathcal{D}', \mathcal{D}} &= \int_{\mathbb{R}} \langle \mathcal{H}^{-1}(\phi  \otimes a)(t), \Tilde{a} \rangle_H \overline{\Tilde{\phi}}(t)\ \mathrm{d}t
     \\& = \int_{\mathbb{R}} \left ( \int_{-\infty}^{t} \phi(s) \langle \Gamma(t,s)a,\Tilde{a}\rangle_H \ \mathrm ds \right ) \overline{\Tilde{\phi}}(t)\ \mathrm{d}t
     \\& = \iint \phi(s) \langle \Gamma(t,s)a,  \Tilde{a} \rangle_H \overline{\Tilde{\phi}}(t) \ \mathrm d s\mathrm d t,
\end{align*}
where we have used Fubini's theorem  and $\Gamma(s,t)=0$ for $s>t$ in the last line. 
\end{proof}
\subsection{The Cauchy problem and the fundamental solution}
In this section, we consider the Cauchy problem on the interval $(0,\infty)$. The coefficients $B_t$ are defined on $(0,\infty)$ and satisfy \eqref{EllipticityAbstract} there.  We fix $\rho \in [2,\infty]$ and  set $\beta={2}/{\rho}$. The Cauchy problem  with initial condition $a\in H$ and  
$g\in L^{\rho'}((0,\infty);H)$ consists in finding a weak solution to
\begin{align}\label{Pb Cauchy homogène}
\left\{
    \begin{array}{ll}
         \partial_t u +\mathcal{B}u = S^\beta g   \ \mathrm{in} \ \mathcal{D}'((0,\infty);E_\infty), \\
        u(0)=a \  \mathrm{in} \ E_\infty.
    \end{array}
\right.
\end{align}

\begin{rem}
Note that when $f \in L^2((0,\infty);K)$, there exists $g\in L^2((0,\infty); H)$ such that $T^*f=Sg$, hence the  $\beta=1$ case  covers the classical Lions equation.
\end{rem}
\begin{defn}
     A weak solution to \eqref{Pb Cauchy homogène} is a function 
    $
        u \in L^2((0,\infty); D_{S,1}),
    $
    with 
    \begin{enumerate}
        \item [(i)]  $u$ solves the the first equation in $\mathcal{D}'((0,\infty);E_\infty)$, that is, for all $\varphi \in \mathcal{D}((0,\infty);E_{-\infty})$
        \begin{align*}
       \int_{0}^{\infty} -\langle u(t), \partial_t{\varphi }(t)\rangle_{H,1} + B_t(u(t),\varphi(t))\  \mathrm d t = 
       \int_{0}^{\infty} \langle g(t), S^\beta \varphi(t)\rangle_{H} \ \mathrm d t.
    \end{align*}
    
     \item [(ii)] $
     \forall \Tilde{a} \in E_{-\infty}, \    \langle u(t),\Tilde{a} \rangle_{H} \rightarrow \langle a, \Tilde{a}\rangle_{H}$, along a sequence tending to $0$.
    
    \end{enumerate}
     
    \end{defn}

A weaker formulation testing against functions $\varphi \in \mathcal{D}([0,\infty);E_{-\infty})$
with right hand side containing the additional term $\langle a, \varphi(0) \rangle_{H}$ is often encountered.
In the end it amounts to the same solutions thanks to \textit{a priori} continuity in $H$,  which only uses the upper bound on $B_t$.

\begin{prop}\label{prop:weaksolCP} Any weak solution to (i) belongs to $C_0([0,\infty); H)$ and $t\mapsto\|u(t)\|_H^2$ satisfies the energy equality
for any  $ \sigma, \tau \in [0,\infty]$ such that $  \sigma < \tau$,
    \begin{align*}
        \left \| u(\tau) \right \|^2_H + 2\mathrm{Re}\int_{\sigma}^{\tau} B_t(u(t),u(t))\  \mathrm d t  = \left \| u(\sigma) \right \|^2_H +  2\mathrm{Re}\int_{\sigma}^{\tau}  
        \langle g(t),S^\beta u(t)\rangle_{H} \  \mathrm d t.
    \end{align*}
    \end{prop}

    \begin{proof}
        We assume $u \in L^2((0,\infty); D_{S,1})$ and the equation implies that $\partial_tu \in L^2((0,\infty); D_{S,-1})+ L^{\rho'}((0,\infty);D_{S,-\beta}).$
        We may apply  Corollary \ref{corenergy}.
    \end{proof}

The main result of this section is the following theorem which puts together all the theory developed so far. 

\begin{thm}\label{ThmCauchy homog} Consider the above assumptions on $B_t$, $\rho, \beta$ and $g, a$.
\begin{enumerate}
   \item There exists a unique weak solution $u$ to the problem \eqref{Pb Cauchy homogène}. Moreover, 
    $u \in C_0([0,\infty);H)\cap L^r((0,\infty);H)$ for any $r \in (2,\infty)$ with $\alpha={2}/{r}$ $($if  $\rho<\infty$, then u is also the restriction to $(0,\infty)$ of an element in $V_{-\beta})$
        and
        $$ \sup_{t\in [0,\infty)} \| u(t)  \|_{H}+\left \| u \right \|_{L^2((0,\infty);D_{S,1})}+\left \| u \right \|_{L^r((0,\infty);D_{S,\alpha})} \leq C \left (  
        \left \| g \right \|_{L^{\rho'}(\mathbb{R};H)}+ \left \| a \right \|_H \right ),$$
where $C=C(M,\nu, \rho)>0$ is a constant independent of $g, a$.
        \item There exists a unique fundamental solution $\Gamma=\{\Gamma(t,s)\}_{0\leq s \leq t <\infty}$ for $\partial_t+\mathcal{B}$ in the sense of Definition \ref{FS} in $(0,\infty)$. In particular, for all $t \in [0,\infty)$, we have the following representation of $u$ :   
        \begin{align}\label{eq:repCauchy0infty}
            u(t)=\Gamma(t,0)a+ \int_{0}^{t} \Gamma(t,s)S^\beta g(s)\ \mathrm ds, 
        \end{align}
         where the integral is weakly defined in $H$ when $\rho<\infty$ and strongly defined in $H$ when $\rho=\infty$ (i.e.,  in the Bochner sense). More precisely, for all $\Tilde{a} \in H$ and $t \in [0,\infty)$, we have equality with  absolutely converging integral
        \begin{align}\label{faible}
            \langle u(t) , \Tilde{a}\rangle_H = \langle \Gamma(t,0)a , \Tilde{a}\rangle_H 
            + \int_{0}^{t} \langle g(s) , S^\beta \Tilde{\Gamma}(s,t)\Tilde{a}\rangle_{H}\  \mathrm ds. 
        \end{align} 
    \end{enumerate}
\end{thm}
\begin{proof}
We start with the existence of such a solution. We extend  $g$ by $0$ on $(-\infty,0]$ and keep the same notation for the extensions. We also extend the family $({B}_t)_t$ to $\mathbb{R}$ by setting ${B}_t=\nu\langle S\cdot, S\cdot \rangle_H$ on $\mathbb{R}\setminus(0,\infty)$ and we keep calling ${\mathcal{B}}$ the operator associated to this family.

Using Proposition \ref{prop:W-beta} when $\rho<\infty$ or Proposition \ref{prop:L1} when $\rho=\infty$ and Corollary \ref{CorRadon}, there exists a unique $ \Tilde{u} \in L^2(\mathbb{R};D_{S,1})$ solution of the equation
\begin{equation*}
    \partial_t \Tilde{u} + \mathcal{B}\Tilde{u} = \delta_0 \otimes a + S^\beta g \ \ \mathrm{in} \ \mathcal{D}'(\mathbb{R};E_\infty).
\end{equation*}
Moreover, $\Tilde{u} \in C_0(\mathbb{R}\setminus\left \{ 0 \right \};H)$, $\Tilde{u}=0$ on $(-\infty,0)$ with $\lim_{t \to 0^+}\Tilde{u}(t)=a$ in $H$ and  if $\rho <\infty$ then the restriction to $(0,\infty)$ of $\Tilde{u}$ is an element in $V_{-\beta}$. Furthermore, we have $\Tilde{u} \in L^r(\mathbb{R};D_{S,\alpha})$ for any $r \in (2,\infty)$ with $\alpha={2}/{r}$ by Proposition \ref{prop:embed r>0} and we have the  estimate
$$ \left \| \Tilde{u} \right \|_{L^\infty(\mathbb{R};H)}+\left \| \Tilde{u} \right \|_{L^2(\mathbb{R};D_{S,1})}+\left \| \Tilde{u} \right \|_{L^r(\mathbb{R};D_{S,\alpha})} \leq C(M,\nu,\rho) \left ( 
\| S^\beta g  \|_{L^{\rho'}(\mathbb{R};H)}+ \left \| a \right \|_H \right ).$$
In addition,  \eqref{WeaksensGreen} in Theorem \ref{ThmGreenReprese} implies \eqref{faible} and \eqref{eq:repCauchy0infty} for $\Tilde{u}(t)$ for all $t\in \mathbb{R}$ with the fundamental solution defined on $\mathbb{R}$.
The candidate $u:= \Tilde{u}_{\scriptscriptstyle{\vert (0,\infty)}}$ satisfies all the required properties of the theorem, proving existence and representation.  

Next, we check uniqueness in the space $L^2((0,\infty);D_{S,1})$. We assume that $u$ is a solution to \eqref{Pb Cauchy homogène} with $a=0$, $f=0$ and $g=0$. By 
Proposition \ref{prop:weaksolCP},
\begin{align*}
          2\mathrm{Re}\int_{0}^{\infty} B_t(u(t),u(t)) \  \mathrm d t = 0.
\end{align*}
Using the coercivity of $B_t$,  we deduce that $u=0$ on $(0,\infty)$. Definition, existence and uniqueness of the fundamental solution in $(0,\infty)$ is \textit{verbatim} as in Section \ref{sec:FS}. 
\end{proof}
\begin{rem}\label{Cauchy Pb homo}
    Uniqueness in the previous proof does not work if we are working on a bounded interval $(0,\mathfrak{T})$ because Corollary \ref{corenergy} fails in this case.
\end{rem}

\begin{rem}
    Of course, by linearity, we can replace $S^\beta g$ by a linear combination of terms  in $L^{\rho'}((0,\infty);D_{S,-\beta})$ for different $\rho$.
\end{rem}

 \section{Inhomogeneous version}\label{sec:Inhomogenous} 
One would like to treat parabolic operators with  elliptic part being ${\mathcal{B}}$ plus lower order terms allowing $T$ to be not injective (e.g., differential operators with Neumann boundary conditions). Here is a setup for doing this effortlessly given the earlier developments.
\subsection{Setup for the inhomogeneous theory}\label{sec:inhomogenouessetup}
As before, consider $T$ and $S=({T}^*{T})^{1/2}$ without assuming that $T$ is injective. One can still define a Borel functional calculus associated to $S$ as in Subsection \ref{Section: calcul bor} by replacing $(0,\infty)$ by $[0,\infty)$. In  the right hand side of the Calder\'on reproducing formula  \eqref{eq:Calderon}, $v$ is replaced by its orthogonal projection onto  $\overline{\mathrm{ran}(S)}$. The most important fact is that for any $\alpha \ge 0$, we can still define $S^\alpha$ as the closed operator $\mathbf{t^{\alpha}}(S)$, which is also positive and self-adjoint but not necessarily injective, having the same null space as $S$.

Let $\Tilde{T}: D(\Tilde{T})= D(T) \rightarrow H\oplus K$ the operator defined by $\Tilde{T}u:=(\lambda u,Tu)$ where $\lambda \in \mathbb{R}_+$. Assume  $\lambda>0$. Then $\Tilde{T}$ is injective and  $\Tilde{S}_\lambda=(\Tilde{T}^*\Tilde{T})^{1/2}=(\lambda^2+S^2)^{{1}/{2}}$  is a  self-adjoint, positive and invertible  operator on $H$, with domain
$D(\Tilde{S}_\lambda)=D(\Tilde{T})=D(T)=D(S)$.

Using that $\Tilde{S}_\lambda =(\lambda^2+S^2)^{{1}/{2}}$, we have that for $\lambda> 0$ and    $\alpha \ge 0$,
\begin{equation*}
    D_{\Tilde{S}_\lambda,\alpha}= D(S^\alpha) \ , \ \ \|  \cdot \|_{\Tilde{S}_\lambda,\alpha} \simeq \|  \cdot \|_{S,\alpha} + \| \cdot \|_H.
\end{equation*}
For $\alpha<0$, we know that the sesquilinear form $(u,v) \mapsto \langle \Tilde{S}_\lambda^\alpha u, \Tilde{S}_\lambda^{-\alpha} v \rangle_H$ defines a canonical duality pairing between $D_{\Tilde{S}_\lambda,\alpha}$ and $D_{\Tilde{S}_\lambda,-\alpha}$. Therefore, for any $u \in D_{\Tilde{S}_\lambda,\alpha}$, there exists $(w,\Tilde{w}) \in H^2$ such that
 $$\langle \Tilde{S}_\lambda^\alpha u, \Tilde{S}_\lambda^{-\alpha} v \rangle_H = \langle w, S^{-\alpha} v \rangle_H+ \langle \Tilde{w},  v \rangle_H,$$
for all $v \in D(S^{-\alpha})$. In this sense, we write $D_{\Tilde{S}_\lambda,\alpha}=S^{-\alpha}H+H$ with norm equivalent to the quotient norm 
$$\|  u \|_{\Tilde{S}_\lambda,\alpha} \simeq \inf_{u=S^{-\alpha}w+\Tilde{w}}(\| w \|_H+\| \Tilde{w} \|_H).$$
  From now on and as before, we set $\Tilde{S}:=\Tilde{S}_1=(1+S^2)^{1/2}$. In conclusion, the ``inhomogeneous" fractional spaces for $S$ become the ``homogeneous" fractional spaces for $\Tilde{S}$, so that  applying the above theory with $\Tilde{S}$ leads to the inhomogeneous theory for $S$ (even if $S$ is non injective).

Finally, we set 
$$ \Tilde{E}_{-\infty}:= \bigcap_{\alpha \in \mathbb{R}} D(\Tilde{S}^{\alpha }) .$$

\subsection{Embeddings and Integral identities}
We begin by noting that Proposition \ref{prop:embed r>0} holds \textit{verbatim} with the same proof, even if $S$ is not necessarily injective. As for continuity and integral identities, we have to modify the statement as follows.
\begin{prop}\label{prop:energyboundedintervalinhomo}
Let $\mathfrak{T} \in (0, \infty]$. Let $u \in L^1((0, \mathfrak{T}); D(S))$ with  $
\int_0^{\mathfrak{T}} \| S u(t) \|_H^2 \, \mathrm{d}t < \infty$
if $\mathfrak{T} < \infty$, or $u \in L^1((0, \mathfrak{T}'); D(S))$ for all $\mathfrak{T}' < \infty$,
with $\int_0^{\infty} \| S u(t) \|_H^2 \, \mathrm{d}t < \infty$ if $\mathfrak{T} = \infty$. Assume that $ \partial_t u = S f + S^\beta g $ in $\mathcal{D}'((0, \mathfrak{T}); \Tilde{E}_{\infty}) $, where $f \in L^2((0, \mathfrak{T}); H)$ and $g \in L^{\rho'}((0, \mathfrak{T}); H)$, with ${\beta} = {2}/{\rho} \in [0, 1)$. When $\mathfrak{T} < \infty$, then $u \in C([0, \mathfrak{T}]; H)$, and the function $t \mapsto \|u(t)\|_H^2$ is absolutely continuous on $[0, \mathfrak{T}]$. For all $\sigma, \tau \in [0, \mathfrak{T}]$ such that $\sigma < \tau$, the following integral identity holds:
\begin{align*}
    \| u(\tau) \|_H^2 - \| u(\sigma) \|_H^2 = 2 \, \mathrm{Re} \int_{\sigma}^{\tau} \langle f(t), S u(t) \rangle_H \, \mathrm{d}t
    + \int_{\sigma}^{\tau} \langle g(t), S^\beta u(t) \rangle_H \, \mathrm{d}t.
\end{align*}
If $\mathfrak{T} = \infty$, then the same conclusion holds on any bounded interval, and $u$ is bounded in $H$.
\end{prop}

\begin{proof}
Using Proposition \ref{prop:embed r>0}, we can express $S f + S^\beta g = S \Tilde{f} + h$, with $\Tilde{f} \in L^2((0, \mathfrak{T}); H)$ and $h \in L^1((0, \mathfrak{T}); H)$.

We start with the case $\mathfrak{T} < \infty$. Consider the orthogonal decomposition $H = \overline{\mathrm{ran}(S)} \oplus \mathrm{nul}(S)$, and write $u = u_1 + u_2$, where $u_1 \in L^1((0, \mathfrak{T}); \overline{\mathrm{ran}(S)})$ satisfies $\int_0^{\mathfrak{T}} \| S u_1(t) \|_H^2 \, \mathrm{d}t < \infty,$
and $u_2 \in L^1((0, \mathfrak{T}); \mathrm{nul}(S))$. Similarly, we decompose $\Tilde{f} = \Tilde{f}_1 + \Tilde{f}_2$ and $h = h_1 + h_2$. We have $\partial_t u_1 = S \Tilde{f}_1 + h_1$ and $ \partial_t u_2 = h_2$ where both equalities hold in $ \mathcal{D}'((0, \mathfrak{T}); \tilde{E}_{\infty})$. We obtain that $\partial_t u_2 \in L^1((0, \mathfrak{T}); \mathrm{nul}(S))$, hence $u_2 \in W^{1,1}((0, \mathfrak{T}); \mathrm{nul}(S)) \hookrightarrow C([0, \mathfrak{T}]; \mathrm{nul}(S))$. Using Corollary \ref{corenergybounded} with $S_{\vert \overline{\mathrm{ran}(S)}}$ which is injective, we conclude that $u_1 \in C([0, \mathfrak{T}]; \overline{\mathrm{ran}(S)})$. Finally, we obtain the energy equality using orthogonality.

When $\mathfrak{T} = \infty$, the conclusion  is already valid on $[0,\infty)$.  To see the behavior at $\infty$, we can use the same decomposition and $u_{1}\in C_{0}([0, \mathfrak{T}]; \overline{\mathrm{ran}(S)})$ from Corollary \ref{corenergy}. As for $u_{2}$ we have by direct integration, that for all $t\ge 0$, $\|u_{2}(t)\|_{H} \le \inf_{\tau\ge 0}\|u_{2}(\tau)\|_{H}+ \| h_{2} \|_{L^1((0, \infty); \mathrm{nul}(S))}$.
\end{proof}

\subsection{The Cauchy problem}\label{sec:inhomogenousCP}
In this section, we are interested in the Cauchy problem on  segments and half-lines, in a non-homogeneous manner. Recall that $\Tilde{S}_\lambda=(\lambda^2+S^2)^{{1}/{2}}$ with $D(\Tilde{S}_\lambda)=D({T})$ and we assume $\lambda\ge 0$ for the moment. Let  $0<\mathfrak{T}\le \infty$. First, let us consider $(\Tilde{B}_t)_{t \in (0,\mathfrak{T})}$ a weakly measurable family of bounded sesquilinear forms on $D({T})\times D({T})$. More precisely, we assume that 
\begin{equation}\label{eq:unifbdd}
     |\Tilde{B}_t(u,v) |\leq  M \| \Tilde{S}_\lambda u \|_H \| \Tilde{S}_\lambda v\|_H,
\end{equation}

for some $M>0$ and for all $u,v \in D({T})$ and $t \in (0,\mathfrak{T})$.
In addition, we assume that the family $(\Tilde{B}_t+\kappa)_{t \in (0,\mathfrak{T})}$ is uniformly coercive for some $\kappa \in \mathbb{R}_+$, \textit{i.e.}, 
\begin{equation}\label{eq:unifcoercive}
     \mathrm{Re}(\Tilde{B}_t(v,v))+\kappa \left\|v \right\|_H^2 \geq \nu  \|\Tilde{S}_\lambda v \|_H^2,
\end{equation}
for some $\nu>0$ and for all $v \in D(T)$ and $t \in (0,\mathfrak{T})$. Notice that $\Tilde{B}_t+\kappa$ satisfies the lower bound in \eqref{EllipticityAbstract} with  $\Tilde{S}$ replacing $S$ and the upper bound with $M+\kappa$ on $(0,\mathfrak{T})$. We denote by $\Tilde{\mathcal{B}}$ the  operator associated to the family $(\Tilde{B}_t)_{t \in (0,\mathfrak{T})}$.
One may represent  $\Tilde{B}_t$ as $ \Tilde{T}^*\Tilde{A}(t)\Tilde{T}$, where $\Tilde{A}(t)$ is bounded on $H\oplus \overline{\mathrm{ran(T)}}$.
If we decide to represent $\Tilde{A}(t)$ in $2\times 2$ matrix form, then $\Tilde{\mathcal{B}}$ writes as ${\mathcal{B}}$ plus lower order terms with bounded operator-valued coefficients.

{On segments, say $(0, \mathfrak{T})$, we can consider the Cauchy problem for all possible values of $\lambda \geq 0$ and $\kappa \geq 0$. On half-lines, say $(0, \infty)$, we restrict the range of the parameters. This leads to the following cases (we will not attempt to track $\lambda$ and $\kappa$ quantitatively).}
\begin{enumerate}
    \item [(a)]  $\mathfrak{T}<\infty$. 
    \item[(b)] $\mathfrak{T}=\infty$, $\lambda>0$ and $\kappa=0$. 
\end{enumerate}
See Remark \ref{rem : cas dégénéré} and  Remark \ref{rem:tinfty,k>0}  for more when  $\mathfrak{T}=\infty$.

We fix $\rho \in [2,\infty]$ and set $\beta={2}/{\rho}$. Given an initial condition $a\in H$ and 
$g\in L^{\rho'}((0,\mathfrak{T});H)$, we wish to solve  the Cauchy problem

\begin{align}\label{Pb Cauchy inhomogène}
\left\{
    \begin{array}{ll}
         \partial_t u +\Tilde{\mathcal{B}}u = 
        \Tilde{S}^\beta {g}    \ \ \mathrm{in} \ \mathcal{D}'((0,\mathfrak{T});\Tilde{E}_{\infty}), \\
         u(0)=a \ \ \mathrm{in} \ \Tilde{E}_{\infty}.
    \end{array}
\right.
\end{align}

Recall that $\Tilde{S}^\beta {g}$ can be written as ${S}^\beta {g_1}+ {g_2}$, with $g_i\in L^{\rho'}((0,\mathfrak{T});H)$, $i=1,2$.

\begin{defn} 
A weak solution to  \eqref{Pb Cauchy inhomogène} is a function $u \in L^1((0,\mathfrak{T});D(S))$ with $\int_{0}^{\mathfrak{T}} \left\|Su(t) \right\|_H^2 \mathrm d t <\infty$ if $\mathfrak{T}<\infty$ or $u \in L^1((0,\mathfrak{T}'); D(S))$ for all $\mathfrak{T'}<\infty$ with $\int_{0}^{\infty} \left\|Su(t) \right\|_H^2 \ \mathrm d t <\infty$ if $\mathfrak{T}=\infty$
and such that
\begin{enumerate}
        \item [(i)] $u$ solves the first equation in $\mathcal{D}'((0,\mathfrak{T});\Tilde{E}_{\infty})$, that is, for all $\varphi \in \mathcal{D}((0,\mathfrak{T});\Tilde{E}_{-\infty})$
    \begin{align*}
         \int_{0}^{\mathfrak{T}} -\langle u(t), \partial_t{\varphi }(t)\rangle_H + \Tilde{B}_t(u(t),\varphi(t))\ \mathrm d t 
        = \int_{0}^{\mathfrak{T}} 
        \langle {g}(t), \Tilde{S}^\beta \varphi(t)\rangle_{H} \ \mathrm d t.
    \end{align*}
    
    \item[(ii)] $
        \forall \Tilde{a} \in \Tilde{E}_{-\infty}, \ \langle u(t),\Tilde{a} \rangle_H \rightarrow \langle  a, \Tilde{a}\rangle_H
   $ along a sequence tending to 0.
\end{enumerate}

\end{defn}

The difference with the homogeneous situation is the global or local $L^1(H)$ condition.
        
Again, a weaker formulation testing against functions $\varphi \in \mathcal{D}([0,\mathfrak{T});\Tilde{E}_{-\infty})$
with right hand side containing the additional term $\langle a, \varphi(0) \rangle_{H}$ can be considered. In the end it amounts to the same solutions thanks to \textit{a priori} continuity in $H$,  which only uses the upper bound on $\Tilde{B}_t$.

\begin{lem}\label{lem:weaksolCPlocal} In case (a), any weak solution to (i) belongs to $C([0,\mathfrak{T}]; H)$, and $t\mapsto\|u(t)\|_H^2$ satisfies the energy equality
for any  $ \sigma, \tau \in [0,\mathfrak{T}]$ such that $  \sigma < \tau$, 
    \begin{align*}
         \left \| u(\tau) \right \|^2_H + 2\mathrm{Re}\int_{\sigma}^{\tau} \Tilde{B}_t(u(t),u(t))\  \mathrm d t
        = \left \| u(\sigma) \right \|^2_H +  2\mathrm{Re}\int_{\sigma}^{\tau}
        \langle {g}(t), \Tilde{S}^\beta u(t)\rangle_{H}\ \mathrm d t.
    \end{align*}
In case (b), we have the same conclusion on any bounded interval.
\end{lem}
\begin{proof} 
Case (b) follows directly from case (a). To prove the latter, using Proposition \ref{prop:embed r>0} with $\Tilde{S}$ replacing $S$, we can  write  $\Tilde{S}^\beta {g}=\Tilde{S}\Tilde{f}+\Tilde{h}$ with $\Tilde{f} \in L^2((0, \mathfrak{T}); H)$ and $\Tilde{h} \in L^1((0, \mathfrak{T}); H)$. This can be expressed as ${S}{f}+{h}$ with $\Tilde{f} \in L^2((0, \mathfrak{T}); H)$ and $\Tilde{h} \in L^1((0, \mathfrak{T}); H)$. We conclude on using Proposition \ref{prop:energyboundedintervalinhomo}.
\end{proof}

The main result of this section is the following theorem which puts together the inhomogeneous version of all the theory developed so far.
\begin{thm}\label{ThmCauchy inhomog} Consider the above assumptions on $\Tilde{B}_t$, $\lambda,\kappa$, $f, g$ and $a$. 
\begin{enumerate}
        \item There exists a unique weak solution $u$ to the problem \eqref{Pb Cauchy inhomogène}. Moreover,  
      $u \in C([0,\mathfrak{T}];H) \cap L^r((0,\mathfrak{T});D(\Tilde{S}^\alpha))$ for all $r\in [2,\infty]$ with $\alpha=2/r$, with $u(\infty)=0$ in case (b) where $\mathfrak{T}=\infty$, and we have the estimate
        \begin{align*}
            \sup_{t\in [0,\mathfrak{T}]} \| u(t) \|_{H}&+\| \Tilde{S}^\alpha u \|_{L^r((0,\mathfrak{T});H)} 
            \leq C  (  
            \left \| g \right \|_{L^{\rho'}((0,\mathfrak{T});H)}+ \left \| a \right \|_H  ),
        \end{align*} 
        where $C=C(M,\nu,\rho,\kappa,\mathfrak{T})>0$ is a constant independent of $g$ and $a$.
        \item There exists a unique fundamental solution $\Gamma_{\Tilde{\mathcal{B}}}=(\Gamma_{\Tilde{\mathcal{B}}}(t,s))_{0\leq s \leq t \leq \mathfrak{T}}$  for $\partial_t+\Tilde{\mathcal{B}}$ in the sense of Definition \ref{FS} in $(0,\mathfrak{T})$ (by convention, set $\Gamma(\infty,s)=0$ if $\mathfrak{T}=\infty$). In particular, for all $t \in [0,\mathfrak{T}]$, we have the following representation of $u$ : 
        \begin{align*}
            u(t)=\Gamma_{\Tilde{\mathcal{B}}}(t,0)a+\int_{0}^{t}\Gamma_{\Tilde{\mathcal{B}}}(t,s)\Tilde{S}^\beta {g}(s) \ \mathrm ds , 
        \end{align*}
         where the  integral is weakly defined in $H$ when $\rho<\infty$ and strongly defined when $\rho=\infty$ (i.e., in the Bochner sense). 
         For all $\Tilde{a} \in H$ and $t \in [0,\mathfrak{T}]$,
        \begin{align*}\label{eq:repCauchy0 T}
            \langle u(t) , \Tilde{a}\rangle_H &= \langle \Gamma_{\Tilde{\mathcal{B}}}(t,0)a , \Tilde{a}\rangle_H   
           +\int_{0}^{t} \langle {g}(s) ,  \Tilde{S}^\beta \Tilde{\Gamma}_{\Tilde{\mathcal{B}}}(s,t)\Tilde{a}\rangle_H \ \mathrm ds. 
        \end{align*}
    \end{enumerate}
\end{thm}
\begin{proof} We begin with existence.

\paragraph{\textit{Existence in case (b) }}  Apply Theorem \ref{ThmCauchy homog} with $\Tilde{S}$ replacing 
$S$ and as the right hand side belongs to $L^{\rho'}((0,\infty);D_{\Tilde{S},-\beta})$. This shows the existence of a weak solution $v$ in $L^2((0,\infty); D(\Tilde{S}))$, which also belongs to $C_0([0,\infty); H)$ and $L^r((0,\infty); D(\Tilde{S}^\alpha))$.

\paragraph{\textit{Existence in case (a) }}
Extend $g$ by $0$ and $\Tilde{B}_t$ by $\nu\langle \Tilde{S}\cdot, \Tilde{S}\cdot \rangle_H$ on $(\mathfrak{T}, \infty)$ if  $\mathfrak{T}<\infty$ and use the same notation. Let $\kappa'>\kappa$. Apply Theorem \ref{ThmCauchy homog} with $\Tilde{S}$ replacing $S$ and with right hand side in $L^{\rho'}((0,\infty);D_{\Tilde{S},-\beta})$ to the auxiliary Cauchy problem 
\begin{align*}
\left\{
    \begin{array}{ll}
        \partial_t v + (\Tilde{\mathcal{B}}+\kappa')v =
        \Tilde{S}^\beta  (e^{-\kappa' t} {g})  \ \ \mathrm{in} \ \mathcal{D}'((0,\infty); \Tilde{E}_\infty), \\
        v(0)=a \  \mathrm{in} \ \Tilde{E}_\infty,
    \end{array}
\right.
\end{align*}  and obtain a weak solution $v$ in $L^2((0,\infty); D(\Tilde{S}))$.
The function $u:=e^{\kappa' t} v$ restricted to $[0,\mathfrak{T}]$ gives us a weak solution with the desired properties.

Next, we prove uniqueness. Assume $u$ is a weak solution to \eqref{Pb Cauchy inhomogène} with $a=0$ and $g=0$.
\paragraph{\textit{Uniqueness in case (b) }} We have $u \in L^1((0,\mathfrak{T}'); D(S))$ for all $\mathfrak{T'}<\infty$ and $\int_{0}^{\infty} \left\|Su(t) \right\|_H^2 \ \mathrm d t <\infty$. Applying Lemma \ref{lem:weaksolCPlocal}, we have $u \in C([0,\mathfrak{T}']; H)$ for all $\mathfrak{T'}>0$, $u(0)=0$ and
 \begin{equation*}
  \left \| u(\mathfrak{T}') \right \|^2_H +  2\mathrm{Re}\int_{0}^{\mathfrak{T}'}  \Tilde{B}_t(u(t),u(t))  \ \mathrm dt=  0.
 \end{equation*}
Using the coercivity of $\Tilde{B}_t $,  we deduce that $u=0$ on $(0,\mathfrak{T}')$ and therefore, $u=0$ on $[0,\infty)$.
\
\paragraph{\textit{Uniqueness in case (a) }}
We have $u \in L^1((0,\mathfrak{T});D(S))$ with $\int_{0}^{\mathfrak{T}} \left\|Su(t) \right\|_H^2 \mathrm d t <\infty$. Let $\kappa'>\kappa$. Set $v = e^{-\kappa' t}u$ on $(0,\mathfrak{T})$ so that $v \in L^1((0,\mathfrak{T});D(S))$ with $\int_{0}^{\mathfrak{T}} \left\|Sv(t) \right\|_H^2 \mathrm d t <\infty$ and $\partial_t v + (\Tilde{\mathcal{B}}+\kappa') v =0$ in $\mathcal{D}'((0,\mathfrak{T});\Tilde{E}_\infty)$.  Applying Lemma \ref{lem:weaksolCPlocal}, we have $v \in C([0,\mathfrak{T}]; H)$, $v(0)=0$ and  
 \begin{equation*}
  \left \| v(\mathfrak{T}) \right \|^2_H +  2\mathrm{Re}\int_{0}^{\mathfrak{T}}  \Tilde{B}_t(v(t),v(t)) + \kappa' \left \| v(t) \right \|^2_H \ \mathrm dt=  0.
 \end{equation*}
Using the coercivity of $\Tilde{B}_t+\kappa'$ resulting from \eqref{eq:unifcoercive}, we deduce that $v=0$ and therefore, $u=0$ on $[0,\mathfrak{T}]$.

Finally, definition, existence and uniqueness of the fundamental solution $\Gamma_{\Tilde{\mathcal{B}}}$ can be obtained easily by proceeding as in Section \ref{sec:FS}.
\end{proof}
\begin{rem}\label{rem:tinfty,k>0}
    If $\mathfrak{T}=\infty$ with $\kappa> 0$, then we can construct a weak solution but it does not satisfy $\int_{0}^{\infty} \|Su(t) \|_H^2 \ \mathrm d t <\infty$.
\end{rem}

\begin{rem}\label{rem : cas dégénéré}
For $\mathfrak{T}=\infty$, $\lambda=0$, $\kappa=0$ and $\Tilde{S}^\beta g$ replaced by $S^\beta g$,  we can apply Theorem \ref{ThmCauchy homog} provided that $T$ is injective. However, when $T$ (hence $S$) is not injective then the proof of  Theorem \ref{ThmCauchy inhomog} provides us with a global solution but not with limit 0 at $\infty$. In fact, the zero limit at $\infty$ fails. Take $u_0 \in \mathrm{nul}(S) \setminus \{0\}$ and set $u(t)=u_0$ for all $t \ge 0$. We have $u \in L^2((0,\mathfrak{T}');H)$ for all $\mathfrak{T'}<\infty$ with $\int_{0}^{\infty} \left\|Su(t) \right\|_H^2 \ \mathrm d t =0 <\infty$. Moreover, $u$ is a weak solution to the abstract heat equation 
\begin{align*}
\left\{
    \begin{array}{ll}
        \partial_t u + S^2u = 0  \ \ \mathrm{in}\ \mathcal{D}'((0,\infty); \Tilde{E}_\infty), \\
        u(0)=u_0 \  \mathrm{in} \ \Tilde{E}_\infty,
    \end{array}
\right.
\end{align*}
with $\lim_{t\rightarrow\infty}u(t)=u_0$.
\end{rem}

\begin{rem}
Consider the special case  $\Tilde{\mathcal{B}}=\mathcal{B}+ \omega$, $\omega>0$, on $(0,\infty)$, keeping the condition \eqref{EllipticityAbstract} for $\mathcal{B}$ with $T$ (and $S$) injective. We have (b) with $\lambda=1$, constant  $\sup(M,\omega)$ in \eqref{eq:unifbdd}    and constant  $\inf(\nu, \omega)$ in \eqref{eq:unifcoercive}.
 The theorem above applies and  gives us  fundamental solution operators  $\Gamma_{\mathcal{B}+ \omega}(t,s)$, defined for $0\le s\le t<\infty$.  Call $\Gamma_{\mathcal{B}}(t,s)$ the one obtained in the previous section. Uniqueness for the Cauchy problem for  $\partial_t + \mathcal{B}+ \omega$ holds in  $L^2((0,\mathfrak{T}), D(S))$ for all $\mathfrak{T}<\infty$ and this shows that   $  \Gamma_{\mathcal{B}+ \omega}(t,s)=e^{-\omega(t-s)}\Gamma_{\mathcal{B}}(t,s)$. Working on $\mathbb{R}$, then we obtain the equality for  $-\infty< s\le t<\infty$.
\end{rem}

\section{The final step towards concrete situations}\label{subsection 4.9} 
The reader might wonder how to apply our theory in concrete situations, where  the abstract  spaces of test functions $\mathcal{D}(I;E_{-\infty})$ or $\mathcal{D}(I;\Tilde{E}_{-\infty})$  might not be related to usual spaces of test functions.
The following result gives us a sufficient condition showing that one can replace $E_{-\infty}$ or $\Tilde{E}_{-\infty}$ by an arbitrary dense set in the domain of $S$, sometimes called a core of $D(S)=D(T)$.

\begin{thm}\label{thm: passage au concret} Let $\mathrm{D}$ be a Hausdorff topological vector space with continuous and dense inclusion $\mathrm{D} \hookrightarrow D(S)$, where $D(S)$ is equipped with the graph norm. Assume a priori  that weak solutions belong to  $L^1_{\mathrm{loc}}(I;H)$, and replace the test function space  by $\mathcal{D}(I;\mathrm{D})$ in their definition, with in the latter case, $\partial_t u $  computed via  : $$\forall \varphi \in  \mathcal{D}(I;\mathrm{D}), \  \llangle  \partial_t u,  \varphi \rrangle = -\int_I \langle u(t),\partial_t\varphi(t)\rangle_{H} \ \mathrm{d}t.$$
Then our well-posedness results are the same: this means that existence with estimates, uniqueness (requiring additionally $u\in L^1_{\mathrm{loc}}(I;H)$ in the uniqueness class), and energy equalities hold. 
\end{thm}

The proof relies on the following density lemma. Denote by $\| u \|_{D(S)}=(\|u\|_H^2+\| Su\|_{H}^2)^{1/2}$ the Hilbertian graph norm  and $ \langle \cdot,\cdot\rangle_{_{D(S)}} $ the corresponding  inner product.
\begin{lem}\label{lem: densité D}  Let $\mathrm{D}$ be as in the above theorem.
    For all open interval $I \subset \mathbb{R}$, $\mathcal{D}(I;\mathrm{D})$ is a dense subspace of $\mathcal{D}(I;D(S))$ in the following sense :  for all $\varphi \in \mathcal{D}(I; D(S))$, there exists a sequence $(\varphi_k)_{k \geq 0} \in  \mathrm{span}(\mathcal{D}(I)\otimes \mathrm{D})^{\mathbb{N}} $ such that 
    \begin{enumerate}
        \item For all $k \geq 0$, $\mathrm{supp}(\varphi_k) \subset \mathrm{supp}(\varphi)$.
        \item For all $k \geq 0$ and $t \in I$,
        $$  \left\|\varphi_k(t) \right\|_{D(S)} \leq 3  \left\| \varphi(t) \right\|_{D(S)}, \ \ \|\partial_t \varphi_k(t) \|_{D(S)}  \leq 3 \left\| \partial_t \varphi(t) \right\|_{D(S)} .$$
        \item For all $t \in I$, $\left\|\varphi_k(t)-\varphi(t) \right\|_{D(S)}+\left\|\partial_t \varphi_k(t) - \partial_t \varphi(t) \right\|_{D(S)}\to 0$ as $k\to \infty$.
        \item For all $\beta \in [0,1]$, $t\in I$ and $k\geq 0$, $\| S^\beta  \varphi_k(t)\|_H \le 3\ \| \varphi(t)  \|_{D(S)} $ and $\| S^\beta(  \varphi_k(t)- \varphi(t) ) \|_H \to 0$ as $k\to \infty$. 
    \end{enumerate}

\end{lem}
\begin{proof}
The space $(D(S),\left\| \cdot \right\|_{D(S)})$ is separable as it is isometric to a subspace of $H\times H$ which is separable. Let $(w_j)_{j\in \mathbb{N}} \in D(S)^\mathbb{N}$ be a Hilbertian basis of $ \left( D(S),\left\| \cdot \right\|_{D(S)} \right)$. As $\mathrm{D}$ is dense in $D(S)$ then for all $j\ge 0$, one can find a sequence $(v_j^k)_{k\in \mathbb{N}} \in \mathrm{D}^\mathbb{N}$ such that for all $k \ge 0$, $\| w_j-v_j^k \|_{D(S)} \leq \frac{1}{2^{j+k}} $. For  $h=\sum_{j\ge 0} \alpha_j w_j\in D(S)$, one can see using Cauchy-Schwarz inequality and Plancherel  that $\|\sum_{j=0}^k \alpha_j (v_j^k-w_j)\|_{D(S)} \le \sum_{j=0}^k \frac{|\alpha_j|}{2^{j+k}} \le  \frac{4}{3\cdot 2^k}\|h\|_{D(S)}$, so that $h_k=\sum_{j=0}^k \alpha_j v_j^k$ satisfies  $\|h_k-h\|_{D(S)} \le \frac{4}{3\cdot 2^k}\|h\|_{D(S)} + \|\sum_{j\ge k+1} \alpha_j w_j\|_{D(S)}$ and $\|h_k\|_{D(S)} \le (\frac{4}{3}+1)\|h\|_{D(S)}$. 
Now, fix $\varphi \in \mathcal{D}(I;D(S))$ and set for all $k \ge 0$ and $t \in I$,
    $$ \varphi_k(t):= \sum_{j=0}^{k} \langle \varphi(t),w_j\rangle_{_{D(S)}}   v_j^k \ . $$
Clearly, the sequence $ (\varphi_k)_{k \in \mathbb{N}} \in \mathrm{span}(\mathcal{D}(I)\otimes \mathrm{D})^{\mathbb{N}}$ with (1), and (2) and (3) follow from the above estimates.  Finally, (4) follows  from the moments inequality combined with (2) and (3).
\end{proof}

\begin{proof}[Proof of Theorem \ref{thm: passage au concret}]  The case using $\mathcal{D}(I;\Tilde{E}_{-\infty})$ being similar, it only suffices to show that with the a priori requirement that  weak solutions also belong to $L^1_{\mathrm{loc}}(I;H)$, the formulations of the equations against test functions in $ \mathcal{D}(I ; E_{-\infty})$  and in $\mathcal{D}(I;\mathrm{D})$ are equivalent, because then they have the same solutions. In fact, they are equivalent to a formulation against test functions in $\mathcal{D}(I ; D(S))$. Indeed, if  $u \in L^2(I;D_{S,1})\cap L^1_{\mathrm{loc}}(I;H)$ then by Lemma \ref{lem: Jalpha bien définie}, we have for all $ \varphi \in \mathcal{D}(I ; E_{-\infty})$,
$$\llangle u , \partial_t \varphi \rrangle_{\mathcal{D}',\mathcal{D}} =\int_{I} \langle u (t), \partial_t \varphi(t) \rangle_{H,1} \ \mathrm{d}t= \int_{I} \langle S u(t) , S^{-1} \partial_t \varphi(t) \rangle_H \ \mathrm{d}t = \int_{I} \langle u (t), \partial_t \varphi(t) \rangle_H \ \mathrm{d}t.$$
Applying that   $ \mathcal{D}(I ; E_{-\infty})$ is dense $\mathcal{D}(I ; D(S))$ as in  Lemma \ref{lem: densité D} and dominated convergence, we can see that the weak formulation for all $\varphi \in \mathcal{D}(I ; D(S))$ holds. Of course we can conversely restrict to test functions $ \mathcal{D}(I ; E_{-\infty})$, showing that the formulations testing  with $\varphi \in \mathcal{D}(I ; E_{-\infty})$ or  $\varphi \in \mathcal{D}(I ; D(S))$ are equivalent. This would be the same starting from another dense set $\mathrm{D}$. Finally, the initial data property in the Cauchy problems testing against elements in $E_{-\infty}$ is equivalent to testing against arbitrary elements in $H$ by density as $u(t)$ belongs almost everywhere to $H$. This would be the same replacing $E_{-\infty}$ by another dense set $\mathrm{D}$ in $D(S)$ as it would also be dense in $H$.
 \end{proof}
\section{Three applications }\label{Section 6}
\subsection{Parabolic Cauchy problems on domains with Dirichlet boundary condition}
Let $n\ge 1$ and $\Omega \subset \mathbb{R}^n$ an open set. We denote by $L^2(\Omega)$ the Hilbert space of square integrable functions on $\Omega$ with respect to the Lebesgue measure $\mathrm{d}x$ with norm  denoted by ${\lVert \cdot  \rVert}_{2}$ and its inner product by $\langle \cdot ,\cdot \rangle$. As usual, we denote by $\mathcal{D}(\Omega)$ the class of smooth and compactly supported functions on $\Omega$. We set $H^1(\Omega):=\left\{ u \in L^2(\Omega)  :  \nabla_x u \in L^2(\Omega) \right\}$ and it is a Hilbert space for the norm $\left\| u \right\|_{H^1(\Omega)}:=( \left\| u \right\|_{2}^2+\left\| \nabla_x u \right\|_{2}^2 )^{1/2}$. Finally, $H^1_0(\Omega)$ is defined as the closure of $\mathcal{D}(\Omega)$ in $(H^1(\Omega),\left\| \cdot \right\|_{H^1(\Omega)})$. 

We denote by $-\Delta_D $ the unbounded operator on $L^2(\Omega)$ associated to the positive symmetric sesquilinear form on $H^1_0(\Omega) \times H^1_0(\Omega)$ defined by
\begin{equation*}
    (u,v) \mapsto \int_{\Omega} \nabla_x u(x) \cdot  \overline{\nabla_x v(x)} \ \mathrm dx.
\end{equation*}

Let $A: I\times \Omega \rightarrow M_n(\mathbb{C})$ be a matrix-valued function with complex measurable entries and such that 
\begin{equation}\label{ellipticité A}
\left | A(t,x)\xi \cdot \zeta  \right |\leq M \left | \xi \right |\left | \zeta  \right |, \ \ \ \  
\nu \left | \xi \right |^2\leq \mathrm{Re}(A(t,x)\xi\cdot \overline{\xi})
\end{equation}
for some $M,\nu>0$ and for all $\zeta, \xi \in \mathbb{C}^n$ and $(t,x) \in I\times \Omega$.
We let $-\mathrm{div}_x$ be the adjoint of $\nabla_x:H^1_0(\Omega) \to L^2(\Omega)^n $ and use the customary notation $\mathcal{B}= -\mathrm{div}_x A(t,\cdot)\nabla_x.$

Fix $\mathfrak{T}>0$,  $\rho \in (2,\infty)$, set $\beta={2}/{\rho} \in (0,1)$ and let $\rho'$ be the H\"older conjugate of $\rho$. For  $ {f} \in L^2((0,\mathfrak{T});L^2(\Omega)^n)$, $g\in L^{\rho'}((0,\mathfrak{T});L^2(\Omega))$, $h \in L^1((0,\mathfrak{T});L^2(\Omega))$ and $\psi \in L^2(\Omega)$,
consider the following Cauchy problem 
\begin{align}\label{Pb Cauchy Dirichlet}
\left\{
    \begin{array}{ll}
        \partial_t u - \mathrm{div}_x(A(t,\cdot)\nabla_x u) =   -\mathrm{div}_x {f} +(-\Delta_D)^{\beta/2}{g}+h  \ \mathrm{in} \  \mathcal{D}'((0,\mathfrak{T})\times \Omega), \\
        u(t) \rightarrow \psi \  \mathrm{ in } \ \mathcal{D'}(\Omega) \ \mathrm{as} \ t \rightarrow 0^+.
    \end{array}
\right.
\end{align}

The first equation is interpreted in the weak sense according to the following definition.
\begin{defn}\label{defn:Dirichlet}
    A weak solution to the first equation in \eqref{Pb Cauchy Dirichlet} is a (complex-valued) {function $u \in L^1((0,\mathfrak{T});H^1_0(\Omega))$ with $\int_0^\mathfrak{T}\|\nabla_x u(t)\|^2_2\, \mathrm{d}t<\infty$} such that     for all $\varphi \in \mathcal{D}((0,\mathfrak{T})\times \Omega)$,
    \begin{align*}
        \int_{0}^{\mathfrak{T}}\int_\Omega 
        &- u(t,x) \partial_t \varphi(t,x)+A(t,x)\nabla_x u(t,x) \cdot \nabla_x \varphi(t,x) \ \mathrm{d}x \mathrm{d}t
        \\& = \int_{0}^{\mathfrak{T}}\int_\Omega {f}(t,x) \cdot \nabla_x \varphi(t,x)+  {g}(t,x) (-\Delta_D)^{\beta/2} \varphi(t,x) \ \mathrm{d}x +h(t,x)\varphi(t,x)\  \mathrm{d}x\mathrm{d}t.
    \end{align*}
\end{defn}
The consequence of our theory is
\begin{thm}[Cauchy problem on $(0,\mathfrak{T})$]\label{thm: Pb Cauchy Dirichlet} Let $f,g,h, \psi$ be as above. 
\begin{enumerate}
    \item 
There exists a unique  weak solution to the Cauchy problem \eqref{Pb Cauchy Dirichlet} as defined above. Moreover, $u \in C([0,\mathfrak{T}];L^2(\Omega))$ with $u(0)=\psi$, the application $t \mapsto \| u(t)  \|^2_{2}$ is absolutely continuous on $[0,\mathfrak{T}]$ and we can write the energy equalities. Furthermore, $u \in L^r((0,\mathfrak{T});D((-\Delta_D)^{\alpha/2}))$ for any $\alpha \in (0,1]$ with $r={2}/{\alpha} \in [2,\infty)$ and we have
\begin{align*}
            \sup_{t\in [0,\mathfrak{T}]} &\| u(t) \|_{2}+ \| (-\Delta_D)^{\alpha/2}u \|_{L^r((0,\mathfrak{T});H)} \\
            &\leq C  ( \left \| f \right \|_{L^2((0,\mathfrak{T});L^2(\Omega)^n)}
            + \left \| g \right \|_{L^{\rho'}((0,\mathfrak{T});L^2(\Omega))}
            + \left \| h \right \|_{L^{1}((0,\mathfrak{T});L^2(\Omega))}+ \| \psi  \|_{2}  ),
\end{align*} 
where $C=C(M,\nu,\rho,\mathfrak{T})>0$ is a constant independent of the data $f,g,h$ and $\psi$. 
\item There exists a unique fundamental solution $\Gamma=(\Gamma(t,s))_{0\leq s \leq t \leq \mathfrak{T}}$ for $\partial_t-\mathrm{div}_x A(t,\cdot)\nabla_x$. In particular,   for all $t \in [0,\mathfrak{T}]$, we have the following representation of $u(t)$ : 
\begin{align*}
    u(t) = \Gamma(t,0)\psi + \int_{0}^{t} \Gamma(t,\tau)( -\mathrm{div}_x{f})(\tau)\ \mathrm d \tau 
    + \int_{0}^{t} \Gamma(t,\tau) (-\Delta_D)^{\beta/2}{g}(\tau)\ \mathrm d \tau +\int_{0}^{t} \Gamma(t,\tau)h(\tau)\ \mathrm d \tau,
\end{align*}
where the two integrals with ${f}$ and ${g}$ are weakly defined in $L^2(\Omega)$ while the other one converges strongly (i.e., in the Bochner sense). More precisely, we have for all $\Tilde{\psi} \in L^2(\Omega)$ and $t \in [0,\mathfrak{T}]$,
 \begin{align*}
            \langle u(t) , \Tilde{\psi} \rangle = \langle \Gamma(t,0)\psi , &
            \Tilde{\psi} \rangle  
            + \int_{0}^{t} \langle {f}(s) , \nabla_x \Tilde{\Gamma}(s,t)\Tilde{\psi}\rangle \, \mathrm ds  \\&
            +\int_{0}^{t} \langle  {g}(s) , (-\Delta_D)^{\beta/2} \Tilde{\Gamma}(s,t) \Tilde{\psi} \rangle \, \mathrm ds
            +\int_{0}^{t} \langle \Gamma(t,s) h(s) , \Tilde{\psi} \rangle \, \mathrm ds.
        \end{align*}
        \end{enumerate}
\end{thm}

\begin{proof}
{As $\mathcal{D}(\Omega)$ is dense in $H^1_0(\Omega)$ with respect to the graph norm of the injective self-adjoint operator $S=(-\Delta_D)^{1/2}$ by definition, we are in the context of Theorem \ref{thm: passage au concret} in Section \ref{subsection 4.9}, which corresponds to Theorem \ref{ThmCauchy inhomog} for each $f,g,h$ by linearity and using that $-\mathrm{div}_x{f} =(-\Delta_D)^{1/2} \Tilde{f}$ with $\Tilde{f}\in L^2((0,\mathfrak{T}); L^2(\Omega))$, with $B_t : H^1_0(\Omega) \times H^1_0(\Omega) \rightarrow \mathbb{C}$ being the sesquilinear form defined via
\begin{equation*}
   \forall u,v \in H^1_0(\Omega) : \  B_t(u,v) :=  \int_{\Omega} A(t,x)\nabla_x u(x) \cdot \overline{\nabla_x v(x)} \, \mathrm{d} x.
\end{equation*}
}
 
\end{proof}

\begin{rems}
    \begin{enumerate}
        \item With the modification in the definition of weak solutions, the statement applies for  the Cauchy problem on $(0,\infty)$ and $u$ has limit 0 at $\infty$. (Use Theorems \ref{thm: passage au concret} and \ref{ThmCauchy homog}).
        \item Remark that the theory applies for complex coefficients. In particular, we do not assume any local regularity for weak solutions and fundamental solutions are merely bounded operators. Bounds on their kernels need additional assumptions. 
        \item If $\Omega$ is bounded (or only bounded in one direction), then  Poincar\'e inequality holds on $H^1_0(\Omega)$ \cite[Proposition 3.25]{egert2024harmonic}, and it follows that  $D(S)=D_{S,1}=H^1_0(\Omega)$ with equivalent norms. In particular, the inhomogeneous and homogeneous theories developed in Section \ref{Section 5} are the same for this concrete case.
        \item We may want to replace the spaces $L^r((0,\mathfrak{T});D((-\Delta_D)^{\alpha/2}))$ by mixed Lebesgue spaces $L^r((0,\mathfrak{T});L^q(\Omega))$.  The embeddings of the domains of the fractional powers $(-\Delta_D)^{\alpha/2}$  into Lebesgue spaces $L^q(\Omega)$ depend on the geometry of the domain. See the discussion in \cite{auscher2023universal}.
    \end{enumerate}
\end{rems}

\subsection{Parabolic integro-differential operators}
The second application is for integro-differential parabolic  operators $\partial_t+ \mathcal{B}$
where  $\mathcal{B}$ is associated with a sesquilinear form  $B_t$
satisfying  \eqref{EllipticityAbstract} for $t\in I$ ($I$ open interval) with $T=S=(-\Delta)^{\gamma/2}$ for some $\gamma>0$. 
The most notable example from the references mentioned in the introduction is that of  $\mathcal{B}$ arising from the family of forms
\begin{align*}
 B_t(u,v) := \iint_{\mathbb{R}^n \times \mathbb{R}^n} K(t,x,y) \frac{(u(x) - u(y)) \overline{(v(x)-v(y))}}{|x-y|^{n+2\gamma}} \, \mathrm{d} x \, \mathrm{d} y,
\end{align*}
for some $\gamma \in (0,1)$ and
 $u,v \in W^{\gamma,2}(\mathbb{R}^n)$. We assume here $K: I\times \mathbb{R}^n \times \mathbb{R}^n \to \mathbb{C}$  to be a measurable kernel that satisfies the accretivity condition for some $\lambda>0$, 
 \begin{align}
\label{eq:ellip}
0<\lambda \leq \mathrm{Re}\,  K(t,x,y) \leq |K(t,x,y)| \leq \lambda^{-1} \qquad (\text{a.e. }(t,x,y) \in I \times \mathbb{R}^n \times \mathbb{R}^n). 
\end{align}
The Sobolev space $W^{\gamma,2}(\mathbb{R}^n)$ is the space of measurable functions $u$ on $\mathbb{R}^n$ with norm $\|u\|_\gamma$ given by 
 \begin{align*}
  \|u\|_\gamma^2=   \int_{\mathbb{R}^n} |u(x)|^2\, \mathrm{d}x + \iint_{\mathbb{R}^n \times \mathbb{R}^n} \frac{|u(x) - u(y)|^2}{|x-y|^{n+2\gamma}} \, \mathrm{d} x \, \mathrm{d} y,
 \end{align*}
 and it is well known that $W^{\gamma,2}(\mathbb{R}^n)$ agrees with the domain of $(-\Delta)^{\gamma/2}$ and that the last term in the expression above is comparable to $\|(-\Delta)^{\gamma/2}u\|_2^2$. Using this observation,   \eqref{eq:ellip} and Cauchy-Scwharz inequality, we can check \eqref{EllipticityAbstract} with $T=S=(-\Delta)^{\gamma/2}$. 
 
From now on, we can apply the theory developed so far and obtain well-posedness results on $I\times \mathbb{R}^n$ but we shall not repeat the statements and leave that to the reader. 
May be the most notable outcome is that there always exists a unique fundamental solution, and this seems new at this level of generality. 

\begin{thm} Let $\gamma>0$.  
    The integro-differential parabolic  operator $\partial_t+ \mathcal{B}$ on $I\times \mathbb{R}^n$  has a unique fundamental solution.
\end{thm}

\subsection{Degenerate parabolic operators}
The third application concerns degenerate parabolic operators on $\mathbb{R}^n$. We fix a weight $\omega$ in the Muckenhoupt class $A_2(\mathbb{R}^n, \mathrm{d}x)$, meaning that $\omega: \mathbb{R}^n \to \mathbb{R}$ is a measurable and positive function satisfying
\begin{equation*}
     [ \omega  ]_{A_2}:= \sup_{Q \subset \mathbb{R}^n  } \left ( \fint_Q\omega(x)\,\mathrm{d} x   \right ) \left ( \fint_Q\omega^{-1}(x)\,\mathrm{d} x   \right ) < \infty,
\end{equation*}
where the supremum is taken over all cubes $Q \subset \mathbb{R}^n$. For background on Muckenhoupt weights and related results, we refer to \cite[Ch. V]{Stein1993_HA}.

We denote by $L^2_\omega(\mathbb{R}^n) := L^2(\mathbb{R}^n, \mathrm{d}\omega)$ the Hilbert space of square-integrable functions with respect to $\mathrm{d}\omega$, with norm denoted by ${\lVert \cdot \rVert}_{2,\omega}$ and inner product $\langle \cdot , \cdot \rangle_{2,\omega}$. It is known that $$\mathcal{D}(\mathbb{R}^n)  
\subset L^2_\omega(\mathbb{R}^n) \subset L^1_{\text{loc}}(\mathbb{R}^n, \mathrm{d}x) \subset \mathcal{D}'(\mathbb{R}^n)$$ and the first inclusion is dense.

We define the weighted Sobolev space $H^1_{\omega}(\mathbb{R}^n)$ (or $W^{1,2}_\omega(\mathbb{R}^n)$) as the space of functions $f \in L^2_\omega(\mathbb{R}^n)$ for which the distributional gradient $\nabla_x f$ belongs to $L^2_\omega(\mathbb{R}^n)^n$, and equip this space with the norm $\left\| f \right\|_{H^1_\omega} := ( \left\| f \right\|_{2,\omega}^2 + \left\| \nabla_x f \right\|_{2,\omega}^2 )^{1/2}$ making it a Hilbert space. It is also known that  $\mathcal{D}(\mathbb{R}^n)$ is dense in $H^1_{\omega}(\mathbb{R}^n)$ (see \cite[Thm. 2.5]{kilpelainen1994weighted}).

Let $I \subset \mathbb{R}$ be an open interval. Let $A: I \times \mathbb{R}^n \to M_n(\mathbb{C})$ be a matrix-valued function with complex measurable coefficients such that
$$
\left| A(t,x) \xi \cdot \zeta \right| \leq M \omega(x) \left| \xi \right| \left| \zeta \right|, \quad \nu \left| \xi \right|^2 \omega(x) \leq \mathrm{Re}(A(t,x) \xi \cdot \overline{\xi}),
$$
for some constants $M, \nu > 0$ and for all $\xi, \zeta \in \mathbb{C}^n$ and $(t,x) \in I \times \mathbb{R}^n$.

For each $t \in I$, we define the sesquilinear form $B_t : H^1_\omega(\mathbb{R}^n) \times H^1_\omega(\mathbb{R}^n) \to \mathbb{C}$ by
$$
B_t(u,v) := \int_{\mathbb{R}^n}  A(t,x) \nabla_x u(x) \cdot \overline{\nabla_x v(x)} \, \mathrm{d}x,
$$
for all $u, v \in H^1_\omega(\mathbb{R}^n)$. The assumptions on $A$ yield 
$$
|B_t(u,v)| \leq M \left\| \nabla_x u \right\|_{2,\omega} \left\| \nabla_x v \right\|_{2,\omega}, \quad \nu \left\| \nabla_x u \right\|_{2,\omega}^2 \leq \mathrm{Re}(B_t(u,u)).
$$
This is \eqref{EllipticityAbstract} with $T=\nabla_x : H^1_\omega(\mathbb{R}^n) \to L^2_\omega(\mathbb{R}^n)^n$.
We note that $T$ is injective since $d\omega$ has infinite mass as a doubling measure on $\mathbb{R}^n$. We denote by $\partial_t - \omega^{-1}(x) \mathrm{div}_x A(t,x) \nabla_x$ the degenerate parabolic operator associated with the family $(B_t)_{t \in I}$.
At this point, we can apply the theory developed above to obtain well-posedness results on $I \times \mathbb{R}^n$, for the Cauchy problems with test functions in $I\times \mathbb{R}^n$ using Theorem \ref{thm: passage au concret}, assuming weak solutions to \textit{a priori} be   in $L^1_{\mathrm{loc}}(I; L^2(\mathbb{R}^n))$ if $I$ is unbounded.

\begin{thm}
The operator $\partial_t - \omega^{-1}(x) \mathrm{div}_x A(t,x) \nabla_x$  on $I \times \mathbb{R}^n$ has a unique fundamental solution.
\end{thm}

\subsubsection*{\textbf{Copyright}}
A CC-BY 4.0 \url{https://creativecommons.org/licenses/by/4.0/} public copyright license has been applied by the authors to the present document and will be applied to all subsequent versions up to the Author Accepted Manuscript arising from this submission.

\bibliographystyle{alpha}
\bibliography{references}

\begin{thebibliography}{HvNVW16}

\bibitem[AE23]{auscher2023universal}
P.~Auscher and M.~Egert.
\newblock A universal variational framework for parabolic equations and
  systems.
\newblock {\em Calc. Var. Partial Differential Equations}, 62(9):Paper No. 249,
  59, 2023.

\bibitem[AEN20]{auscher20202}
P.~Auscher, M.~Egert, and K.~Nystr\"om.
\newblock {$\rm L^2$} well-posedness of boundary value problems for parabolic
  systems with measurable coefficients.
\newblock {\em J. Eur. Math. Soc. (JEMS)}, 22(9):2943--3058, 2020.

\bibitem[AMN97]{auscher1997holomorphic}
P.~Auscher, A.~McIntosh, and A.~Nahmod.
\newblock Holomorphic functional calculi of operators, quadratic estimates and
  interpolation.
\newblock {\em Indiana Univ. Math. J.}, 46(2):375--403, 1997.

\bibitem[AN25]{ataei2024fundamental}
A.~Ataei and K.~Nystr{\"o}m.
\newblock On fundamental solutions and {G}aussian bounds for degenerate
  parabolic equations with time-dependent coefficients.
\newblock {\em Potential Anal.}, 62(3):465--483, 2025.

\bibitem[Aro67]{aronson1967bounds}
D.~G. Aronson.
\newblock Bounds for the fundamental solution of a parabolic equation.
\newblock {\em Bull. Amer. Math. Soc.}, 73:890--896, 1967.

\bibitem[Aro68]{aronson1968non}
D.~G. Aronson.
\newblock Non-negative solutions of linear parabolic equations.
\newblock {\em Ann. Scuola Norm. Sup. Pisa Cl. Sci. (3)}, 22:607--694, 1968.

\bibitem[BJ12]{bogdan2012estimates}
K.~Bogdan and T.~Jakubowski.
\newblock Estimates of the {G}reen function for the fractional {L}aplacian
  perturbed by gradient.
\newblock {\em Potential Anal.}, 36(3):455--481, 2012.

\bibitem[CS07]{caffarelli2007extension}
L.~Caffarelli and L.~Silvestre.
\newblock An extension problem related to the fractional {L}aplacian.
\newblock {\em Comm. Partial Differential Equations}, 32(7-9):1245--1260, 2007.

\bibitem[Dav95]{davies1995spectral}
E.~B. Davies.
\newblock {\em Spectral theory and differential operators}, volume~42 of {\em
  Cambridge Studies in Advanced Mathematics}.
\newblock Cambridge University Press, Cambridge, 1995.

\bibitem[EHMT24]{egert2024harmonic}
M.~Egert, R.~Haller, S.~Monniaux, and P.~Tolksdorf.
\newblock Harmonic {A}nalysis {T}echniques for {E}lliptic {O}perators.
\newblock {\em Lecture notes. Available online
  \url{https://www.mathematik.tu-darmstadt.de/media/analysis/lehrmaterial_anapde/ISem_complete_lecture_notes.pdf}},
  2024.

\bibitem[Fri64]{friedman2008partial}
A.~Friedman.
\newblock {\em Partial differential equations of parabolic type}.
\newblock Prentice-Hall, Inc., Englewood Cliffs, NJ, 1964.

\bibitem[Haa06]{haase2006functional}
M.~Haase.
\newblock {\em The functional calculus for sectorial operators}, volume 169 of
  {\em Operator Theory: Advances and Applications}.
\newblock Birkh\"auser Verlag, Basel, 2006.

\bibitem[HvNVW16]{hytonen2016analysis}
T.~Hyt\"onen, J.~van Neerven, M.~Veraar, and L.~Weis.
\newblock {\em Analysis in {B}anach spaces. {V}ol. {I}. {M}artingales and
  {L}ittlewood-{P}aley theory}, volume~63 of {\em Ergebnisse der Mathematik und
  ihrer Grenzgebiete. 3. Folge. A Series of Modern Surveys in Mathematics
  [Results in Mathematics and Related Areas. 3rd Series. A Series of Modern
  Surveys in Mathematics]}.
\newblock Springer, Cham, 2016.

\bibitem[Kap66]{kaplan1966abstract}
S.~Kaplan.
\newblock Abstract boundary value problems for linear parabolic equations.
\newblock {\em Ann. Scuola Norm. Sup. Pisa Cl. Sci. (3)}, 20:395--419, 1966.

\bibitem[Kat61]{kato1961abstract}
T.~Kato.
\newblock Abstract evolution equations of parabolic type in {B}anach and
  {H}ilbert spaces.
\newblock {\em Nagoya Math. J.}, 19:93--125, 1961.

\bibitem[Kat95]{kato2013perturbation}
T.~Kato.
\newblock {\em Perturbation theory for linear operators}.
\newblock Classics in Mathematics. Springer-Verlag, Berlin, 1995.
\newblock Reprint of the 1980 edition.

\bibitem[Kil94]{kilpelainen1994weighted}
T.~Kilpel\"ainen.
\newblock Weighted {S}obolev spaces and capacity.
\newblock {\em Ann. Acad. Sci. Fenn. Ser. A I Math.}, 19(1):95--113, 1994.

\bibitem[KW04]{kunstmann2004maximal}
P.~C. Kunstmann and L.~Weis.
\newblock Maximal {$L_p$}-regularity for parabolic equations, {F}ourier
  multiplier theorems and {$H^\infty$}-functional calculus.
\newblock In {\em Functional analytic methods for evolution equations}, volume
  1855 of {\em Lecture Notes in Math.}, pages 65--311. Springer, Berlin, 2004.

\bibitem[KW23]{kassmann2023upper}
M.~Kassmann and M.~Weidner.
\newblock Upper heat kernel estimates for nonlocal operators via {A}ronson's
  method.
\newblock {\em Calc. Var. Partial Differential Equations}, 62(2):Paper No. 68,
  27, 2023.

\bibitem[Lio57]{lions1957problemes}
J.-L. Lions.
\newblock Sur les probl\`emes mixtes pour certains syst\`emes paraboliques dans
  des ouverts non cylindriques.
\newblock {\em Ann. Inst. Fourier (Grenoble)}, 7:143--182, 1957.

\bibitem[Lio61]{lions2013equations}
J.-L. Lions.
\newblock {\em \'Equations diff\'erentielles op\'erationnelles et probl\`emes
  aux limites}, volume Band 111 of {\em Die Grundlehren der mathematischen
  Wissenschaften}.
\newblock Springer-Verlag, Berlin-G\"ottingen-Heidelberg, 1961.

\bibitem[Lio69]{lions1969quelques}
J.-L. Lions.
\newblock {\em Quelques m\'ethodes de r\'esolution des probl\`emes aux limites
  non lin\'eaires}.
\newblock Dunod, Paris; Gauthier-Villars, Paris, 1969.

\bibitem[LSU68]{ladyzhenskaia1968linear}
O.~A. Ladyženskaja, V.~A. Solonnikov, and N.~N Ural’ceva.
\newblock {\em Linear and quasilinear equations of parabolic type}, volume~23
  of {\em Translations of Mathematical Monographs}.
\newblock American Mathematical Society, Providence, RI, 1968.
\newblock Translated from the Russian by S. Smith.

\bibitem[Nas58]{nash1958continuity}
J.~Nash.
\newblock Continuity of solutions of parabolic and elliptic equations.
\newblock {\em Amer. J. Math.}, 80:931--954, 1958.

\bibitem[Nys17]{nystrom2017l2}
K.~Nystr\"om.
\newblock {$L^2$} solvability of boundary value problems for divergence form
  parabolic equations with complex coefficients.
\newblock {\em J. Differential Equations}, 262(3):2808--2939, 2017.

\bibitem[RS80]{reed1980methods}
M.~Reed and B.~Simon.
\newblock {\em Methods of modern mathematical physics: Functional analysis}.
\newblock Academic Press, Inc. [Harcourt Brace Jovanovich, Publishers], New
  York, second edition, 1980.

\bibitem[Ste93]{Stein1993_HA}
E.M. Stein.
\newblock {\em Harmonic {A}nalysis: {R}eal-{V}ariable {M}ethods,
  {O}rthogonality, and {O}scillatory {I}ntegrals}, volume~43 of {\em Princeton
  Mathematical Series}.
\newblock Princeton University Press, Princeton, NJ, 1993.
\newblock With the assistance of Timothy S. Murphy, Monographs in Harmonic
  Analysis, III.

\bibitem[Wei01]{weis2001operator}
L.~Weis.
\newblock Operator-valued {F}ourier multiplier theorems and maximal
  {$L_p$}-regularity.
\newblock {\em Math. Ann.}, 319(4):735--758, 2001.

\bibitem[Zui02]{zuily2002elements}
C.~Zuily.
\newblock {\em {\'E}l{\'e}ments de distributions et d'{\'e}quations aux
  d{\'e}riv{\'e}es partielles: cours et probl{\`e}mes r{\'e}solus}, volume 130.
\newblock Dunod, 2002.

\end{thebibliography}

\end{document}